\UseRawInputEncoding
\documentclass[11pt]{amsart}
\usepackage[misc]{ifsym}
\usepackage{subfigure} 
\usepackage{float}
\usepackage{caption}

\usepackage{graphicx} 
\usepackage[ruled,linesnumbered]{algorithm2e} 
\usepackage{array} 
\usepackage{multirow} 
\usepackage{mathrsfs} 
\usepackage{amssymb,amsmath,mathbbol,amsfonts} 
\usepackage[dvips,left=2.8cm,right=2.8cm,top=3.3cm,bottom=3.3cm]{geometry} 
\usepackage{enumerate} 
\usepackage{arydshln} 

\newtheorem{lemma}{Lemma}[section]
\newtheorem{theorem}{Theorem}[section]
\newtheorem{remark}{Remark}[section]
\newtheorem{definition}{Definition}[section]
\newtheorem{example}{Example}[section]

\makeatletter
\renewcommand{\maketag@@@}[1]{\hbox{\m@th\normalsize\normalfont#1}}
\makeatother 
\begin{document}
	
	\title{Construction and analysis of the quadratic finite volume  methods on tetrahedral meshes}\footnote{This paper has been accepted for publication in SCIENCE CHINA Mathematics.}

\author {Peng Yang}
\address{School of Mathematics, Jilin University, Changchun 130012, China. }
\email{1443117432@qq.com}

\author {Xiang Wang}
\address{School of Mathematics, Jilin University, Changchun 130012, China. }
\email{wxjldx@jlu.edu.cn}

\author {Yonghai Li}
\address{School of Mathematics, Jilin University, Changchun 130012, China. }
\email{yonghai@jlu.edu.cn}

\thanks{This work is partially supported by the
	National Science Foundation of China (No.12071177), and the Science Challenge Project (No.
	TZ2016002).}

	\begin{abstract}
		A family of quadratic finite volume method (FVM) schemes are constructed and analyzed over tetrahedral meshes. In order to prove  stability and error estimate, we propose the minimum V-angle condition on tetrahedral meshes, and the surface and volume  orthogonal conditions on dual meshes.
		Through the element analysis technique, the local stability is equivalent to a positive definiteness of a $9\times9$ element matrix, which is  difficult to analyze directly or even numerically. 
		With the help of the surface  orthogonal condition and  congruent transformation,  this element matrix is reduced into a block diagonal matrix, then we carry out the stability result under the minimum V-angle condition.
		It is worth mentioning that the minimum V-angle condition of the tetrahedral case is very different from a simple extension of the minimum angle condition for triangular meshes, while it is also convenient to use in practice. Based on the stability, we prove the optimal $ H^{1} $ and $L^2$ error estimates respectively, 
		where the orthogonal conditions play an important role in ensuring optimal $L^2$ convergence rate.
		Numerical experiments are presented to illustrate our theoretical results.
		
		
	\end{abstract}
	
	\maketitle
	
	\section{Introduction}
	\label{intro}
	The finite volume method (FVM) \cite{Bank1987,Cai2003,Cai1991,ChenL2010,ChenZY2002,Eymard2000,HeWM,LiRH2000,Liebau1996,Schmidt1993,Süli1991} is one of the main numerical methods for solving partial differential equations, which is known for preserving the local conservation property. Till now, many progresses have been made on the stability, error estimate and superconvergence of the FVMs. The linear FVM schemes have been well studied on any spatial dimension  \cite{Carstensen2018,Ewing2002,Hackbusch1989,Hongqi,LiYH1999,LvJL2010,XuJC2009}, and complete results of arbitrary $k$-order FVM schemes on one dimension (1D) are given by \cite{Cao2013,WangX2021}. For high order FVMs on triangular meshes, \cite{ChenZY2012,ChenZY2015} present a unified analysis of the stability under the assumption of the minimum angle condition, which restricts the minimum interior angle of the triangular elements. While, for some quadratic FVM schemes, the angle value of this condition is improved by \cite{XuJC2009,ZhouYH2020,ZhouYH2020_2,ZouQS2017}, where \cite{ZhouYH2020,ZhouYH2020_2,ZouQS2017} are based on a new trial-to-test mapping. And, through proposing the orthogonal conditions, \cite{WangX2016} carries out optimal $L^{2}$ error estimates for arbitrary $k$-order FVM schemes on triangular meshes. For the FVMs on quadrilateral meshes, \cite{LvJL2012} proves the stability and optimal $L^2$ error estimate of the biquadratic FVM schemes, and \cite{LinYP2015,YangM2016,ZhangZM2014,ZhangZM2015} present the stability and optimal $L^2$ estimates of arbitrary high order FVM schemes by considering the bilinear form of the FVMs as the Gaussian quadrature of the bilinear form of the finite element methods (FEMs). The dual meshes of the schemes in \cite{LinYP2015,LvJL2012,YangM2016,ZhangZM2015} are based on the Gauss points. Compared with the big progresses in 1D and two dimension (2D),  the FVMs on three dimension (3D), which has more applications in practice, are mainly concentrated on the linear schemes \cite{Carstensen2018,Gao2008,LiJ2012,YangM2008}, and there are few results of high order FVMs on 3D (see, e.g., \cite{YangM2012}). Specially for the quadratic FVMs on tetrahedral meshes, no result has been published yet. 
	
	In this paper, we construct and analyze a family of quadratic FVM schemes on tetrahedral meshes for the following elliptic boundary value problem
	\begin{equation} \label{eq:epselliptic}
		\left\{ \begin{array}{rcl}
			-\nabla\cdot(\kappa\nabla u) &=&{ f \quad {\rm in}\ \Omega, }\\
			u &=&{ 0 \quad {\rm on}\ \Gamma , } \\
		\end{array} \right.
	\end{equation}
	where $\Omega \subset \mathbb{R}^{3}$ is a bounded convex polyhedron with boundary $\Gamma=\partial\Omega$,  $f \in L^{2}(\Omega)$ and the diffusion coefficient $\kappa(x_1,x_2,x_3)$ is a piecewise smooth function and bounded almost everywhere with
	positive lower and upper bounds: $ c_{*} $ and $ c^{*} $, respectively. 
	
	We introduce three parameters $ (\alpha,\beta,\gamma) $ to construct the dual mesh, and a mapping $ \Pi_{\lambda}^{*} $ from trial space to test space for theoretical analysis. For the stability and error analysis, we propose two key restrictions: the first is the orthogonal conditions on the surface and volume for the dual mesh, which control the construction of the dual elements; the second is the minimum V-angle condition for the primary tetrahedral mesh, which restricts the local shape around each vertex of the tetrahedral elements. Under the orthogonal condition on the surface and the minimum V-angle condition, we prove the local stability by element analysis and obtain optimal $ H^{1} $ error estimate. Based on an  equivalent discrete norm, the local stability is converted to a  positive definiteness of a $ 9\times 9 $ symbolic matrix. With the help of the orthogonal condition on the surface and the congruent transformation, the $ 9\times 9 $ symbolic matrix is reduced to a block diagonal matrix containing a  $ 3\times3$ matrix and a $6\times6 $ matrix, where the $ 3\times3 $ matrix is proved to be unconditionally positive definite, while the analysis of the $ 6\times6 $ matrix is a challenge. Fortunately, we derive that for fixed parameters $ (\alpha,\beta,\gamma,\lambda) $, the $ 6\times 6 $ symbolic matrix  only relies on five certain plane angles of a tetrahedral element. Then, under the minimum V-angle condition,  the positive definiteness of the $6\times6 $ matrix is guaranteed numerically, and the stability is obtained. On the other hand, under  the orthogonal conditions on the surface and volume, we prove optimal $ L^{2} $ error estimate by the Aubin-Nitsche technique.

	Let us summarize the contributions of this paper: 1) We first use three parameters $ (\alpha,\beta,\gamma) $ to  construct a family of quadratic FVM schemes on tetrahedral meshes; 2) Under the proposed minimum V-angle condition and the orthogonal condition on the surface, we prove the stability by element analysis; 3) Under the proposed orthogonal conditions on the surface and volume, we obtain optimal $ L^{2} $ error estimate.
	
	An outline of this paper is as follows. In Section~\ref{Section:2},  a family of quadratic FVM schemes on tetrahedral meshes is constructed. In Section~\ref{Section:3}, we prove the local stability by element analysis. Then, optimal $ H^{1} $ and $ L^{2} $ error estimates are given in Section~\ref{Section:4}. Numerical experiments are provided to confirm our theoretical results in Section~\ref{Section:5}. Finally,  we draw the conclusion in Section~\ref{Section:6}. Some symbolic matrices are put in Appendix A, and some relations in a tetrahedron and two proofs based on them are included in Appendix B.
	
	In the rest of this paper, $``A \lesssim B"$ means that $ A $ can be bounded by $ B $ mutiplied by a constant which is independent of the parameters that $ A $ and $ B $ may depend on. $``A \sim B"$ means both $``A \lesssim B"$ and $``B \lesssim A"$.
	\section{Preliminary}\label{Section:2}
	\subsection{The quadratic finite volume method schemes}
	\paragraph{Primary mesh and trial function space.}
	Let the primary mesh $\mathcal{T}_{h}=\left\lbrace K\right\rbrace$ be a conforming tetrahedral partition of $ \Omega $, where $ h=\max_{K \in \mathcal{T}_{h}} h_{K} $ and $ h_{K} $ is the length of the largest edge of the tetrahedral element $ K $. Assume that $\mathcal{T}_{h}$ is a regular partition, i.e.,  there exists a positive constant $ \sigma $, independent of $ h $, satisfying
	\begin{equation}\label{shape-regular}
		\frac{h_{K}}{\rho_{K}}\leq\sigma, \qquad \forall K \in \mathcal{T}_{h},
	\end{equation}
	where $ \rho_{K} $ is the diameter of the inscribed sphere of $ K $. 
	
	Denote the four vertices and six edge midpoints of $ K $ by $ \mathcal{N}_{K} $, and let $ \mathcal{N}_{h}=\cup_{K \in \mathcal{T}_{h} } \mathcal{N}_{K}$.  Then, define the trial function space over  $ \mathcal{T}_{h} $ as
	\begin{equation*}
		\mathit{U}_{h}=\left\lbrace u_{h}: u_{h}\in \mathit{C}(\overline{\Omega}),\,\, u_{h}\vert_{K} \in P^{2}(K)\,\, \forall K\in \mathcal{T}_{h}, \,\, u_{h}|_{\partial \Omega}=0 \right\rbrace ,
	\end{equation*}
	where $ P^{2}(K) $ is the quadratic polynomial space on $ K $, and $ u_{h}\vert_{K} $ is determined by its ten node values on $ \mathcal{N}_{K} $.
	
	\paragraph{Dual mesh and test function space.}
	\begin{figure}[ht!]
		\centering
		\begin{minipage}[t]{.9\textwidth}
			\includegraphics[width=220pt]{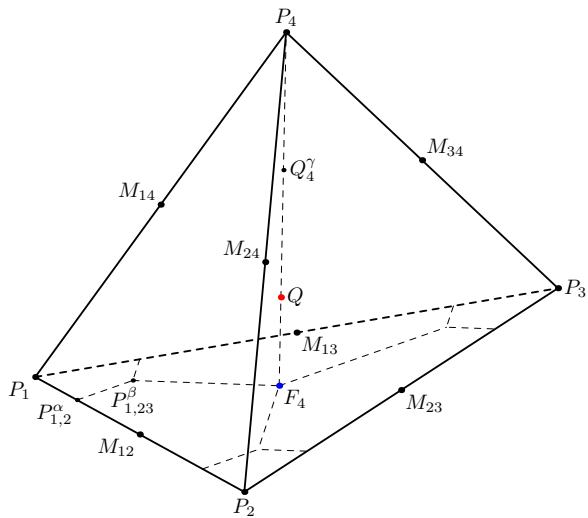}
			\caption{A  tetrahedral element $ K= \bigtriangleup^{4}P_{1}P_{2}P_{3}P_{4}$.}
			\label{fig:tetrahedron}
		\end{minipage}
	\end{figure}
	\begin{figure}[ht!]
		\centering
		\begin{minipage}[t]{.9\textwidth}
			\includegraphics[width=220pt]{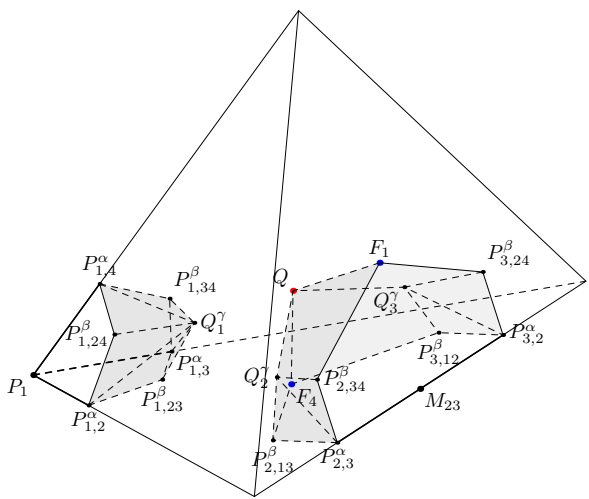}
			\caption{Two subdomains $ D^{K}_{P_{1}} $, $ D^{K}_{M_{23}} $ of the polyhedral partition.}
			\label{fig:dualareas}
		\end{minipage}
	\end{figure}   
	Define two sets
	\begin{equation*}
		\mathcal{Z}_{n}^{(1)}=\left\lbrace 1,2,...,n \right\rbrace,\qquad  \mathcal{Z}_{4}^{(2)}=\left\lbrace (j,k)\vert j,k\in\mathcal{Z}_{4}^{(1)},j<k \right\rbrace.
	\end{equation*}
	For a tetrahedral element $ K=\bigtriangleup^{4}P_{1}P_{2}P_{3}P_{4}$ (see Fig.~\ref{fig:tetrahedron}), $ T_{i} $ is the triangular face of $ K $ opposite to the vertex $ P_{i}$  $\big(i\in\mathcal{Z}_{4}^{(1)}\big) $, $ F_{i} $ is the barycenter of $ T_{i}$  $\big(i\in\mathcal{Z}_{4}^{(1)}\big) $,  and $ M_{jk} $ is the midpoint of the edge $ \overline{P_{j}P_{k}}$  $\big((j,k)\in\mathcal{Z}_{4}^{(2)}\big) $.
	It is well known that the four central lines $ \left\lbrace \overline{P_{i}F_{i}} \right\rbrace _{i\in\mathcal{Z}_{4}^{(1)}}$ intersect at $ Q $, the barycenter of $K$, and
	\begin{equation*}
		\begin{split}
			\frac{|\overline{P_{i_{1}}F_{i_{2}}}|}{|\overline{P_{i_{1}}M_{jk}}|}&=\frac{2}{3},\qquad  (j,k)\in\mathcal{Z}_{4}^{(2)}, \left\lbrace i_{1},i_{2},j,k\right\rbrace =\mathcal{Z}_{4}^{(1)}, \\
			\frac{|\overline{P_{i}Q}|}{|\overline{P_{i}F_{i}}|}&=\frac{3}{4},\qquad i\in\mathcal{Z}_{4}^{(1)}.
		\end{split}
	\end{equation*}
	
	We introduce three parameters $ \alpha $, $ \beta$, $ \gamma$ to locate the dual nodes in $ K$, such that
	\begin{equation} \label{parameters}
		\begin{split}
			\frac{|\overline{P_{i_1}P^{\alpha}_{i_{1},i_{2}}}|}{|\overline{P_{i_{1}}P_{i_{2}}}|}&=\alpha\in(0,1/2), \qquad i_{1},i_{2}\in\mathcal{Z}_{4}^{(1)},\,i_{1}\neq i_{2},\\ \frac{|\overline{P_{i}P^{\beta}_{i,jk}}|}{|\overline{P_{i}M_{jk}}|}&=\beta\in(0,2/3),\qquad i\in\mathcal{Z}_{4}^{(1)},(j,k)\in\mathcal{Z}_{4}^{(2)},i\notin \left\lbrace j,k \right\rbrace,\\
			\frac{|\overline{P_{i}Q^{\gamma}_{i}}|}{|\overline{P_{i}F_{i}}|}&=\gamma\in(0,3/4),\qquad i\in\mathcal{Z}_{4}^{(1)},
		\end{split}
	\end{equation}
	where $P^{\alpha}_{i_{1},i_{2}}$, $P^{\beta}_{i,jk}$, $Q^{\gamma}_{i}$ are the dual nodes on the edges $ \overline{P_{i_{1}}P_{i_{2}}}$, the midlines $\overline{P_{i}M_{jk}}$ 
	and the central lines  $\overline{P_{i}F_{i}}$, respectively. See $ P^{\alpha}_{1,2} $ on $ \overline{P_{1}P_{2}}$, $ P^{\beta}_{1,23} $ on $ \overline{P_{1}M_{23}}$, and $ Q^{\gamma}_{4} $ on $ \overline{P_{4}F_{4}}$ in Fig.~\ref{fig:tetrahedron}.
	
	Now, we show how to connect the above dual nodes in each element $ K $ to form the dual mesh. For every dual node $ P^{\beta}_{i,jk} $ on the triangular face $ T_{l}$ $( \left\lbrace i,j,k,l \right\rbrace=\mathcal{Z}_{4}^{(1)})  $, connect the following three line segments
	\begin{equation*}
		\overline{P^{\beta}_{i,jk}P^{\alpha}_{i,j}},\quad\overline{P^{\beta}_{i,jk}P^{\alpha}_{i,k}},\quad\overline{P^{\beta}_{i,jk}F_{l}}.
	\end{equation*}
	See connecting lines on $ T_{4} $ in Fig.~\ref{fig:tetrahedron}. For every  dual node $ Q^{\gamma}_{i} $ in the interior of $ K $, connect the following seven line segments
	\begin{equation*}
		\overline{Q^{\gamma}_{i}P^{\alpha}_{i,i_{1}}}\,(i_{1}\in\mathcal{Z}_{4}^{(1)}\setminus\{i\}),\quad
		\overline{Q^{\gamma}_{i}P^{\beta}_{i,jk}}\,\big((j,k)\in\mathcal{Z}_{4}^{(2)},i\notin \left\lbrace j,k \right\rbrace\big),\quad
		\overline{Q^{\gamma}_{i}F_{i}}.
	\end{equation*}
	Then, we obtain a polyhedral partition $ \left\lbrace D^{K}_{P}, P \in \mathcal{N}_{K}\right\rbrace $ of $ K $, where $ D^{K}_{P} $ is a subdomain of $ K $ surrounding  $ P $.  See $ D^{K}_{P_{1}} $ surrounding $ P_{1} $ and $ D^{K}_{M_{23}} $ surrounding $ M_{23} $ in Fig.~\ref{fig:dualareas}. Denote by $ K^{\ast}_{P}=\cup_{K \in \mathcal{T}_{h} }D^{K}_{P} $ the dual element associated with $ P \in \mathcal{N}_{h} $, and $ \mathcal{T}_{h}^{*}=\left\lbrace K^{\ast}_{P}, P \in \mathcal{N}_{h} \right\rbrace $ the dual mesh.
	
	The test function space over $ \mathcal{T}_{h}^{*} $ is defined as
	\begin{equation*}
		\mathit{V}_{h}=\{ v_h\in L^{2}(\Omega) : v_h\vert_{\mathit{K}_{P}^{*}}=\mathrm{constant},\, \forall\mathit{K}_{P}^{*}\in\mathcal{T}_{h}^{*}
		;\, v_h\vert_{\mathit{K}_{P}^{*}}=0, \forall P\in\partial \Omega \cap \mathcal{N}_{h} \}.
	\end{equation*}
	\paragraph{The  quadratic FVM schemes.} The quadratic FVM for (\ref{eq:epselliptic}) is to find $u_{h}\in U_{h}$, such that
	\begin{equation} \label{eq:FVM}
		a_{h}(u_{h},v_{h})=(f,v_{h})\qquad \forall v_{h}\in \mathit{V}_{h},
	\end{equation}
	where 
	\begin{equation}\label{def:bilinear form}
		a_{h}(u_{h},v_{h}) = -\sum\limits_{K^{*}\in\mathcal{T}_{h}^{*}}  {\iint_{\partial K^{*}} (\kappa\nabla u_{h})\cdot \textbf{n}\, v_{h} \,\mathrm{d} S},\qquad
		(f,v_{h}) = \sum\limits_{K^{*}\in\mathcal{T}_{h}^{*}} \iiint_{K^{*}} f  v_{h} \,\mathrm{d} x_1\mathrm{d} x_2\mathrm{d} x_3.
	\end{equation}
	Here $\textbf{n}$ is the unit outer normal vector of $\partial K^{*}$.
	\begin{remark}
		The dual mesh $ \mathcal{T}_{h}^{*} $ depends on the three parameters $ (\alpha,\beta,\gamma)$, so equation (\ref{eq:FVM}) actually leads to a family of quadratic FVM schemes.
	\end{remark}
	\subsection{Volume coordinates}
	We present the volume coordinates related with a tetrahedron. Let $ K=\bigtriangleup^{4}P_{1}P_{2}P_{3}P_{4}$ be a tetrahedron with vertices $ P_{i}=(x_{1,i},x_{2,i},x_{3,i})$ for $i\in\mathcal{Z}_{4}^{(1)}$ (see Fig.~\ref{fig:tetrahedron}). If these vertices are not coplanar, $ K$ has positive volume
	\begin{equation*}
		|K|=\frac{1}{6}\left| \begin{array}{cccc}
			1&x_{1,1}&x_{2,1}&x_{3,1}\\
			1&x_{1,2}&x_{2,2}&x_{3,2}\\
			1&x_{1,3}&x_{2,3}&x_{3,3}\\
			1&x_{1,4}&x_{2,4}&x_{3,4}\\
		\end{array}
		\right|.
	\end{equation*}
	Let $ K_i $ $(i\in\mathcal{Z}_{4}^{(1)})$ be the tetrahedrons subtended at $ P=(x_1,x_2,x_3) $ by the faces of $ K $, their volumes are $ |K_i| $ $(i\in\mathcal{Z}_{4}^{(1)})$, where $ |K_i| $ is obtained from $ |K| $ by replacing the elements $ 1,x_{1,i},x_{2,i},x_{3,i} $ by $ 1,x_{1},x_{2},x_{3} $.
	The volume coordinates $ L_{i}$ $ (i\in\mathcal{Z}_{4}^{(1)}) $ are defined by the volume-ratios
	\begin{equation}\label{def:volume coordinates}
		L_{i}=\dfrac{|K^{(i)}|}{|K|},\quad  i\in\mathcal{Z}_{4}^{(1)}.
	\end{equation}
	
	We have relations between $ (L_1,L_2,L_3,L_4) $ and $ (x_1,x_2,x_3) $ as follows
	\begin{equation}\label{eq:relation-coordinates}
		\left\{
		\begin{array}{lcr}
			x_1=x_{1,1}L_{1}+x_{1,2}L_{2}+x_{1,3}L_{3}+x_{1,4}L_{4},\\
			x_2=x_{2,1}L_{1}+x_{2,2}L_{2}+x_{2,3}L_{3}+x_{2,4}L_{4},\\
			x_3=x_{3,1}L_{1}+x_{3,2}L_{2}+x_{3,3}L_{3}+x_{3,4}L_{4},\\
			L_{1}+L_{2}+L_{3}+L_{4}=1,
		\end{array}
		\right.
	\end{equation}
	which transforms the reference element
	\begin{equation}\label{reference element}
		\hat{K}=\left\lbrace (L_{1},L_{2},L_{3})|L_{1}\geq 0,L_{2}\geq 0,L_{3}\geq 0,L_{1}+L_{2}+L_{3}\leq 1  \right\rbrace,
	\end{equation}
	into any tetrahedral element $ K=\bigtriangleup^{4}P_{1}P_{2}P_{3}P_{4}$.
	
	A direct calculation of (\ref{def:volume coordinates}) yields
	\begin{equation}\label{eq:nabda_L}
		\nabla L_{i}=(\dfrac{\partial L_{i}}{\partial x_{1}},\dfrac{\partial L_{i}}{\partial x_{2}},\dfrac{\partial L_{i}}{\partial x_{3}})^{T}=-\dfrac{|T_{i}|}{3|K|}\textbf{n}_{i},\quad  i\in\mathcal{Z}_{4}^{(1)},
	\end{equation}
	where  $\textbf{n}_{i},|T_{i}|$ are the unit outer normal vector and area of  $ T_{i} $ respectively. Let $ \theta_{jk} $ be the dihedral angle associated with the edge $ \overline{P_{j}P_{k}}$  $\big((j,k)\in\mathcal{Z}_{4}^{(2)}\big) $ in $ K $, and
	\begin{equation}\label{def:rij}
		\begin{array}{lcl}
			r_{jk}=|\overline{P_{j}P_{k}}|\cot\theta_{jk}\quad\forall (j,k)\in\mathcal{Z}_{4}^{(2)},\qquad
			R_{i}=\sum \limits_{(j,k)\in\mathcal{Z}_{4}^{(2)},i\notin\left\lbrace j,k\right\rbrace }r_{jk}\quad\forall i\in\mathcal{Z}_{4}^{(1)}.
		\end{array}
	\end{equation}
	Then we have Lemma~\ref{lemma:relation Li}.
	\begin{lemma}\label{lemma:relation Li}
		For the volume coordinates $ L_i $ $ (i\in\mathcal{Z}_{4}^{(1)}) $ given by (\ref{def:volume coordinates}), there holds
		\begin{equation*}
			\left\{
			\begin{array}{ll}
				6|K|(\nabla L_{j_{1}}\!\!\cdot\!\nabla L_{k_{1}})\!=\!-r_{j_{2}k_{2}},\qquad\quad   &(j_{1},k_{1}),(j_{2},k_{2})\in\mathcal{Z}_{4}^{(2)},\,\left\lbrace j_{1},k_{1},j_{2},k_{2}\right\rbrace=\mathcal{Z}_{4}^{(1)}, \\[0.35 cm]
				6|K|(\nabla L_{i}\!\cdot\!\nabla L_{i})=R_{i},\qquad &i\in\mathcal{Z}_{4}^{(1)}.
			\end{array} 
			\right.
		\end{equation*}
	\end{lemma}
	\begin{proof}
		By (\ref{eq:nabda_L}) and $ \textbf{n}_{j_{1}}\cdot\textbf{n}_{k_{1}}=-\cos\theta_{j_{2}k_{2}} $, we have
		\begin{equation*}
			6|K|(\nabla L_{j_{1}}\!\!\cdot\!\nabla L_{k_{1}})=-\dfrac{2|T_{j_{1}}||T_{k_{1}}|}{3|K|}\cos\theta_{j_{2}k_{2}}.
		\end{equation*}
		Then the fisrt relation follows from the volume formula that $ 3|K|=(2|T_{j_{1}}||T_{k_{1}}|\sin\theta_{j_{2}k_{2}})/|\overline{P_{j_{2}}P_{k_{2}}}| $.
		Combining the fact $ \nabla(L_{1}\!+\!L_{2}\!+\!L_{3}\!+\!L_{4})\!=\!0 $, we  have the second relation.
		\end{proof}
	\subsection{The mapping from  trial space to test  space}
	We define a transform operator from  trial space to test  space, which is  meaningful in the theoretical analysis of the quadratic FVM schemes, especially for stability analysis.
	\begin{definition}\label{definition:mapping}
		For $ \lambda\in \mathbb{R},\lambda\neq0 $, we define a  transform operator $ \Pi_{\lambda}^{*}$ from the trial space  $\mathit{U}_{h} $ to the test space $\mathit{V}_{h} $, such that for any $u_h\in\mathit{U}_{h} $, 
		\begin{equation*}
			\left\{
			\begin{array}{lcl}
				(\Pi^{\ast}_{\lambda}u_{h}) (P) =u_{h}(P), \\[0.2cm]
				(\Pi^{\ast}_{\lambda}u_{h}) (M )= \dfrac{1-\lambda}{2} \big(u_{h}(P_{M}^{1})+u_{h}(P_{M}^{2})\big) +\lambda u_{h}(M),
			\end{array} 
			\right.
		\end{equation*}
		where $ P\in\mathcal{N}_{h}  $ is the vertex, and $ M\in\mathcal{N}_{h} $ is the midpoint of the edge $ \overline{P_{M}^{1}P_{M}^{2}} $.
	\end{definition}
	\begin{remark}\label{remark:2}
		The mapping $ \Pi_{\lambda}^{*} $ is proposed only for the theoretical analysis of the quadratic FVM schemes, and it has no effect on pratical computations of these schemes.
		While taking $ \lambda=1 $,  $ \Pi_{1}^{*}$ is the traditional mapping $ \Pi_{h}^{*} $ (see \cite{LiRH2000}).
	\end{remark}
	
	For simplicity, we write the six edge midpoints in  $ K=\bigtriangleup^{4}P_{1}P_{2}P_{3}P_{4} $ (see Fig.~\ref{fig:tetrahedron}) as
	\begin{equation*}
		P_{5}=M_{23},\, P_{6}=M_{13},\, P_{7}=M_{12},\,
		P_{8}=M_{14},\, P_{9}=M_{24},\, P_{10}=M_{34}.
	\end{equation*}
	For each nodal point $ P_{i}$  $(i\in \mathcal{Z}_{10}^{(1)}) $, denote by $ \phi_{P_{i}}$ or $ \phi_{i} $ the corresponding local  quadratic Lagrange basis function, and $ \chi_{i} $ the corresponding local characteristic function of dual element $ \mathit{K}_{P_{i}}^{*}$. Let
	\begin{equation*}\label{def:Phi}
		\Phi=(\phi_{1},\phi_{2},...,\phi_{10})^T,\quad
		\varLambda=(\chi_{1},\chi_{2},...,\chi_{10})^T,
	\end{equation*}
	and
	\begin{equation}\label{S}
		\textbf{S}=
		{\scriptsize \begin{pmatrix}
				1&0&0&0&0&\frac{1-\lambda}{2}&\frac{1-\lambda}{2}&\frac{1-\lambda}{2}&0&0\\
				0&1&0&0&\frac{1-\lambda}{2}&0&\frac{1-\lambda}{2}&0&\frac{1-\lambda}{2}&0\\
				0&0&1&0&\frac{1-\lambda}{2}&\frac{1-\lambda}{2}&0&0&0&\frac{1-\lambda}{2}\\
				0&0&0&1&0&0&0&\frac{1-\lambda}{2}&\frac{1-\lambda}{2}&\frac{1-\lambda}{2}\\
				0&0&0&0&\lambda&0&0&0&0&0\\
				0&0&0&0&0&\lambda&0&0&0&0\\
				0&0&0&0&0&0&\lambda&0&0&0\\
				0&0&0&0&0&0&0&\lambda&0&0\\
				0&0&0&0&0&0&0&0&\lambda&0\\
				0&0&0&0&0&0&0&0&0&\lambda\\
		\end{pmatrix}}.
	\end{equation}
	According to Definition~\ref{definition:mapping}, Lemma~\ref{lemma:Pi} shows a relation between $ \Pi_{\lambda}^{*} $ and  $ \Pi_{1}^{*} $ associated with $ \textbf{S} $. 
	
	\begin{lemma} \label{lemma:Pi}
		For the mapping $ 	\Pi^{\ast}_{\lambda} $ restricted on any tetrahedral element $ K $, there holds
		\begin{equation}\label{eq:relation_mapping}
			\Pi^{\ast}_{\lambda}\Phi=\textbf{S} \varLambda=\textbf{S} \Pi^{\ast}_{1}\Phi.  
		\end{equation}
	\end{lemma}
	%
\begin{proof}
	We have the following forms of piecewise quadratic function $ u_{h} $ and piecewise constant function $ \Pi^{\ast}_{\lambda}u_{h} $ on $ K $
	\begin{equation*}
		u_{h}=\textbf{u}_{K}^{T}\Phi,\qquad \Pi^{\ast}_{\lambda}u_{h}=\widetilde{\textbf{u}}_{K}^{T}\varLambda,
	\end{equation*}
	where $ \textbf{u}_{K}=(u_{1},u_{2},...,u_{10})^T $ with $ u_{i}=u_{h}(P_{i}) $, and $ \widetilde{\textbf{u}}_{K}=(\tilde{u}_{1},\tilde{u}_{2},...,\tilde{u}_{10})^T  $ with $ \tilde{u}_{i}=(\Pi^{\ast}_{\lambda}u_{h})(P_{i}) $  for $i\in\mathcal{Z}_{10} $. By Definition~\ref{definition:mapping}, it is observed  that
	\begin{equation*}
		\widetilde{\textbf{u}}_{K}^{T}=\textbf{u}_{K}^{T}\textbf{S}.
	\end{equation*}
	Then, we have
	\begin{equation*}
		\textbf{u}_{K}^{T}\Pi^{\ast}_{\lambda}\Phi=\Pi^{\ast}_{\lambda}u_{h}=\textbf{u}_{K}^{T}\textbf{S}\varLambda
	\end{equation*}
	which indicates
	$ \Pi^{\ast}_{\lambda}\Phi=\textbf{S}\varLambda $.
	Taking $ \lambda=1 $, we get $ \Pi^{\ast}_{1}\Phi=\varLambda$, and this completes the proof. 
	 \end{proof}

\subsection{The orthogonal conditions}
The orthogonal conditions proposed for the FVM schemes on triangular meshes \cite{WangX2016} are used to prove optimal $ L^{2} $ error estimate. Here we propose the orthogonal conditions on the surface and  volume for the quadratic FVM schemes on tetrahedral meshes. The orthogonal condition on the surface is also  helpful to stability analysis.
\begin{definition}[Orthogonal conditions]
	A quadratic FVM scheme or the corresponding dual mesh $ \mathcal{T}^{*}_{h} $ is called to satisfy the orthogonal condition on the surface if the following equation associated with the mapping $ \Pi_{\lambda}^{*} $ holds
	\begin{equation} \label{orthogonal_plane}
		\iint_{T_{i}}g_{1}(v_{1}-\Pi_{\lambda}^{*}v_{1}) \,\mathrm{d} S =0\qquad  \forall g_{1},v_{1}\in P^{1}(T_{i}),\, T_{i}\in \partial K,\,\mathit{K}\in\mathcal{T}_{h},
	\end{equation}
	and it is called to satisfy the orthogonal condition on the volume if the following equation associated with $ \Pi_{\lambda}^{*} $ holds
	\begin{equation}\label{orthogonal_space}
		\iiint_{K}g_{2}(v_{2}-\Pi_{\lambda}^{*}v_{2}) \,\mathrm{d} x_1\mathrm{d} x_2\mathrm{d} x_3 =0\qquad  \forall g_{2},v_2\in P^{1}(K),\,\mathit{K}\in\mathcal{T}_{h}.
	\end{equation}
	Here $ P^{1} $ is the linear function space.
\end{definition}
\begin{lemma}\label{lemma:equivalent orthogonal condition}
	The orthogonal condition on the surface (\ref{orthogonal_plane}) is equivalent to  parameter equation
	\begin{equation} \label{orthogonal_plane_para} \alpha\beta(-\dfrac{1}{2}+\dfrac{1}{3}\alpha+\dfrac{1}{4}\beta)+\dfrac{1}{54}=0,
	\end{equation}
	and the orthogonal condition on the  volume (\ref{orthogonal_space}) is equivalent to  parameter equation
	\begin{equation} \label{orthogonal_space_para}
		\alpha\beta\gamma(-1+\dfrac{1}{2}\alpha+\dfrac{3}{8}\beta+\dfrac{1}{3}\gamma)+\dfrac{1}{480}=0.
	\end{equation}
\end{lemma}

\begin{proof}
Firstly, consider equation (\ref{orthogonal_plane}) on the reference element $\hat{K} $ (\ref{reference element}). It is equivalent to solve
\begin{equation}\label{equal:plane}
	\left\{
	\begin{split}
		&\iint_{\hat{T}_{4}}L_{1}^{n}(1-\Pi_{\lambda}^{*}1)\,\mathrm{d}L_{1}\mathrm{d}L_{2} =0, \\
		&\iint_{\hat{T}_{4}}L_{1}^{n}(L_{1}-\Pi_{\lambda}^{*}L_{1})\,\mathrm{d}L_{1}\mathrm{d}L_{2} =0,
	\end{split}
	\right.\qquad n=0,1.
\end{equation}
Since $ 1=\hat{\phi}_{1}+\hat{\phi}_{2}+\hat{\phi}_{3}+\hat{\phi}_{5}+\hat{\phi}_{6}+\hat{\phi}_{7} $ and $ L_{1}=\hat{\phi}_{1}+(\hat{\phi}_{6}+\hat{\phi}_{7})/2 $ holds on $ \hat{T}_{4} $, by Lemma~\ref{lemma:Pi}, it is easy to verify that $ \iint_{\hat{T}_{4}}(1-\Pi_{\lambda}^{*}1)\,\mathrm{d}L_{1}\mathrm{d}L_{2} =0 $ and $ \iint_{\hat{T}_{4}}(L_1-\Pi_{\lambda}^{*}L_1)\,\mathrm{d}L_{1}\mathrm{d}L_{2} =0 $.
In addition, we have the following integral results
\begin{equation}\label{Ints} 
	\left\{
	\begin{array}{ll}
		t_{_1}:=\iint_{\hat{T}_{4}}L_{1}\chi_{1}\,\mathrm{d}L_{1}\mathrm{d}L_{2}=\dfrac{\alpha\beta}{2}(1-\dfrac{\alpha+\beta}{3}),\\[2mm]
		t_{_2}:=\iint_{\hat{T}_{4}}L_{1}\chi_{2,3}\,\mathrm{d}L_{1}\mathrm{d}L_{2}=\dfrac{\alpha\beta(\alpha+\beta)}{12},\\[2mm]
		t_{_3}:=\iint_{\hat{T}_{4}}L_{1}\chi_{6,7}\,\mathrm{d}L_{1}\mathrm{d}L_{2}=\dfrac{2}{27}-\dfrac{\alpha\beta}{4}(1-\dfrac{\beta}{6}),\\[2mm]
		t_{_4}:=\iint_{\hat{T}_{4}}L_{1}\chi_{5}\,\mathrm{d}L_{1}\mathrm{d}L_{2}=\dfrac{1}{54}-\dfrac{\alpha\beta^2}{12},
	\end{array} 
	\right.
\end{equation}
where $ \chi_{i} $ is the local characteristic function of  dual element $ \hat{\mathit{K}}_{P_{i}}^{*}$, and $  \chi_{i,j}$ means $\chi_{i}  $ or $ \chi_{j} $.
Then, using above integral results yields   $ \iint_{\hat{T}_{4}}L_{1}(1-\Pi_{\lambda}^{*}1)\,\mathrm{d}L_{1}\mathrm{d}L_{2} =0 $, and
\begin{equation*}
	\iint_{\hat{T}_{4}}L_1(L_1-\Pi_{\lambda}^{*}L_1)\,\mathrm{d}L_{1}\mathrm{d}L_{2}=\iint_{\hat{T}_{4}}L^2_1\,\mathrm{d}L_{1}\mathrm{d}L_{2}-\iint_{\hat{T}_{4}}L_1(\chi_{1}+\dfrac{\chi_{6}}{2}+\dfrac{\chi_{7}}{2})\,\mathrm{d}L_{1}\mathrm{d}L_{2}=\frac{1}{12}-t_{_1}-t_{_3}.
\end{equation*}
Thus solving (\ref{equal:plane}) yields $ 1/12-t_{_1}-t_{_3}=0 $,
which leads to (\ref{orthogonal_plane_para}). Noticing  
\begin{equation*}
	t_{1}+2t_{2}+2t_{3}+t_{4}=\iint_{\hat{T}_{4}}L_1(\chi_{1}+\chi_{2}+\chi_{3}+\chi_{5}+\chi_{6}+\chi_{7})\,\mathrm{d}L_{1}\mathrm{d}L_{2}=\dfrac{1}{6},
\end{equation*}
parameter equation (\ref{orthogonal_plane_para}) can also be  written as
\begin{equation}\label{orthogonal_plane_para_t}
	-t_{_1}+2t_{_2}+t_{_4}=0.
\end{equation}

On the other hand, equation (\ref{orthogonal_space}) is equivalent to
\begin{equation*}
	\left\{
	\begin{split}
		&\iiint_{\hat{K}}L^{n}_{1}(1-\Pi_{\lambda}^{*}1)\,\mathrm{d}L_{1}\mathrm{d}L_{2}\mathrm{d}L_{3} =0, \\
		&\iiint_{\hat{K}}L^{n}_{1}(L_{1}-\Pi_{\lambda}^{*}L_{1})\,\mathrm{d}L_{1}\mathrm{d}L_{2}\mathrm{d}L_{3} =0,
	\end{split}
	\qquad n=0,1.
	\right.
\end{equation*}
Similarly, (\ref{orthogonal_space_para}) can be  derived. 
 \end{proof}

\begin{remark}
Note that $ \lambda $ doesn't appear in  parameter equations (\ref{orthogonal_plane_para}) and (\ref{orthogonal_space_para}). This means that the orthogonal conditions associated with $ \Pi_{\lambda}^{*} $ and $ \Pi_{1}^{*} $ lead to same  parameter equations. In fact, $ \Pi_{\lambda}^{*}L_{i}=\Pi_{1}^{*}L_{i}$  $(i\in{Z}_{4}^{(1)})$ for any given $\lambda\in \mathbb{R},\lambda\neq0 $.
\end{remark}
\begin{lemma}\label{solutions}
The parameter equations (\ref{orthogonal_plane_para}) and (\ref{orthogonal_space_para}) have infinite solutions. 
In addition, for any given $ \alpha $ in the range
\begin{equation}\label{range:alpha}
	\dfrac{1}{2}-\dfrac{\sqrt{6}}{6}\,(\approx 0.091752)<\alpha<\dfrac{1}{2},
\end{equation}
there is a unique solution for (\ref{orthogonal_plane_para}) and (\ref{orthogonal_space_para}) as follows
\begin{equation*}
	\left\{
	\begin{split}
		&\beta=(1-\dfrac{2}{3}\alpha)-\sqrt{(1-\dfrac{2}{3}\alpha)^{2}-\dfrac{2}{27\alpha}},\\[1mm]
		&\gamma=(\dfrac{3}{8}+\dfrac{1}{24\alpha\beta})-\sqrt{(\dfrac{3}{8}+\dfrac{1}{24\alpha\beta})^{2}-\dfrac{1}{160\alpha\beta}}. \\
	\end{split}
	\right.
\end{equation*}

\end{lemma}
\begin{proof}
We start by analyzing parameter equation (\ref{orthogonal_plane_para}) for $ \alpha\in(0,1/2)$ and $\beta\in(0,2/3) $ separately.
Considering (\ref{orthogonal_plane_para}) as an equation for $ \beta $, the two roots $ \beta_{1},\beta_{2}$  $(\beta_{1}\leq\beta_{2}) $ satisfy $ \beta_{1}+\beta_{2}=2-(4\alpha)/3>4/3 $. Thus, the reasonable root is
\begin{equation} \label{equation:beta}
	\beta=\beta_{1}(\alpha)=\dfrac{(\tfrac{1}{2}\alpha-\tfrac{1}{3}\alpha^{2})-\sqrt{(\tfrac{1}{2}\alpha-\tfrac{1}{3}\alpha^{2})^{2}-\tfrac{1}{54}\alpha}}{\tfrac{1}{2}\alpha}.
\end{equation}
Since $ 0<\beta_{1}(\alpha)<2/3 $, a direct calculation of this inequality about $ \alpha $ yields the range  (\ref{range:alpha}) of $ \alpha$.   Similarly, considering (\ref{orthogonal_plane_para}) as an equation for $ \alpha $, we have 
\begin{equation}\label{range:beta}
	\dfrac{2}{3}-\dfrac{2\sqrt{6}}{9}\,(\approx 0.122336)<\beta<\dfrac{2}{3}.
\end{equation}
Acctually, every reasonable solution of (\ref{orthogonal_plane_para}) exactly meets (\ref{range:alpha}) and (\ref{range:beta}). 
Moreover, equation (\ref{equation:beta}) indicates that $ \alpha\beta $ can be represented by  $ \alpha $ as follows
\begin{equation*}
	\alpha\beta=(\alpha-\dfrac{2}{3}\alpha^{2})-\sqrt{(\alpha-\dfrac{2}{3}\alpha^{2})^{2}-\dfrac{2}{27}\alpha}.
\end{equation*}
By (\ref{range:alpha}), we compute to arrive at the range of $ \alpha\beta $ that
\begin{equation}\label{range:ab}
	0.049913491374002\leq\alpha\beta<\dfrac{1}{3}-\dfrac{\sqrt{6}}{9}\,(\approx 0.061168).
\end{equation}

Then, considering parameter equation (\ref{orthogonal_space_para}) as an equation for $ \gamma\in(0,3/4) $, the two roots $ \gamma_{1},\gamma_{2}$  $(\gamma_{1}\leq\gamma_{2}) $ satisfy $ \gamma_{1}+\gamma_{2}=3-(3\alpha)/2-(9\beta)/8>3/2$. Using (\ref{orthogonal_plane_para}) to simplify (\ref{orthogonal_space_para}), we obtain
\begin{equation*}
	(\frac{1}{3}\alpha\beta)\gamma^{2}-(\frac{1}{4}\alpha\beta+\frac{1}{36})\gamma+\dfrac{1}{480}=0.
\end{equation*}
Thus, the reasonable root is
\begin{equation}\label{equation:gamma}
	\gamma=\gamma_{1}(\alpha\beta)=\dfrac{(\tfrac{1}{4}\alpha\beta+\tfrac{1}{36})-\sqrt{(\tfrac{1}{4}\alpha\beta+\tfrac{1}{36})^{2}-\tfrac{1}{360}\alpha\beta}}{\tfrac{2}{3}\alpha\beta}.
\end{equation}
Combining the range (\ref{range:ab}) of $ \alpha\beta $, we have
\begin{equation}\label{range:gamma}
	\dfrac{30+5\sqrt{6}-\sqrt{960+270\sqrt{6}}}{40}\,(\approx 0.049533)<\gamma\leq 0.052908895445995,
\end{equation}
which lies in $ (0,3/4) $.

Above discussions show that for any given $ \alpha $ in  (\ref{range:alpha}), parameters $ \beta $ (\ref{equation:beta}) and $ \gamma $ (\ref{equation:gamma}) are uniquely obtained from equations (\ref{orthogonal_plane_para}) and (\ref{orthogonal_space_para}).  Simplifying (\ref{equation:beta}) and (\ref{equation:gamma}) completes the proof. 
 \end{proof}
\begin{remark}
The quadratic FVM schemes satisfying the orthogonal conditions on tetrahedral meshes are infinite, this is different from the case of the unique one on  triangular meshes \cite{WangX2016}.
\end{remark}

\section{Stability analysis}\label{Section:3}
This section is devoted to analysis for stability of the quadratic FVM schemes  (\ref{eq:FVM}). Assume the  diffusion coefficient $ \kappa=1 $.
Our goal is to prove the  local stability that for any $ K \in  \mathcal{T}_{h} $,
\begin{equation}\label{eq:constant_element ellipitic}
a_{h}^{K}(u_{h},\Pi_{\lambda}^{*}u_{h}):=a_{h}(u_{h},\Pi_{\lambda}^{*}u_{h})\big|_{K}\gtrsim |u_{h}|_{1,K}^{2}   \qquad  \forall  u_{h}\in \mathit{U}_{h}.
\end{equation}

In Subsection~\ref{form and reduce}, the local stability (\ref{eq:constant_element ellipitic}) is converted to a positive definiteness of a $ 9\times9 $ symbolic matrix based on an equivalent discrete  norm. In Subsection~\ref{subsection:stability}, under the orthogonal condition on the surface (\ref{orthogonal_plane_para_t}), firstly, for any regular tetrahedron $ K \in  \mathcal{T}_{h} $, it is proved that the $ 9\times9 $ symbolic matrix is positive definite for given parameter $ \lambda $ in a certain range; secondly, for any general tetrahedron $ K \in  \mathcal{T}_{h} $, by the congruent transformation, the $ 9\times9 $  matrix is reduced to a block diagonal matrix  containing a $ 3\times3 $ matrix and  a $ 6\times6 $ matrix, where the $ 3\times3 $ matrix is proved to be  unconditionally positive definite. In Subsection~\ref{subsection:restriction3}, we derive that for fixed parameters $ (\alpha,\beta,\gamma,\lambda) $, the $ 6\times6 $ symbolic matrix only relies on five certain plane angles of a tetrahedral element. Then, the minimum V-angle condition (\ref{restriction_3}) is proposed to ensure the positive definiteness of the $ 6\times6 $ matrix numerically. Under the two restrictions (\ref{orthogonal_plane_para_t}) and (\ref{restriction_3}), the stability is presented in the end of this section.
\subsection{The element matrices}\label{form and reduce}
The $ 10\times10 $ element stiffness matrix and an equivalent discrete  norm  of  $ |u_{h}|_{1,K}^{2} $ is presented. Then, the local stability (\ref{eq:constant_element ellipitic}) is  converted to a positive definiteness of a $ 9\times9 $ symbolic matrix.

\vspace{2mm}
According to equation (\ref{def:bilinear form}), we have the  bilinear form on any element  $ K \in  \mathcal{T}_{h} $ that
\begin{equation*}
a_{h}^{K}(u_{h},\Pi_{\lambda}^{*}u_{h})=-\!\sum\limits_{K^{*}\in\mathcal{T}_{h}^{*}}  {\iint_{\partial K^{*}\cap K} \nabla u_{h}\cdot \textbf{n}\, \Pi_{\lambda}^{*}u_{h}\,\mathrm{d} S}.
\end{equation*}
By Lemma~\ref{lemma:Pi}, we have
\begin{equation*}
a_{h}^{K}(u_{h},\Pi_{\lambda}^{*}u_{h})=a_{h}^{K}(\textbf{u}^{T}_{K}\Phi,\Pi_{\lambda}^{*}(\textbf{u}^{T}_{K}\Phi))=\textbf{u}^{T}_{K}\mathbb{A}_{K,\lambda}\textbf{u}_{K},
\end{equation*}
where the element stiffness matrix $  \mathbb{A}_{K,\lambda}=\big(a_{mn,\lambda}\big) _{m,n\in\mathcal{Z}_{10}^{(1)}}$ with $ a_{mn,\lambda}=a_{h}^{K}(\phi_{n},\Pi_{\lambda}^{*}\phi_{m}) $, and
\begin{equation}\label{A_Kp}
\mathbb{A}_{K,\lambda}=\textbf{S}\mathbb{A}_{K,1},
\end{equation}
where the element matrix
$ \mathbb{A}_{K,1}=\left(  a_{mn}^{1} \right) _{m,n\in\mathcal{Z}_{10}^{(1)}}  $ with
$ a_{mn}^{1}\!=\!a_{h}^{K}(\phi_{n},\chi_{m})\!=\!-\iint_{ \partial K_{P_{m}}^{*}\cap  K} \nabla \phi_{n}\cdot \textbf{n}\,\mathrm{d} S $.

By the Green's formula, one gets
\begin{equation}\label{equation:a_hK}
-\!\iint_{ \partial K_{P_{m}}^{*}\cap  K}\!\!\!\nabla \phi_{n}\!\cdot \textbf{n}\,\mathrm{d} S=\!\iint_{  \partial K\cap K_{P_{m}}^{*}} \!\!\!\!\!\!\!\!\nabla \phi_{n}\!\cdot \textbf{n}\,\mathrm{d} S\!-\!\!\iiint_{K_{P_{m}}^{*}\!\!\cap K} \!\!\!\Delta\phi_{n}\,\mathrm{d}x_{1}\mathrm{d}x_{2}\mathrm{d}x_{3}\quad \forall m,n\in\mathcal{Z}_{10}^{(1)}.
\end{equation}
Since $ \Delta\phi_{n} $ $ (n\in\mathcal{Z}_{10}^{(1)}) $ are constants, we split the element matrix $ \mathbb{A}_{K,1} $  into two parts
\begin{equation}\label{A_k1}
\mathbb{A}_{K,1}=\textbf{A}-\,\textbf{v}_{1}\textbf{v}_{2}^T,
\end{equation}
where  $ \textbf{A} $ is a $ 10\times10 $ matrix, and $\textbf{v}_{1},\textbf{v}_{2}  $ are two column vectors. They are given by
\begin{gather*}
\textbf{A}=\left( \iint_{  \partial K\cap K_{P_{m}}^{*}} \nabla \phi_{n}\cdot \textbf{n}\,\mathrm{d} S \right)  _{m,n\in\mathcal{Z}_{10}^{(1)}},\\
\textbf{v}_{1}=
{\tiny
	\left(  \dfrac{1}{6|K|}\iiint_{K_{P_{m}}^{*}\cap K} 1\,\mathrm{d}x_{1}\mathrm{d}x_{2}\mathrm{d}x_{3} \right)  _{m\in\mathcal{Z}_{10}^{(1)}}},\quad
\textbf{v}_{2}=
\big(  6|K|\,\Delta\phi_{n} \big)  _{n\in\mathcal{Z}_{10}^{(1)}}.
\end{gather*}
By definition (\ref{parameters}) of parameters $ (\alpha,\beta,\gamma )$,  we have explicit form of $ \textbf{v}_{1} $ as follows
\begin{equation*}
\textbf{v}_{1}=\big(\tfrac{\alpha\beta\gamma}{6},\tfrac{\alpha\beta\gamma}{6},\tfrac{\alpha\beta\gamma}{6},\tfrac{\alpha\beta\gamma}{6},\tfrac{1-4\alpha\beta\gamma}{36}, \tfrac{1-4\alpha\beta\gamma}{36},\tfrac{1-4\alpha\beta\gamma}{36},\tfrac{1-4\alpha\beta\gamma}{36}, \tfrac{1-4\alpha\beta\gamma}{36},\tfrac{1-4\alpha\beta\gamma}{36}\big)^T.
\end{equation*}
By Lemma~\ref{lemma:relation Li} and the expressions of  ten local quadratic Lagrange  basis functions
\begin{equation}\label{basis functions}
\phi_{P_{i}}=L_{i}(2L_{i}-1)\quad \forall i\in\mathcal{Z}_{4}^{(1)},\qquad\phi_{M_{jk}}=4L_{j}L_{k}\quad \forall (j,k)\in\mathcal{Z}_{4}^{(2)},
\end{equation} 
we have explicit form of $ \textbf{v}_{2} $ as follows
\begin{equation*}
\textbf{v}_{2}=4(R_{_1},\,R_{_2},\,R_{_3},\,R_{_4},\,- 2r_{_{\!14}},- 2r_{_{\!24}}, - 2r_{_{\!34}}, - 2r_{23} , - 2r_{_{\!13}},- 2r_{_{\!12}} )^{T}.
\end{equation*}
From (\ref{basis functions}) and $ L_{1}+L_{2}+L_{3}+L_{4}=1 $, we have $ \nabla \phi_{n}=\sum_{i_1,i_2\in\mathcal{Z}_{4}^{(1)}}c_{i_1,i_2}L_{i_{1}}\nabla L_{i_{2}}$ with some constants $ c_{i_1,i_2} $. In view of this, Lemma~\ref{Lemma:integral results} presents some important equations for deriving $ \textbf{A} $.
\begin{lemma}\label{Lemma:integral results}
For the vertices $ P_{i} $ $ (i\in\mathcal{Z}_{4}^{(1)}) $ in element $ K $ and $ L_{i_{1}}\nabla L_{i_{2}} $ $ (i_{1},i_{2}\in\mathcal{Z}_{4}^{(1)}) $, we have the following integral results
\begin{equation*}
	\iint_{  \partial K \cap K_{P_{i}}^{*}}\!\! L_{i_{1}}\!\nabla L_{i_{2}}\!\cdot\!\textbf{n}\,\mathrm{d} S=
	\left\{
	\begin{array}{ll}
		6|K|\left( \,\nabla L_{i}\!\cdot\!\nabla L_{i_{2}}\,\right) t_{1},\qquad\quad &i_{1}=i,\\[0.3cm]
		6|K|\left( \sum\limits_{l\in\left\lbrace i,i_{1} \right\rbrace}\nabla L_{l}\!\cdot\!\nabla L_{i_{2}}\right) t_{2},\qquad\quad &i_{1}\neq i.
	\end{array} 
	\right.
\end{equation*}
For the midpoints $ M_{jk} $ $ \big((j,k)\in\mathcal{Z}_{4}^{(1)}\big) $ in element $ K $ and $ L_{i_{1}}\nabla L_{i_{2}} $ $ (i_{1},i_{2}\in\mathcal{Z}_{4}^{(1)}) $, we have the following integral results
\begin{equation*}
	\iint_{\partial K \cap  K_{M_{jk}}^{*}} \!\!L_{i_{1}}\!\nabla L_{i_{2}}\!\cdot\! \textbf{n}\,\mathrm{d} S=
	\left\{
	\begin{array}{lll}
		6|K|\left( \sum\limits_{l\in\left\lbrace j,k \right\rbrace}\nabla L_{l}\!\cdot\!\nabla L_{i_{2}}\right) t_{3},\qquad\quad &i_{1}\in\left\lbrace j,k \right\rbrace,\\[0.42cm]
		6|K|\left( \sum\limits_{l\in\left\lbrace j,k,i_{1} \right\rbrace}\nabla L_{l}\!\cdot\!\nabla L_{i_{2}}\right) t_{4},\qquad\quad &i_{1}\notin\left\lbrace j,k \right\rbrace.
	\end{array} 
	\right.
\end{equation*}
Here  $t_{i},i\in\mathcal{Z}_{4}^{(1)}$ are constants given by (\ref{Ints}).
\end{lemma}
\begin{proof}
Noticing $ \partial K\cap K_{P_{i_{3}}}^{*}=\cup_{i_4\in\mathcal{Z}_{4}^{(1)}}\big(T_{i_4}\cap K_{P_{i_{3}}}^{*}\big)$  $ (i_{3}\in\mathcal{Z}_{10}^{(1)}) $, we have
\begin{equation*}
	\iint_{\partial K\cap K_{P_{i_{3}}}^{*}}\!\!\!L_{i_{1}}\!\nabla L_{i_{2}}\!\cdot\! \textbf{n}\,\mathrm{d} S=\sum_{i_4\in\mathcal{Z}_{4}^{(1)}}\iint_{T_{i_{4}}}\!\!L_{i_{1}}\chi_{i_{3}}\nabla L_{i_{2}}\cdot \textbf{n}_{i_{4}}\,\mathrm{d} S \qquad \forall i_{3}\in\mathcal{Z}_{10}^{(1)},
\end{equation*}
where $ \chi_{i_3} $ is the local characteristic function of  dual element $ \mathit{K}_{P_{i_3}}^{*}$.
By (\ref{eq:nabda_L}), replacing vector $ \textbf{n}_{i_{4}} $ by $-(6|K|)/(2|T_{i_{4}}|)\nabla L_{i_{4}}$ yields
\begin{equation*}
	\begin{split}
		\sum_{i_4\in\mathcal{Z}_{4}^{(1)}}\iint_{T_{i_{4}}} L_{i_{1}}\chi_{i_{3}}\nabla L_{i_{2}}\cdot \textbf{n}_{i_{4}}\,\mathrm{d} S&=-6|K|\sum_{i_4\in\mathcal{Z}_{4}^{(1)}}\iint_{T_{i_{4}}} \dfrac{1}{2|T_{i_{4}}|}\,L_{i_{1}}\chi_{i_{3}}\, \nabla L_{i_{4}}\cdot\nabla L_{i_{2}}\,\mathrm{d} S\\
		&=-6|K|(\nabla L_{i_{4}}\cdot\nabla L_{i_{2}})\sum_{i_4\in\mathcal{Z}_{4}^{(1)}}\iint_{\hat{T}_{i_{4}}} L_{i_{1}}\chi_{i_{3}}\,\mathrm{d} \hat{S}.
	\end{split}
\end{equation*}
Combining the facts that $ L_{i}=0 $ holds on $ T_{i} $ $ (i\in\mathcal{Z}_{4}^{(1)}) $, $ \nabla(L_{1}+L_{2} +L_{3} +L_{4}) =0 $, and the integral results (\ref{Ints}), one can reach the conclusion of Lemma~\ref{Lemma:integral results}.
 \end{proof}
By Lemma~\ref{Lemma:integral results}, taking $ \textbf{A}(1,1) $ as an example, we have
\begin{align*}
\textbf{A}(1,1)&=\iint_{  \partial K\cap K_{P_{1}}^{*}}\!\!\!\! \nabla \phi_{P_{1}}\cdot \textbf{n}\,\mathrm{d} S=\iint_{  \partial K\cap K_{P_{1}}^{*}} \!\!\!\!3L_{1}\nabla L_{1}\cdot \textbf{n}\,\mathrm{d} S-\iint_{  \partial K\cap K_{P_{1}}^{*}} \!\!\!\!(L_{2}+L_{3}+L_{4})\nabla L_{1}\cdot \textbf{n}\,\mathrm{d} S\\
&=6|K|(\nabla L_{1}\!\!\cdot\!\nabla L_{1})3t_{1}-6|K|(3\nabla L_{1}+\nabla L_{2}+\nabla L_{3}+\nabla L_{4})\!\cdot\!\nabla L_{1}t_{2}=(3t_1-2t_2)R_{1}.
\end{align*}
Other entries are derived similarly,  and the explicit form of $ \textbf{A} $ is put in Appendix A.1.

\vspace{2mm}
Then, consider $ |u_{h}|_{1,K}^{2} $ with $ u_{h}\vert_{K}=\sum_{i\in\mathcal{Z}_{10}^{(1)}}u_{i}\phi_{P_{i}}$ and $ u_{i}=u_{h}(P_{i}) $ $(i\in\mathcal{Z}_{10}) $. By (\ref{basis functions}) and  $ L_{4}=1\!-\!L_{1}\!-\!L_{2}\!-\!L_{3} $, we have
\begin{align}
u_{h}\vert_{K}=&u_{4}+L_{1}(-u_{1}-3u_{4}+4u_{8})+L_{2}(-u_{2}-3u_{4}+4u_{9})+L_{3}(-u_{3}-3u_{4}+4u_{10})\nonumber\\
&+2L^{2}_{1}(u_{1}+u_{4}-2u_{8})+2L^{2}_{2}(u_{2}+u_{4}-2u_{9})+2L^{2}_{3}(u_{3}+u_{4}-2u_{10})\nonumber\\
&+4L_{2}L_{3}(u_{4}+u_{5}-u_{9}-u_{10})+4L_{1}L_{3}(u_{4}+u_{6}-u_{8}-u_{10})+4L_{1}L_{2}(u_{4}+u_{7}-u_{8}-u_{9})\nonumber\\
=&u_{4}+\big(L_{1},L_{2},L_{3},L^{2}_{1},L^{2}_{2},L^{2}_{3},L_{2}L_{3},L_{1}L_{3},L_{1}L_{2}\big)\textbf{G}\textbf{u}_{K},\label{u_K}
\end{align}
where $ \textbf{u}_{K}=\big(u_{1},u_{2},...,u_{10}\big)^T $, and $ \textbf{G} $ is a $ 9\times 10 $ matrix given by
\begin{equation}\label{G}
\textbf{G}=
{\scriptsize
	\left( 
	\begin{array}{rrrrrrrrrr}
		-1&0&0&-3&0&0&0&4&0&0\\
		0&-1&0&-3&0&0&0&0&4&0\\
		0&0&-1&-3&0&0&0&0&0&4\\
		2&0&0&2&0&0&0&-4&0&0\\
		0&2&0&2&0&0&0&0&-4&0\\
		0&0&2&2&0&0&0&0&0&-4\\
		0&0&0&4&4&0&0&0&-4&-4\\
		0&0&0&4&0&4&0&-4&0&-4\\
		0&0&0&4&0&0&4&-4&-4&0\\
	\end{array}
	\right) 
}
\end{equation}
Note that the constant term of $ u_{h}\vert_{K} $ vanishes in the derivative of $ u_{h}\vert_{K} $. In view of this, Lemma~\ref{Lemma:Descrete equivalent norm} presents an  equivalent discrete norm of $ |u_{h}|^{2}_{1,K} $ associated with $ \textbf{G} $.
\begin{lemma}[Discrete norm]\label{Lemma:Descrete equivalent norm}
If $ \mathcal{T}_{h} $ is a regular partition (\ref{shape-regular}), 
then for each $ K \in  \mathcal{T}_{h}$, we have
\begin{equation}\label{equi_norm}
	|u_{h}|_{1,K}^{2}\sim h_{K} \|\textbf{G}\textbf{u}_{K}\|^{2}, 
\end{equation}
where $ \|\cdot\| $ is the Euclidean norm, and $ \textbf{G} $ is given by (\ref{G}). 
\end{lemma}
\begin{proof}
Under the translation (\ref{eq:relation-coordinates}), let $ \hat{u}_{h}\vert_{\hat{K}} $ be the interpolation function on reference element $ \hat{K} $ (\ref{reference element}).
Since $ \mathcal{T}_{h} $ is a regular partition, the following relation for Sobolev semi-norms holds \cite{LiRH2000}:
\begin{equation}\label{equi:uh^_uh}
	|u_{h}|_{1,K}^{2}\sim h_{K} |\hat{u}_{h}|_{1, \hat{K}}^{2}.
\end{equation}
By (\ref{u_K}), a direct calculation of $ |\hat{u}_{h}|_{1, \hat{K}}^{2} $ yields
\begin{align*}
	|\hat{u}_{h}|_{1,\hat{K}}^{2}&=\iiint_{\hat{K}}\sum_{i\in\mathcal{Z}_{3}^{(1)}}(\dfrac{\partial\hat{u}_{h}}{\partial L_{i}})^{2}\,dL_{1}dL_{2}dL_{3}\\
	&=\iiint_{\hat{K}}\left( (\textbf{w}_{1}\textbf{G}\textbf{u}_{K})^{T}(\textbf{w}_{1}\textbf{G}\textbf{u}_{K})+(\textbf{w}_{2}\textbf{G}\textbf{u}_{K})^{T}(\textbf{w}_{2}\textbf{G}\textbf{u}_{K})+(\textbf{w}_{3}\textbf{G}\textbf{u}_{K})^{T}(\textbf{w}_{3}\textbf{G}\textbf{u}_{K})
	\right) dL_{1}dL_{2}dL_{3}\\
	&=(\textbf{G}\textbf{u}_{K})^{T}\textbf{W}(\textbf{G}\textbf{u}_{K}),
\end{align*}
where 
\begin{align*}
	\textbf{w}_{1}=(1,0,0,2L_{1},0,0,0,L_{3},L_{2}), \textbf{w}_{2}=(0,1,0,0,2L_{2},0,L_{3},0,L_{1}), \textbf{w}_{3}=(0,0,1,0,0,2L_{3},L_{2},L_{1},0),
\end{align*}
and
\begin{equation*}
	\textbf{W}=\iiint_{\hat{K}}\big(\textbf{w}_{1}^{T}\textbf{w}_{1}+\textbf{w}_{2}^{T}\textbf{w}_{2}+\textbf{w}_{3}^{T}\textbf{w}_{3}\big)\,dL_{1}dL_{2}dL_{3}={\scriptsize \tfrac{1}{120}\begin{pmatrix}
			20&0&0&10&0&0&0&5&5\\
			0&20&0&0&10&0&5&0&5\\
			0&0&20&0&0&10&5&5&0\\
			10&0&0&8&0&0&0&2&2\\
			0&10&0&0&8&0&2&0&2\\
			0&0&10&0&0&8&2&2&0\\
			0&5&5&0&2&2&4&1&1\\
			5&0&5&2&0&2&1&4&1\\
			5&5&0&2&2&0&1&1&4\\
	\end{pmatrix}}.
\end{equation*}

It is verified that the real symmetric matrix $ \textbf{W} $ is positive definite. Therefore,   $|\hat{u}_{h}|_{1,\hat{K}}^{2}\!\sim\!\|\textbf{G}\textbf{u}_{K}\|^{2} $, which together with (\ref{equi:uh^_uh}) completes the proof. 
 \end{proof}

Let $ \textbf{T} $ be the unique Moore-Penrose inverse of $ \textbf{G} $, such that $ \textbf{T}\textbf{G}\textbf{T}\!=\!\textbf{T},\textbf{G}\textbf{T}\textbf{G}\!=\!\textbf{G},(\textbf{T}\textbf{G})^{T}\!=\!\textbf{T}\textbf{G},(\textbf{G}\textbf{T})^{T}\!=\!\textbf{G}\textbf{T} $. By computing in \textit{Matlab}, we get the $ 10\times9 $ matrix
\begin{equation}\label{T}
\textbf{T}=
{\scriptsize
	\tfrac{1}{40}\left( 
	\begin{array}{rrrrrrrrr}
		30&-10&-10&33&-7&-7&-1&-1&-1\\
		-10&30&-10&-7&33&-7&-1&-1&-1\\
		-10&-10&30&-7&-7&33&-1&-1&-1\\
		-10&-10&-10&-7&-7&-7&-1&-1&-1\\
		-10&10&10&-7&3&3&9&-1&-1\\
		10&-10&10&3&-7&3&-1&9&-1\\
		10&10&-10&3&3&-7&-1&-1&9\\
		10&-10&-10&3&-7&-7&-1&-1&-1\\
		-10&10&-10&-7&3&-7&-1&-1&-1\\
		-10&-10&10&-7&-7&3&-1&-1&-1\\
	\end{array}
	\right),
}
\end{equation}	
and
\begin{equation}\label{TG}
\textbf{T}\textbf{G}=\mathbb{E}_{_{10}}-\dfrac{1}{10}\mathbb{1},
\end{equation}
where $\mathbb{E}_{_{n}}  $ is the $ n\times n $ identity matrix, and $ \mathbb{1} $  is the $10\times 10  $ matrix in which all entries are 1. The following Lemma~\ref{Lemma:99matix} helps us to simplify the local stability (\ref{eq:constant_element ellipitic}).
\begin{lemma} \label{Lemma:99matix}
For the element stiffness matrix $\mathbb{A}_{K,\lambda}  $ in (\ref{A_Kp}), we have
\begin{equation*}
	\mathbb{A}_{K,\lambda}=\textbf{G}^{T}\textbf{T}^{T}\mathbb{A}_{K,\lambda}\textbf{T}\textbf{G},
\end{equation*}
where $ \textbf{G} $ and $ \textbf{T} $ are given by (\ref{G}) and (\ref{T}), respectively.
\end{lemma}
\begin{proof}
Obviously, $ \mathbb{1}\textbf{S}=\mathbb{1} $ and  $ \mathbb{A}_{K,1}\mathbb{1}=\mathbb{1}\mathbb{A}_{K,1}=\mathbb{0}_{_{10}} $, where $ \mathbb{0}_{_{n}} $  is the $n\times n  $ zero matrix. Then, we derive  $ \mathbb{A}_{K,\lambda}\mathbb{1}=\mathbb{1}\mathbb{A}_{K,\lambda}=\mathbb{0}_{_{10}} $ from (\ref{A_Kp}).
By (\ref{TG}), we obtain
\begin{equation*}
	(\textbf{T}\textbf{G})^{T}\mathbb{A}_{K,\lambda}\textbf{T}\textbf{G}=(\mathbb{E}_{_{10}}-\dfrac{1}{10}\mathbb{1})^{T}\mathbb{A}_{K,\lambda}(\mathbb{E}_{_{10}}-\dfrac{1}{10}\mathbb{1})=\mathbb{A}_{K,\lambda}(\mathbb{E}_{_{10}}-\dfrac{1}{10}\mathbb{1})=\mathbb{A}_{K,\lambda}.
\end{equation*}
This  completes the  proof. 
 \end{proof}

Let $ \mathbb{B}_{K,\lambda} $ be a $ 9\times9 $  symbolic matrix, given by
\begin{equation}\label{B_Kp}
\mathbb{B}_{K,\lambda}=\textbf{T}^{T}\mathbb{A}_{K,\lambda}\textbf{T} . 
\end{equation}
By Lemma~\ref{Lemma:99matix}, we have
\begin{equation*}
a_{h}^{K}(u_{h},\Pi_{\lambda}^{*}u_{h})=\dfrac{1}{2}\textbf{u}_{K}^{T}(\mathbb{A}_{K,\lambda}+\mathbb{A}_{K,\lambda}^T)\textbf{u}_{K}=\dfrac{1}{2}\textbf{u}_{K}^{T}(\textbf{G}^{T}\mathbb{B}_{K,\lambda}\textbf{G}+\textbf{G}^{T}\mathbb{B}_{K,\lambda}^{T}\textbf{G})\textbf{u}_{K}=(\textbf{G}\textbf{u}_{K})^{T}\overline{\mathbb{B}}_{K,\lambda}(\textbf{G}\textbf{u}_{K}),
\end{equation*} 
where $ \overline{\mathbb{B}}_{K,\lambda} $ is the symmetrization of $ \mathbb{B}_{K,\lambda} $, i.e.,
\begin{equation}\label{B}
\overline{\mathbb{B}}_{K,\lambda}=\dfrac{\mathbb{B}_{K,\lambda}+\mathbb{B}_{K,\lambda}^{T}}{2}.
\end{equation}
Recalling the equivalent discrete norm (\ref{equi_norm}), the local stability (\ref{eq:constant_element ellipitic}) is equivalent to 
\begin{equation*}
(\textbf{G}\textbf{u}_{K})^{T}\overline{\mathbb{B}}_{K,\lambda}(\textbf{G}\textbf{u}_{K}) \gtrsim h_{K}(\textbf{G}\textbf{u}_{K})^{T}(\textbf{G}\textbf{u}_{K}).
\end{equation*}
Therefore,  the local stability (\ref{eq:constant_element ellipitic}) is converted into a positive definiteness of the matrix  $ \frac{1}{h_{K}}\overline{\mathbb{B}}_{K,\lambda} $,  which is proved in the following subsections.
\subsection{A restriction on $ \mathcal{T}_{h}^{*} $ and some results}\label{subsection:stability}
In this subsection, under the orthogonal condition on the surface (\ref{orthogonal_plane_para_t}), we prove a positive definiteness of $ \frac{1}{h_{K}}\overline{\mathbb{B}}_{K,\lambda} $ for any regular tetrahedron $ K \in  \mathcal{T}_{h} $, and reduce $ \frac{1}{h_{K}}\overline{\mathbb{B}}_{K,\lambda} $ for any general tetrahedron $ K \in  \mathcal{T}_{h} $.

\vspace{2mm}
Firstly, we try to deal with $\mathbb{B}_{K,\lambda}$ in (\ref{B_Kp}). By (\ref{A_Kp}) and (\ref{A_k1}), one gets
\begin{equation}\label{B_0} \mathbb{B}_{K,\lambda}=\textbf{T}^{T}\mathbb{A}_{K,\lambda}\textbf{T}=\textbf{T}^{T}\textbf{S}\mathbb{A}_{K,1}\textbf{T}=\textbf{T}^T\textbf{S}\textbf{A}\textbf{T}-(\textbf{T}^T\textbf{S}\textbf{v}_{1})\big(\textbf{v}_{2}^T\textbf{T}\big).
\end{equation}
By computing in  \textit{Matlab}, we have
\begin{equation}\label{part1}
\begin{split}
	&\textbf{T}^T\textbf{S}\textbf{v}_{1}=s_{_0}\big( 0,0,0,3,3,3,-1,-1,-1\big) ^T,\\
	&\textbf{v}_{2}^T\textbf{T}=(0,0,0,2R_{_1},2R_{_2},2R_{_3},-2r_{_{\!14}},-2r_{_{\!24}},-2r_{_{\!34}}),
\end{split}
\end{equation}
and organize $ \textbf{T}^T\textbf{S}\textbf{A}\textbf{T} $ into a $ 3\times3 $ block matrix as follows
\begin{flalign}\label{part2}
\textbf{T}^T\textbf{S}\textbf{A}\textbf{T}\!=\!\!
{\small \left(\begin{array}{c;{3pt/1pt}cc} 
		s_{_1}\textbf{M}^{K}_{_{3\!\times\!3}} &\dfrac{s_{_1}}{2}\textbf{M}^{K}_{_{3\!\times\!3}}\!\!+\!\!\dfrac{s^{*}}{2}\textbf{Q}^{(1)}_{_{3\!\times\!3}}&\quad\dfrac{s_{_1}}{4}(\textbf{M}^{K}_{_{3\!\times\!3}}\textbf{C}_{_{3\!\times\!3}})\!+\!\!\dfrac{s^{*}}{4}\textbf{Q}^{(2)}_{_{3\!\times\!3}} \\[2mm]
		\hdashline[3pt/1pt]
		\vspace{1ex}
		(s_{_1}\!-\!\dfrac{2s_{_2}+s_{_3}}{2})\textbf{M}^{K}_{_{3\!\times\!3}} & \textbf{L}^{(1)}_{_{3\!\times\!3}}&\textbf{L}^{(2)}_{_{3\!\times\!3}}\\
		\dfrac{2s_{_2}+s_{_3}}{4}(\textbf{C}_{_{3\!\times\!3}}\textbf{M}^{K}_{_{3\!\times\!3}})& \textbf{L}^{(3)}_{_{3\!\times\!3}}&\textbf{L}^{(4)}_{_{3\!\times\!3}}
	\end{array}
	\right)}.
\end{flalign}
Here
\begin{gather} 
s_{_0}=\dfrac{1}{240}+\dfrac{4\alpha\beta\gamma-1}{144}\lambda,\quad
s_{_1}=t_{_1}+2t_{_2}+2t_{_3}+t_{_4}=\dfrac{1}{6},\quad s_{_2}=t_{_3}\lambda,\quad s_{_3}=t_{_4}\lambda,\label{s_{_0}-3}\\	
s^{*}=-t_{_1}+2t_{_2}+t_{_4}\mbox{ and $  s^{*}\!=\!0 $ coincides with the orthogonal condition on the surface (\ref{orthogonal_plane_para_t})},\nonumber\\
\textbf{M}^{K}_{_{3\!\times\!3}}\!=\! \left( \begin{array}{rrr}
R_{_1}&-r_{_{\!34}}&-r_{_{\!24}}\\
-r_{_{\!34}}&R_{_2}&-r_{_{\!14}}\\
-r_{_{\!24}}&-r_{_{\!14}}&R_{_3}\\
\end{array}
\right),\qquad
\textbf{C}_{_{3\!\times\!3}}\!=\! \begin{pmatrix}
0&1&1\\
1&0&1\\
1&1&0\\
\end{pmatrix}.\label{M,C}
\end{gather}
We put those symbolic matrices $\textbf{Q}^{(1)}_{_{3\!\times\!3}}  $, $\textbf{Q}^{(2)}_{_{3\!\times\!3}}  $, $ \textbf{L}^{(1)}_{_{3\!\times\!3}} $, $ \textbf{L}^{(2)}_{_{3\!\times\!3}} $, $ \textbf{L}^{(3)}_{_{3\!\times\!3}} $, $ \textbf{L}^{(4)}_{_{3\!\times\!3}} $ in Appendix A.2.	
Lemma~\ref{determine_relation} shows a property of $ \textbf{M}^{K}_{_{3\!\times\!3}} $, and the proof is included in Appendix B.2.
\begin{lemma}\label{determine_relation}
If $\mathcal{T}_{h}$ is a regular partition, then $ \frac{1}{h_{K}}\textbf{M}^{K}_{_{3\!\times\!3}}  $ is unconditionally positive definite for any tetrahedral element $ K\in \mathcal{T}_{h}$.
\end{lemma}

\vspace{2mm}
Now, we consider $ \frac{1}{h_{K}}\overline{\mathbb{B}}_{K,\lambda} $ for any regular tetrahedron $ K\in \mathcal{T}_{h}$. The following result shows  that   $\frac{1}{h_{K}}\overline{\mathbb{B}}_{K,\lambda}$ is  positive definite for given $ \lambda $ in a certain range.
\begin{lemma}\label{Lemma:regular tetra}
Assume that  $ K$ is a regular tetrahedron. If the orthogonal condition on the surface (\ref{orthogonal_plane_para_t}) (i.e., $ s^*=0 $) holds, then the element matrix
$ \frac{1}{h_{K}}\overline{\mathbb{B}}_{K,\lambda} $ in (\ref{B}) is positive definite if and only if $ \lambda $  satisfies
\begin{flalign} \label{p}
\dfrac{(2t_{_3}+5t_{_4})-2\sqrt{2t_{_4}(2t_{_3}+3t_{_4})}}{6(2t_{_3}+t_{_4})^2}<\lambda<\dfrac{(2t_{_3}+5t_{_4})+2\sqrt{2t_{_4}(2t_{_3}+3t_{_4})}}{6(2t_{_3}+t_{_4})^2},
\end{flalign}
where $t_{3},t_{4}$ are constants in (\ref{Ints}).
\end{lemma}
\begin{proof}

It is observed from (\ref{B}) and (\ref{B_0}) that every entry of $\tfrac{1}{h_{K}}\overline{\mathbb{B}}_{K,\lambda}$ is a linear combination of $ \frac{r_{jk}}{h_{K}}$  $\big((j,k)\in\mathcal{Z}_{4}^{(2)}\big)$. By (\ref{def:rij}), since $ K $ is a regular tetrahedron, we get $ \tfrac{r_{jk}}{h_{K}} =\tfrac{|\overline{P_{j}P_{k}}|\cot\theta_{jk} }{h_{K}}=\cot\theta_{jk}=c$  $\,\,\forall(j,k)\in\mathcal{Z}_{4}^{(2)}$, where $ c $ is a  positive constant.


Choosing an invertible $ 9\times 9 $ matrix
\begin{equation*}
\textbf{C}_{1}=
{\footnotesize \left(\begin{array}{r;{3pt/1pt}rrrrrr} 
		\mathbb{E}_{_{3}}&\multicolumn{3}{c}{ \mathbb{0}_{_{3}} } &\multicolumn{3}{c}{ \mathbb{0}_{_{3}} } \\  
		\hdashline[3pt/1pt]
		\multirow{3}{*}{$ \mathbb{0}_{_{3}} $}&1&0&0&0&0&0\\
		&-1&1&0&0&0&0\\
		&0&-1&1&0&0&0\\
		\multirow{3}{*}{$ \mathbb{0}_{_{3}} $}&0&0&1&1&0&0\\
		&0&-2&2&-1&1&0\\
		&0&2&0&0&-1&1\\
	\end{array}
	\right)},
\end{equation*} 
such that the first eight entries of $ \textbf{C}_{1}^T\textbf{T}^T\textbf{S}\textbf{v}_{1} $ and $ \textbf{v}_{2}^T\textbf{T}\textbf{C}_{1} $ vanish, then parameter $ \gamma $ only appears in the corner of $ \textbf{C}_{1}^{T}\tfrac{1}{h_{K}}\overline{\mathbb{B}}_{K,\lambda}\textbf{C}_{1} $. According to Lemma~\ref{determine_relation}, substituting  $ s_{_1}=1/6 $ and $ s^{*}=0 $ into (\ref{part2}), we compute the determinants that
\begin{equation*}
\begin{split}
	\det\big(\textbf{C}_{1}^{T}\tfrac{1}{h_{K}}\overline{\mathbb{B}}_{K,\lambda}\textbf{C}_{1} (1:4,1:4)\big)&=c^4\varphi_1(s_{_2},s_{_3}),\\
	\det\big(\textbf{C}_{1}^{T}\tfrac{1}{h_{K}}\overline{\mathbb{B}}_{K,\lambda}\textbf{C}_{1} (1:5,1:5)\big)&=8c^5(s_{_2}-s_{_3})\varphi_1(s_{_2},s_{_3}),\\
	\det\big(\textbf{C}_{1}^{T}\tfrac{1}{h_{K}}\overline{\mathbb{B}}_{K,\lambda}\textbf{C}_{1} (1:6,1:6)\big)&=\dfrac{81c^6}{4}(s_{_2}-s_{_3})\varphi_1^2(s_{_2},s_{_3}),\\
	\det\big(\textbf{C}_{1}^{T}\tfrac{1}{h_{K}}\overline{\mathbb{B}}_{K,\lambda}\textbf{C}_{1} (1:7,1:7)\big)&=\dfrac{243c^7}{8}(s_{_2}-s_{_3})^2\varphi_1^2(s_{_2},s_{_3}),\\
	\det\big(\textbf{C}_{1}^{T}\tfrac{1}{h_{K}}\overline{\mathbb{B}}_{K,\lambda}\textbf{C}_{1} (1:8,1:8)\big)&=\dfrac{2187c^8}{32}(s_{_2}-s_{_3})^2\varphi_1^3(s_{_2},s_{_3}),\\
	\det\big(\textbf{C}_{1}^{T}\tfrac{1}{h_{K}}\overline{\mathbb{B}}_{K,\lambda}\textbf{C}_{1} (1:9,1:9)\big)&=\dfrac{729c^9}{1280}(s_{_2}-s_{_3})^2\varphi_1^3(s_{_2},s_{_3})\,\varphi_2(s_{_0},s_{_2},s_{_3}),
\end{split}
\end{equation*}
where  $ s_{_0},s_{_2},s_{_3}$ are given by (\ref{s_{_0}-3}), and
\begin{equation*}
\begin{split}
	\varphi_1(s_{_2},s_{_3})&=\dfrac{2}{27}\big(-3(2s_{2}+s_{3})^2+(2s_{_2}+5s_{_3})-\dfrac{1}{12} \big)\\
	&=\dfrac{2}{27}\big(-3(2t_{_3}+t_{_4})^{2}\lambda^2+(2t_{_3}+5t_{_4})\lambda-\dfrac{1}{12} \big),\\
	\varphi_2(s_{_0},s_{_2},s_{_3})&=-240s_{_0}-20s_{_2}-10s_{_3}+1=\dfrac{5}{3}\alpha\beta(3-4\gamma)\lambda.
\end{split}
\end{equation*}

It is clear that  $ s_{_2}-s_{_3}=(t_{_3}-t_{_4})\lambda=(\alpha\beta^2/8 - \alpha\beta/4 + 1/18 )\lambda$ and $ \varphi_2(s_{_0},s_{_2},s_{_3}) $ have the same signs with $ \lambda $ for any $ \alpha \in (0,1/2)$, $ \beta\in (0,2/3)$ and $ \gamma \in (0,3/4) $. Therefore, $ \tfrac{1}{h_{K}}\overline{\mathbb{B}}_{K,\lambda} $ is positive definite if and only if both $ \varphi_1(s_{_2},s_{_3})$ and $ \lambda $ are positive. A straight calculation of $ \varphi_1(s_{_2},s_{_3})>0$ yields (\ref{p}), which implies $ \lambda>0 $.  In fact, since $ 2t_{_4}$, $ 2t_{_3}+3t_{_4} $ and $ 2t_{_3}+t_{_4} $ are all positive for $ \alpha \in (0,1/2)$ and $ \beta\in (0,2/3)$, we have $ (2t_{_3}+5t_{_4})-2\sqrt{2t_{_4}(2t_{_3}+3t_{_4})}>0 $ by considering $ 2t_{_4}$  and $ 2t_{_3}+3t_{_4}$ as values of $ x $ and $ y $ in  $ x+y-2\sqrt{xy}>0$  $ (y>x>0)$.
 \end{proof}
\begin{remark}
$ \lambda=1/(12t_{_3}+6t_{_4})=1/(1-3\alpha\beta) $ and $ \lambda=1 $ are always in the range (\ref{p}). 
\end{remark}

For any general tetrahedon $ K\in \mathcal{T}_{h}$, we complete a reduction of $\tfrac{1}{h_{K}}\overline{\mathbb{B}}_{K,\lambda}$ under the orthogonal condition on the surface (\ref{orthogonal_plane_para_t}). Our initial idea is to construct a $ 3\times3 $ block transformation matrix $ \textbf{C}_{2} $ for the $ 3\times3 $ block target matix $\mathbb{B}_{K,\lambda}$, such that the first block row and column of $\mathbb{B}_{K,\lambda}$ act on the other two block rows and columns. We mainly focus on these changes of the elements in the  first block row and column of $\mathbb{B}_{K,\lambda}$. Consider a transformation matrix depending on two parameters $ \eta_{1},\eta_{2} $ as follows

\begin{equation*}
\textbf{C}_{2}=
\left(\begin{array}{c;{3pt/1pt}cc} 
\mathbb{E}_{_{3}}&\mathbb{0}_{_{3}}&\mathbb{0}_{_{3}} \\  
\hdashline[3pt/1pt]
\eta_{1}\mathbb{E}_{_{3}} & \multicolumn{2}{c}{\multirow{2}{*}{$ \mathbb{E}_{_{6}}$}}  \\
\eta_{2}\textbf{C}_{_{3\!\times\!3}}& 
\end{array}
\right).
\end{equation*}
By (\ref{B_0}), the corresponding congruent matrix of $\mathbb{B}_{K,\lambda}$ is
\begin{equation}\label{trans_B}
\textbf{C}_{2}\mathbb{B}_{K,\lambda}\textbf{C}_{2}^{T}=\textbf{C}_{2}\textbf{T}^T\textbf{S}\textbf{A}\textbf{T}\textbf{C}_{2}^{T}-\textbf{C}_{2}\textbf{T}^T\textbf{S}\textbf{v}_{1}\textbf{v}_{2}^T\textbf{T}\textbf{C}_{2}^{T}.
\end{equation}
Noticing the first three elements of $ \textbf{T}^T\textbf{S}\textbf{v}_{1} $ and $ \textbf{v}_{2}^T\textbf{T} $ in (\ref{part1})  are all zero, we have
\begin{flalign*}
\begin{split}
\textbf{C}_{2}\textbf{T}^T\textbf{S}\textbf{v}_{1}\textbf{v}_{2}^T\textbf{T}\textbf{C}_{2}^{T}
=(\textbf{T}^T\textbf{S}\textbf{v}_{1})\big(\textbf{v}_{2}^T\textbf{T}\big)=
s_{_0}\left(\begin{array}{c;{3pt/1pt}cc} 
	\mathbb{0}_{_{3}} &\mathbb{0}_{_{3}}&\mathbb{0}_{_{3}} \\ \hdashline[3pt/1pt]
	\mathbb{0}_{_{3}}&6\textbf{J}^{(1)}_{_{3\!\times\!3}}&-6\textbf{J}^{(2)}_{_{3\!\times\!3}} \\[0.1cm]
	\mathbb{0}_{_{3}}&-2\textbf{J}^{(1)}_{_{3\!\times\!3}}&2\textbf{J}^{(2)}_{_{3\!\times\!3}}
\end{array}
\right),
\end{split}&
\end{flalign*}
where the symbolic matrices $\textbf{J}^{(1)}_{_{3\!\times\!3}} $, $ \textbf{J}^{(2)}_{_{3\!\times\!3}} $ can be found in Appendix A.2.  
By (\ref{part2}), we have
\begin{flalign*}
\begin{split}
\textbf{C}_{2}\textbf{T}^T\textbf{S}\textbf{A}\textbf{T}\textbf{C}_{2}^{T}=
\left({\small \begin{array}{c;{3pt/1pt}cc} 
		s_{_1}\textbf{M}^{K}_{_{3\!\times\!3}} &(\frac{1}{2}+\eta_{1})s_{_1}\textbf{M}^{K}_{_{3\!\times\!3}}\!+\!\frac{s^{*}}{2}\textbf{Q}^{(1)}_{_{3\!\times\!3}}&(\frac{1}{4}+\eta_{2})s_{_1}(\textbf{M}^{K}_{_{3\!\times\!3}}\textbf{C}_{_{3\!\times\!3}})\!+\!\frac{s^{*}}{4}\textbf{Q}^{(2)}_{_{3\!\times\!3}}  \\[0.2cm] \hdashline[3pt/1pt]
		\big((1+\eta_{1})s_{_1}\!-\!\frac{2s_{_2}+s_{_3}}{2}\big)\textbf{M}^{K}_{_{3\!\times\!3}} & \tilde{\textbf{L}}^{(1)}_{_{3\!\times\!3}}&\tilde{\textbf{L}}^{(2)}_{_{3\!\times\!3}} \\[0.2cm]
		(\eta_{2}s_{_1}\!+\!\frac{2s_{_2}+s_{_3}}{4})(\textbf{C}_{_{3\!\times\!3}}\textbf{M}^{K}_{_{3\!\times\!3}}) & \tilde{\textbf{L}}^{(3)}_{_{3\!\times\!3}}&\tilde{\textbf{L}}^{(4)}_{_{3\!\times\!3}} 
\end{array}}
\right)
\end{split}&
\end{flalign*}
with
\begin{flalign*}
\begin{split} &{\small\tilde{\textbf{L}}^{(1)}_{_{3\!\times\!3}}=\textbf{L}^{(1)}_{_{3\!\times\!3}}+\left( (\tfrac{3}{2}+\eta_{1})s_{_1}-\tfrac{2s_{_2}+s_{_3}}{2}\right) \eta_{1}\textbf{M}^{K}_{_{3\!\times\!3}}+\tfrac{s^{*}}{2}\eta_{1}\textbf{Q}^{(1)}_{_{3\!\times\!3}}},\\
&{\small \tilde{\textbf{L}}^{(2)}_{_{3\!\times\!3}}=\textbf{L}^{(2)}_{_{3\!\times\!3}}+\left(\tfrac{s_{_1}}{4}\eta_{1}+\big((1+\eta_{1})s_{_1}-\tfrac{2s_{_2}+s_{_3}}{2}\big)\eta_{2}\right)(\textbf{M}^{K}_{_{3\!\times\!3}}\textbf{C}_{_{3\!\times\!3}})+\tfrac{s^{*}}{4}\eta_{1}\textbf{Q}^{(2)}_{_{3\!\times\!3}}},\\
&{\small \tilde{\textbf{L}}^{(3)}_{_{3\!\times\!3}}=\textbf{L}^{(3)}_{_{3\!\times\!3}}+\left(\tfrac{s_{_1}}{2}\eta_{2}+(\eta_{2}s_{_1}+\tfrac{2s_{_2}+s_{_3}}{4})\eta_{1}\right)(\textbf{C}_{_{3\!\times\!3}}\textbf{M}^{K}_{_{3\!\times\!3}})+\tfrac{s^{*}}{2}\eta_{2}(\textbf{C}_{_{3\!\times\!3}}\textbf{Q}^{(1)}_{_{3\!\times\!3}})},\\
&{\small \tilde{\textbf{L}}^{(4)}_{_{3\!\times\!3}}=\textbf{L}^{(4)}_{_{3\!\times\!3}}+\left(\tfrac{s_{_1}}{4}\eta_{2}+(\eta_{2}s_{_1}+\tfrac{2s_{_2}+s_{_3}}{4})\eta_{2}\right)(\textbf{C}_{_{3\!\times\!3}}\textbf{M}^{K}_{_{3\!\times\!3}}\textbf{C}_{_{3\!\times\!3}})+\tfrac{s^{*}}{4}\eta_{2}(\textbf{C}_{_{3\!\times\!3}}\textbf{Q}^{(2)}_{_{3\!\times\!3}})}.
\end{split}
\end{flalign*}
Then, by (\ref{trans_B}), we obtain
\begin{equation*}
\textbf{C}_{2}\tfrac{1}{h_{K}}\overline{\mathbb{B}}_{K,\lambda}	\textbf{C}_{2}^{T}=\tfrac{1}{2h_{K}}\textbf{C}_{2}(\mathbb{B}_{K,\lambda}+\mathbb{B}_{K,\lambda}^{T})\textbf{C}_{2}^{T}=\\
\tfrac{1}{h_{K}} \left(\begin{array}{c;{3pt/1pt}cc} 
s_{_1}\textbf{M}^{K}_{_{3\!\times\!3}} & (\star)_{_{1}}& (\star)_{_{2}} \\ \hdashline[3pt/1pt]
(\star)^{T}_{1} &\multicolumn{2}{c}{\multirow{2}{*}{$ \textbf{N}^{K,\lambda}_{_{6\!\times\!6}}$}}\\
(\star)^{T}_{2}& 
\end{array}
\right),
\end{equation*}
where 
\begin{gather} {\small (\star)_{_{1}}=\big((\dfrac{3}{4}+\eta_{1})s_{_1}\!-\!\dfrac{2s_{_2}+s_{_3}}{4}\big)\textbf{M}^{K}_{_{3\!\times\!3}}\!+\!\dfrac{s^{*}}{4}\textbf{Q}^{(1)}_{_{3\!\times\!3}}},\nonumber\\ {\small (\star)_{_{2}}=\big((\dfrac{1}{8}+\eta_{2})s_{_1}\!+\!\dfrac{2s_{_2}+s_{_3}}{8}\big)(\textbf{M}^{K}_{_{3\!\times\!3}}\textbf{C}_{_{3\!\times\!3}})\!+\!\dfrac{s^{*}}{8}\textbf{Q}^{(2)}_{_{3\!\times\!3}}},\nonumber
\end{gather}
and
\begin{gather}
\textbf{N}^{K,\lambda}_{_{6\!\times\!6}}={\small \frac{1}{2} \left(\begin{array}{cc} 
	(\tilde{\textbf{L}}^{(1)}_{_{3\!\times\!3}}-6s_{_0}\textbf{J}^{(1)}_{_{3\!\times\!3}})+(\tilde{\textbf{L}}^{(1)}_{_{3\!\times\!3}}-6s_{_0}\textbf{J}^{(1)}_{_{3\!\times\!3}})^{T} & (\tilde{\textbf{L}}^{(2)}_{_{3\!\times\!3}}+6s_{_0}\textbf{J}^{(2)}_{_{3\!\times\!3}})+(\tilde{\textbf{L}}^{(3)}_{_{3\!\times\!3}}+2s_{_0}\textbf{J}^{(1)}_{_{3\!\times\!3}})^{T} \\[2mm] 
	(\tilde{\textbf{L}}^{(3)}_{_{3\!\times\!3}}+2s_{_0}\textbf{J}^{(1)}_{_{3\!\times\!3}})+(\tilde{\textbf{L}}^{(2)}_{_{3\!\times\!3}}+6s_{_0}\textbf{J}^{(2)}_{_{3\!\times\!3}})^{T}&(\tilde{\textbf{L}}^{(4)}_{_{3\!\times\!3}}-2s_{_0}\textbf{J}^{(2)}_{_{3\!\times\!3}})+(\tilde{\textbf{L}}^{(4)}_{_{3\!\times\!3}}-2s_{_0}\textbf{J}^{(2)}_{_{3\!\times\!3}})^{T}\\
\end{array}
\right)}.\label{N}
\end{gather}
Thanks to the orthogonal condition on the surface (\ref{orthogonal_plane_para_t}) which means $ s^*=0 $, and meanwhile, choosing parameters $\eta_{1},\eta_{2}$ as follows
\begin{equation*}
\eta_{1}=\dfrac{2s_{_2}+s_{_3}}{4s_{_1}}-\dfrac{3}{4},\qquad\eta_{2}=-\dfrac{2s_{_2}+s_{_3}}{8s_{_1}}-\dfrac{1}{8},
\end{equation*}
then we have $ (\star)_{_{1}}=\mathbb{0}_{_{3}} $ and $ (\star)_{_{2}}=\mathbb{0}_{_{3}} $. Thus, $ \frac{1}{h_{K}}\overline{\mathbb{B}}_{K,\lambda} $ is reduced to a block diagonal matrix. According to Lemma~\ref{determine_relation},  a positive definiteness of  $ \frac{1}{h_{K}}\overline{\mathbb{B}}_{K,\lambda} $ is reducible to a positive definiteness of  $ \frac{1}{h_{K}}\textbf{N}^{K,\lambda}_{_{6\!\times\!6}} $, which is proved in the next subsection.
\subsection{A restriction on $\mathcal{T}_{h}$ and the stability}\label{subsection:restriction3} 

In this subsection, the minimum V-angle condition on  $ \mathcal{T}_{h}$ is proposed to ensure the positive definiteness of $ \frac{1}{h_{K}}\textbf{N}^{K,\lambda}_{_{6\!\times\!6}} $ numerically, furthermore, the stability of the quadratic FVM schemes is presented.

\vspace{2mm}
For simplicity, write all plane angles in $ K=\bigtriangleup^{4} P_{1}P_{2}P_{3}P_{4}  $ (see Fig.~\ref{fig:tetrahedron}) as
\begin{equation}\label{definition:plane angles}
\left\{\begin{array}{lcl}
\theta_{_{1,P_{1}}}=\angle P_2P_1P_4,&\,\, \theta_{_{2,P_{1}}}=\angle P_2P_1P_3,&\,\,  \theta_{_{3,P_{1}}}=\angle P_3P_1P_4,\\
\theta_{_{1,P_{2}}}=\angle P_1P_2P_4,&\,\, \theta_{_{2,P_{2}}}=\angle P_1P_2P_3,&\,\, \theta_{_{3,P_{2}}}=\angle P_3P_2P_4,\\
\theta_{_{1,P_{3}}}=\angle P_1P_3P_4,&\,\,  \theta_{_{2,P_{3}}}=\angle P_1P_3P_2,&\,\,  \theta_{_{3,P_{3}}}=\angle P_2P_3P_4,\\
\theta_{_{1,P_{4}}}=\angle P_1P_4P_2,&\,\,  \theta_{_{2,P_{4}}}=\angle P_1P_4P_3,&\,\,  \theta_{_{3,P_{4}}}=\angle P_2P_4P_3.
\end{array} \right.
\end{equation}
Firstly, we define the $``$V-angle$"$ at the vertex $P_{i}  $ of $ K $ as 
\begin{equation}\label{definition:vertex angles}
\theta_{_{P_{i}}}=\theta_{_{1,P_{i}}}+\theta_{_{2,P_{i}}}+\theta_{_{3,P_{i}}}-2\max\left\lbrace \theta_{_{1,P_{i}}},\theta_{_{2,P_{i}}},\theta_{_{3,P_{i}}}\right\rbrace, \quad i\in\mathcal{Z}_{4}^{(1)}.
\end{equation} 
In fact, $ \theta_{_{P_{i}}} $ is the difference between the sum of two smaller plane angles and the largest one at vertex $ P_{i} $. Let $ \theta_K:=\min\{\theta_{_{P_{i}}},i\in\mathcal{Z}_{4}^{(1)}\} $, the following Lemma~\ref{lemma:theta_K} shows the properties of $ \theta_K$.
\begin{lemma}\label{lemma:theta_K}
The minimum V-angle $ \theta_K$ has the following properties
\begin{itemize}
\item[(i)] $ \theta_{_{i_{1},P_{i_{2}}}}\geq\theta_K\qquad\forall i_{1}\in\mathcal{Z}_{3}^{(1)},\,i_{2}\in\mathcal{Z}_{4}^{(1)}  $;
\item[(ii)] $ 0^{\circ}<\theta_{K}\leq60^{\circ}  $, and $\theta_{K}=60^{\circ} $ means that $ K $ is a regular tetrahedron.
\end{itemize}
\end{lemma}
\begin{proof}
Combining  $ \theta_{_{1,P_{i}}}\!+\theta_{_{2,P_{i}}}\!+\theta_{_{3,P_{i}}}\leq 2\max\left\lbrace  \theta_{_{1,P_{i}}},\theta_{_{2,P_{i}}},\theta_{_{3,P_{i}}}\right\rbrace+\min\left\lbrace \theta_{_{1,P_{i}}},\theta_{_{2,P_{i}}},\theta_{_{3,P_{i}}}\right\rbrace $  $ (i\in\mathcal{Z}_{4}^{(1)}) $ and (\ref{definition:vertex angles}),  we have $ \theta_{_{P_{i}}}\leq \min\left\lbrace \theta_{_{1,P_{i}}},\theta_{_{2,P_{i}}},\theta_{_{3,P_{i}}}\right\rbrace $ $(i\in\mathcal{Z}_{4}^{(1)}) $, then the first property follows from $ \theta_K=\min\{\theta_{_{P_{i}}},i\in\mathcal{Z}_{4}^{(1)}\} $. It is clear that $ \theta_{K}>0^{\circ}$, then we prove $ \theta_{K}\leq60^{\circ} $. If not, every plane angle of $ K $ is larger than $ 60^{\circ} $, which contradicts the fact that the sum of the three interior angles of a triangle is $ 180^{\circ} $. Moreover, if $ \theta_{K}=60^{\circ} $, then all the plane angles in $ K $  equal to $ 60^{\circ} $, this means that $ K $ is  a regular tetrahedron . 
 \end{proof}

\begin{remark}\label{bad restriction}
The poorly-shaped tetrahedrons classified in \cite{Liu1994} possess at least one of the following two types of local shapes around a vertex.
i) See Fig.~\ref{fig:local_shape_1} , the three plane angles around a vertex are all small, the local shape of which is performed as $``sharp"$; ii) See Fig.~\ref{fig:local_shape_2},  there exists an edge  that is  close to the angle (not small) formed by the other two edges around a vertex, the local shape of which is performed as $``flat"$.
The minimum V-angle $ \theta_K $ bounded below ensures the shape regularity of the tetrahedron $K$ by controlling the local shapes at each vertex. However, if we only 
restrict the minimum plane angle for each triangular face of a tetrahedron $ K $, similar to \cite{ChenZY2012,XuJC2009,ZhouYH2020,ZhouYH2020_2}  for triangular meshes in 2D, then we can not guarantee the shape regularity of the tetrahedron $ K $. See Fig.~\ref{fig:terrible_tetrahedron} as an example, the minimum plane angle of a degenerated tetrahedron $ K $ is $ 45^{\circ} $, which is big enough in $ (0^{\circ}, 60^{\circ}] $.
\end{remark} 
\begin{figure}[ht!]
\centering
\subfigure[The local shape around a vertex is $``sharp"$.]{
\begin{minipage}[t]{.45\textwidth}
	\centering
	\includegraphics[width=17pt]{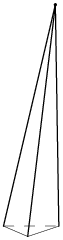}
	\label{fig:local_shape_1}
\end{minipage}}
\subfigure[The local shape around a vertex is $``flat"$.]{\begin{minipage}[t]{.45\textwidth}
	\centering
	\includegraphics[width=50pt]{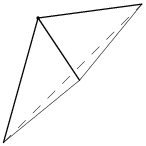}
	\label{fig:local_shape_2}
\end{minipage}}
\caption{Illustration of two types of local shapes.}
\end{figure}   
\begin{figure}[ht!]
\centering
\begin{minipage}[t]{.7\textwidth}
\centering
\includegraphics[width=40pt]{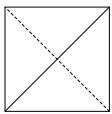}
\caption{A degenerated tetrahedron $ K $ whose four vertices coinciding with the vertices of a square.}
\label{fig:terrible_tetrahedron}
\end{minipage}
\end{figure}

Let  $ \Theta_{5}:=\left(  \theta_{_{1,P_{1}}},\theta_{_{2,P_{1}}},\theta_{_{1,P_{2}}},\theta_{_{2,P_{2}}},\theta_{_{3,P_{2}}} \right)  $, five plane angles of $K$. Lemma~\ref{lemma:indentity} implies that a tetrahedon $ K $ is determined by its circumradius $ R_{K} $ and $ \Theta_{5} $. The proof is included in Appendix B.3.

\begin{lemma}\label{lemma:indentity}
All plane angles $ \theta_{i_1,P_{i_2}}$  $ (i_1\in\mathcal{Z}_{3}^{(1)},i_2\in\mathcal{Z}_{4}^{(1)}) $ in (\ref{definition:plane angles}) and dihedral angles $ \theta_{jk} $ $ \big( (j,k)\in\mathcal{Z}_{4}^{(2)} \big) $ in a tetrahedron $ K $ can be represented by the five plane angles $ \Theta_{5} $.
\end{lemma}
According to Lemma~\ref{lemma:angles}, Lemma~\ref{lemma:L-R} listed in Appendix B.1 and Lemma~\ref{lemma:indentity}, each of $ r_{jk}$  $ \big((j,k)\in\mathcal{Z}_{4}^{(2)}\big) $ defined  in (\ref{def:rij}) can be represented by $ R_{K} $ multiplying by a continuous function of $ \Theta_{5} $. Taking $ r_{12} $ as an example, we have
\begin{equation*}
{\small r_{_{\!12}}=R_{K} \dfrac{2\cos\theta_{12}}{\sqrt{(1-\cos^2\theta_{12})+\cot^2\theta_{2,P_3}+\cot^2\theta_{1,P_4}-2\,\cot\theta_{2,P_3}\cot\theta_{1,P_4}\cos\theta_{12}}}},
\end{equation*}
and
\begin{gather*}
{\small \theta_{_{1,P_{4}}}=180^{\circ}-\theta_{_{1,P_{1}}}-\theta_{_{1,P_{2}}}},\quad
{\small \theta_{_{2,P_{3}}}=180^{\circ}-\theta_{_{2,P_{1}}}-\theta_{_{2,P_{2}}}},\\
{\small \cos\theta_{12}=\dfrac{\cos\theta_{_{3,P_{2}}}-\cos \theta_{_{1,P_{2}}}\cos \theta_{_{2,P_{2}}}}{\sin \theta_{_{1,P_{2}}}\sin \theta_{_{2,P_{2}}}}}.
\end{gather*}
For simplicity, the similar
representations of other $ r_{jk}$ $ \big( (j,k)\in\mathcal{Z}_{4}^{(2)} \big) $ are omitted here.

Obviously, every entry of $\textbf{N}^{K,\lambda}_{_{6\!\times\!6}}$ in (\ref{N}) is a linear combination of $r_{jk}$ $ \big((j,k)\in\mathcal{Z}_{4}^{(2)}\big)$. Thus, $\frac{1}{h_{K}}\textbf{N}^{K,\lambda}_{_{6\!\times\!6}}$ can be rewritten as $ \frac{R_{K}}{h_{K}}\widetilde{\textbf{N}}_{6\times6}(\Theta_{5},\alpha,\beta,\gamma,\lambda) $.
We turn to discuss the positive definiteness of $\widetilde{\textbf{N}}_{6\times6}(\Theta_{5},\alpha,\beta,\gamma,\lambda) $ under a regular partition $ \mathcal{T}_{h} $. 
Subsequently, for any fixed  group  of parameters $ (\alpha,\beta,\gamma,\lambda)$, we   restrict the lower bound of the minimum V-angle $ \theta_{K} $ to ensure the positiveness of $ \det\big( \widetilde{\textbf{N}}_{6\times6}(\Theta_{5},\alpha,\beta,\gamma,\lambda)\big)$.

Denote the reasonable range of $ \Theta_{5} $ for a tetrahedron $ K $ satisfying $ \theta_K \geq v$ by $ \mathbb{Q}_{v} $ as below

\begin{equation*}
\begin{split}
\mathbb{Q}_{v}\!=\!\left\lbrace \Theta_{5}>0^{\circ}\big| \theta_{_{1,P_{1}}}\!\!\!+\!\theta_{_{1,P_{2}}}\!\!<\!180^{\circ};\,
\theta_{_{2,P_{1}}}\!\!\!+\!\theta_{_{2,P_{2}}}\!\!<\!180^{\circ};\,
\theta_{_{1,P_{2}}}\!\!\!+\!\theta_{_{2,P_{2}}}\!\!\!+\!\theta_{_{3,P_{2}}}\!\!<\!360^{\circ};\,\theta_{_{P_{i}}}\geq v,i\in\mathcal{Z}_{4}^{(1)}\right\rbrace.
\end{split}
\end{equation*}
By Lemma~\ref{lemma:indentity}, $ \theta_{_{P_{i}}}\geq v$  $(i\in\mathcal{Z}_{4}^{(1)}) $ means some relations between $ \Theta_{5} $ and $ v $. In addition, Lemma~\ref{lemma:theta_K} indicates that each angle of $ \Theta_{5}\in \mathbb{Q}_{v} $ lies in  $ [v,180^{\circ}-2v] $. 
For parameters $ (\alpha,\beta,\gamma,\lambda)$ fixed by
(\ref{parameters}), (\ref{orthogonal_plane_para}) and (\ref{p}), we  define an angle set as
\begin{equation}\label{def:V*}
V^{*}(\alpha,\beta,\gamma,\lambda)=\{v \in\left(0^{\circ},60^{\circ} \right]\, \big|\, \det\big( \widetilde{\textbf{N}}_{6\times6}(\Theta_{5},\alpha,\beta,\gamma,\lambda)\big)>0\,\,\,\,
\forall \Theta_{5}\in \mathbb{Q}_{v} \}.
\end{equation}
Note that $V^{*}(\alpha,\beta,\gamma,\lambda)$ is nonempty since Lemma~\ref{Lemma:regular tetra} shows $ 60^{\circ}\in V^{*}(\alpha,\beta,\gamma,\lambda)$.
Let 
$ v^{*}(\alpha,\beta,\gamma,\lambda)=\inf\, V^{*}(\alpha,\beta,\gamma,\lambda) $, the following restriction on the primary mesh $ \mathcal{T}_{h} $ plays an important role for stability analysis.
\begin{definition}[minimum V-angle condition]
A quadratic FVM scheme or the corresponding primary mesh $ \mathcal{T}_{h} $ is called to satisfy the minimum V-angle condition, if there exist $ \varepsilon^{*}>0 $ and $ \lambda $ in the range (\ref{p}) such that 
\begin{flalign} \label{restriction_3}
\theta_{K}\geq v^{*}(\alpha,\beta,\gamma,\lambda)+\varepsilon^{*}, \quad\forall K \in  \mathcal{T}_{h}.
\end{flalign} 
\end{definition}
\begin{remark}
Under the traditional mapping $ \Pi_{h}^{*} $,  parameter $ \lambda $  in the minimum V-angle condition (\ref{restriction_3}) is fixed to be 1. The mapping $ \Pi_{\lambda}^{*} $  gives us more chances to find a better $ \lambda $ in (\ref{p}), such that  $ v^{*}(\alpha,\beta,\gamma,\lambda) $ is smaller, which leads (\ref{restriction_3}) to be a weaker restriction. Acctually, for a given scheme (fixed $ \alpha,\beta,\gamma $), we care about when $ v^{*}(\alpha,\beta,\gamma,\lambda) $ reaches its minimum value for $ \lambda $.
\end{remark}
\begin{remark}
The minimum V-angle condition (\ref{restriction_3}) for tetrahedral meshes is as convenient as  the minimum angle condition \cite{ChenZY2012,XuJC2009,ZhouYH2020,ZhouYH2020_2}  for 2D triangular meshes in application.  Other restrictions on tetrahedral meshes are referred to \cite{Liu1994}. 
\end{remark}

See \textbf{Algorithm}~\ref{searching v}, we show a way to compute $ v^{*}(\alpha,\beta,\gamma,\lambda) $ numerically. The basic idea is to find the minimum $ v \in\left(0^{\circ},60^{\circ} \right] $ by the bisection process, such that $ \det\big( \widetilde{\textbf{N}}_{6\times6}(\Theta_{5},\alpha,\beta,\gamma,\lambda)\big)$ is positive for  $
\Theta_{5}\in \mathbb{Q}_{v} $. In this process, for each $ v $, we compute to check whether the positiveness of $ \det\big( \widetilde{\textbf{N}}_{6\times6}(\Theta_{5},\alpha,\beta,\gamma,\lambda)\big)$ is satisfed for $ \Theta_{5} $ in  $\mathbb{P}_{v}^{(1)}\cap\mathbb{Q}_{v}$, ..., $\mathbb{P}_{v}^{(q)}\cap\mathbb{Q}_{v}$, where  $ \mathbb{P}_{v}^{(n)} $ $(n\in\mathcal{Z}_{q}^{(1)}) $ are discrete point sets. Here points in $ \mathbb{P}_{v}^{(n)} $ are evenly selected in $ [v,180^{\circ}-2v]^{5} $, i.e.,
\begin{equation*}
\mathbb{P}_{v}^{(n)}\!=\!\left\lbrace (v,v,v,v,v)+\frac{180^{\circ}-3v}{N_{n}}(i_{1},i_{2},i_{3},i_{4},i_{5})\,\,\,\,\forall i_{1},...,i_{5}\in\{0\}\cup\mathcal{Z}_{N_n}^{(1)}\right\rbrace,
\end{equation*}
where $ N_n $ is a division number of $ [v,180^{\circ}-2v] $. We take $ N_n $ $(n\in\mathcal{Z}_{q}^{(1)})$ as  a group of increasing prime numbers to avoid repeating calculations of $ \det\big( \widetilde{\textbf{N}}_{6\times6}(\Theta_{5},\alpha,\beta,\gamma,\lambda)\big) $. 
In Section~\ref{Section:5}, we will show the numerical performances of  $ v^{*}(\alpha,\beta,\gamma,\lambda) $ for four given quadratic FVM schemes.
\begin{algorithm}[ht!]
\caption{}
\label{searching v}
Given $ \alpha\in(0,1/2)$ and $\beta\in(0,2/3)$ satisfying (\ref{orthogonal_plane_para}), $\lambda $ in (\ref{p}), $\gamma\in (0,3/4)$, a group of increasing prime numbers $ N_{n}$  $(n\in\mathcal{Z}_{q}^{(1)})$, and precision  $\epsilon>0^{\circ}$;\\
Set $ \mathbb{t}_{0}=0^{\circ} $, $ \mathbb{t}_{1}=60^{\circ} $ ;\\
\While{$\mathbb{t}_{1}-\mathbb{t}_{0}>\epsilon$}{
$ v_{0}=\dfrac{\mathbb{t}_{0}+\mathbb{t}_{1}}{2} $; n=1;\\
\While{$ \det\big( \widetilde{\textbf{N}}_{6\times6}(\Theta_{5},\alpha,\beta,\gamma,\lambda)\big)>0$ for $\Theta_{5}\in \mathbb{P}_{v_{0}}^{(n)}\cap\mathbb{Q}_{v_{0}}$, and $ n\leq q $}{
n=n+1;

}
\If{n=q+1}{Set $ \mathbb{t}_{1}=v_{0}$;

\Else{Set $ \mathbb{t}_{0}=v_{0}$;}
}}	
$ v^{*}(\alpha,\beta,\gamma,\lambda)=\mathbb{t}_{1} $.
\end{algorithm}
\begin{lemma}\label{lemma:continuity}
Assume that  $ \textbf{F}_{_{m\!\times\!m}}(X)$  $(X\in S_{_0}) $ is a $ m\times m $ real symmetric matrix and elementwisely continunous in  a connected region $ S_{_0} $ of $ \mathbb{R}^{n} $. If there exists $ X_{_1}\in S_{_0}$ such that $ \textbf{F}_{_{m\!\times\!m}}(X_{_1}) $ is positive definite, and the determinant of $ \textbf{F}_{_{m\!\times\!m}}(X) $ is always positive in $ S_{_0} $, then $ \textbf{F}_{_{m\!\times\!m}}(X) $ is positive definite for every $ X\in S_{_0}$.
\end{lemma}
\begin{proof}
If not, there exist even number of eigenvalues of $ \textbf{F}_{_{m\!\times\!m}}(X) $ for a certain $ X_{_2}\in S_{_0} $ that are negative. Let one of these negative eigenvalues be $ \mu_{_0}(X_{_2})$  $(<0) $. Since $ \mu_{_0}(X_{_1})>0 $, by a simple continuity argument,  there is a point $ X_{_3}\in S_{0} $ such that  $ \mu_{_0}(X_{_3})=0 $. This contradicts that the determinant of $ \textbf{F}_{_{m\!\times\!m}}(X) $ is positive for every $ X\in S_{_0}$. 
 \end{proof}

Under the orthogonal condition on the surface (\ref{orthogonal_plane_para}), Lemma~\ref{Lemma:regular tetra} implies that $ \widetilde{\textbf{N}}_{6\times6}(\Theta_{5},\alpha,\beta,\gamma,\lambda) $ with $\Theta_{5}=(60^{\circ},60^{\circ},60^{\circ},60^{\circ},60^{\circ})  $ is positive definite for given $ \lambda $  in (\ref{p}). On the other hand, if $ \mathcal{T}_{h} $ satisfies the minimum V-angle condition (\ref{restriction_3}), then $ \det\left(\widetilde{\textbf{N}}_{6\times6}(\Theta_{5},\alpha,\beta,\gamma,\lambda)\right)>0$ for $\Theta_{5}\in \mathbb{Q}_{v} $ holds for every $ K\in\mathcal{T}_{h} $. Thus, under the two restrictions (\ref{orthogonal_plane_para}) and (\ref{restriction_3}), Lemma~\ref{lemma:continuity} ensures the local stability (\ref{eq:constant_element ellipitic}), in which the hidden constants have a common lower bound. According to above disscussions, we present the stability as follows.

\begin{theorem}[Stability]\label{theorem:stability}
Assume that the diffusion coefficient  $ \kappa $ is piecewise constant over $ \mathcal{T}_{h} $. If a quadratic FVM scheme (\ref{eq:FVM}) satisfies  the orthogonal condition on the surface (\ref{orthogonal_plane_para}) and the minimum V-angle condition (\ref{restriction_3}),
then the local stability
(\ref{eq:constant_element ellipitic}) holds. Furthermore, the bilinear
form $ a_{h}(\cdot,\Pi_{\lambda}^{*}\cdot) $ is uniformly ellipitic, i.e.,
\begin{equation*}
a_{h}(u_{h},\Pi_{\lambda}^{*}u_{h})\gtrsim |u_{h}|_{1}^{2},   \qquad  \forall  u_{h}\in \mathit{U}_{h}.
\end{equation*}
\end{theorem}

\section{Error analysis}\label{Section:4}
In this section, we present optimal $ H^{1} $ and $ L^{2} $ error estimates of the quadratic FVM schemes.  The $ H^{1} $ error estimate is based on the continuity and the stability, and the $ L^{2} $ error estimate follows the $ H^{1} $ result and the orthogonal conditions (\ref{orthogonal_plane}) and (\ref{orthogonal_space}).

\vspace{2mm}
Define the piecewise $ H^{2} $ space over $\mathcal{T}_{h}$ as
\begin{equation*}
\mathit{H}_{h}^{2}(\Omega)=\{ u\in \mathit{C}(\Omega): u\vert_{K} \in H^{2}(K)\,\,\,\, \forall K\in \mathcal{T}_{h}\}.
\end{equation*}
Then we have the continuity of the quadratic FVM schemes (\ref{eq:FVM}).
\begin{lemma}
For the bilinear form $ a_{h}(\cdot,\Pi_{\lambda}^{*}\cdot) $ in (\ref{def:bilinear form}), we have
\begin{equation}\label{continue}
|a_{h}(u,\Pi_{\lambda}^{*}u_{h})| \lesssim (|u|_{1}+h|u|_{2,h})|u_{h}|_{1}   \qquad  \forall \, u\in \mathit{H}_{0}^{1}(\Omega)\cap\mathit{H}_{h}^{2}(\Omega) ,\, u_{h}\in \mathit{U}_{h},
\end{equation}
where $|u|_{2,h}= \big(\sum\limits_{K\in \mathcal{T}_{h}}|u|^{2}_{2,K}\big)^{\frac{1}{2}} $.
\end{lemma}
\begin{proof}
For  $ u\in \mathit{H}_{0}^{1}(\Omega)\cap\mathit{H}_{h}^{2}(\Omega)$ and $ u_{h}\in \mathit{U}_{h} $, we rewrite
\begin{equation*}
a_{h}(u,\Pi_{\lambda}^{*}u_{h})=-\sum\limits_{K\in \mathcal{T}_{h}}\sum\limits_{s^{*}\in \mathcal{S}_{K}^{*}}\int\!\!\!\int_{s^{*}}(\kappa\nabla u)\cdot \textbf{n}\, [\Pi_{\lambda}^{*}u_{h}]_{s^{*}}\mathrm{d} S, 
\end{equation*}
where $ {S}_{K}^{*}$ is the set of all common faces of the dual elements contained in the interior of $ K $. For all polygonal faces $ s^{*} $ in $ {S}_{K}^{*}$, this means  
$ \cup_{s^{*}\in{S}_{K}^{*}}s^{*}=\cup_{K^{*}\in\mathcal{T}_{h}^{*}}   \big(\partial K^{*}\cap K\big)$.
And $ \textbf{n} $ is the unit normal vector on $ s^{*} $ from a dual element $ K^{*}_{1} $ to its neighboring dual element $ K^{*}_{2} $, and $ [\Pi_{\lambda}^{*}u_{h}]_{s^{*}}:=\Pi_{\lambda}^{*}u_{h}\vert_{K^{*}_{1}}-\Pi_{\lambda}^{*}u_{h}\vert_{K^{*}_{2}} $ is the jump of $ \Pi_{\lambda}^{*}u_{h} $ on $ s^{*} $.

Then, by the Cauchy-Schwartz inequality
\begin{align}
|a_{h}(u,\Pi_{\lambda}^{*}u_{h})| &\lesssim \left( \sum\limits_{K\in \mathcal{T}_{h}}\sum\limits_{s^{*}\in \mathcal{S}_{K}^{*}}\!\!|s^{*}|^{-1}\!\!\int\!\!\!\int_{s^{*}}[\Pi_{\lambda}^{*}u_{h}]_{s^{*}}^2\mathrm{d} S \right) ^{\frac{1}{2}}\!\! \left( \sum\limits_{K\in \mathcal{T}_{h}}\sum\limits_{s^{*}\in \mathcal{S}_{K}^{*}}\!\!|s^{*}|\int\!\!\!\int_{s^{*}}((\kappa\nabla u)\cdot \textbf{n})^{2}\, \mathrm\mathrm{d} S \right) ^{\frac{1}{2}}\nonumber\\
&\lesssim\left( \sum\limits_{K\in \mathcal{T}_{h}}\sum\limits_{s^{*}\in \mathcal{S}_{K}^{*}}\!\!h_{K}\,[\Pi_{\lambda}^{*}u_{h}]_{s^{*}}^2 \right) ^{\frac{1}{2}}\!\!\left( \sum\limits_{K\in \mathcal{T}_{h}}\sum\limits_{s^{*}\in \mathcal{S}_{K}^{*}}\!\frac{|s^{*}|}{h_{K}}\int\!\!\!\int_{s^{*}}|\nabla u|^{2}\, \mathrm\mathrm{d} S \right) ^{\frac{1}{2}}. \label{lemma_eq:continuity_1}
\end{align}

We start with the first term on the right hand of (\ref{lemma_eq:continuity_1}). Definition~\ref{definition:mapping} indicates that $ [\Pi_{\lambda}^{*}u_{h}]_{s^{*}}$   $(s^{*}\in \mathcal{S}_{K}^{*})$ are linear combinations of $ u_{h}(P_{i}) $ $ (i\in \mathcal{Z}_{10}^{(1)} ) $. Thus, there exists a matrix $ \textbf{K}$ with 10 columns such that
\begin{equation*}
\sum\limits_{s^{*}\in \mathcal{S}_{K}^{*}}[\Pi_{\lambda}^{*}u_{h}]_{s^{*}}^2=(\textbf{K}\textbf{u}_K)^{T}(\textbf{K}\textbf{u}_K),
\end{equation*}
where $ \textbf{u}_{K}=(u_{1},u_{2},...,u_{10})^T $ and $ u_{i}=u_{h}(P_{i}) $ $(i\in\mathcal{Z}_{10}) $. Obviously, $u_{i_{1}}=u_{i_{2}}$ for $ i_{1},i_{2}\in \mathcal{Z}_{10}^{(1)} $ yields $ [\Pi_{\lambda}^{*}u_{h}]_{s^{*}}=\Pi_{\lambda}^{*}u_{h}\vert_{K^{*}_{1}}-\Pi_{\lambda}^{*}u_{h}\vert_{K^{*}_{2}}=0$ for $ s^{*}\in \mathcal{S}_{K}^{*}$. Thus, taking $\textbf{u}_{K}=u_{1}(1,1,...,1)^T$,  one obtains $ \sum_{s^{*}\in \mathcal{S}_{K}^{*}}[\Pi_{\lambda}^{*}u_{h}]_{s^{*}}^2=0 $ which indicates $ \textbf{K}\textbf{u}_{K}$ being a zero vector. That means the row sum of $ \textbf{K}$  is zero. By relation (\ref{TG}), one arrives at
\begin{equation*}
(\textbf{K}\textbf{u}_K)^{T}(\textbf{K}\textbf{u}_K)=(\textbf{K}\textbf{T}\textbf{G}\textbf{u}_K)^{T}(\textbf{K}\textbf{T}\textbf{G}\textbf{u}_K)=(\textbf{G}\textbf{u}_K)^{T}(\textbf{K}\textbf{T})^{T}(\textbf{K}\textbf{T})(\textbf{G}\textbf{u}_K)\lesssim(\textbf{G}\textbf{u}_K)^{T}(\textbf{G}\textbf{u}_K).
\end{equation*}
Recalling the equivalent discrete  norm (\ref{equi_norm}), we have
\begin{equation}\label{Pip_uh}
\left( \sum\limits_{K\in \mathcal{T}_{h}}\sum\limits_{s^{*}\in \mathcal{S}_{K}^{*}}h_{K}\,[\Pi_{\lambda}^{*}u_{h}]_{s^{*}}^2 \right) ^{\frac{1}{2}} \lesssim \left( \sum\limits_{K\in \mathcal{T}_{h}}h_{K}\,\|\textbf{G}\textbf{u}_K\|^2 \right) ^{\frac{1}{2}}\lesssim \left( \sum\limits_{K\in \mathcal{T}_{h}} |u_{h}|^{2}_{1,K}\right) ^{\frac{1}{2}}= |u_{h}|_{1}.
\end{equation}

For  the second term on the right hand of (\ref{lemma_eq:continuity_1}), let $ \varphi=\nabla u_{h} $. Since $\mathcal{T}_{h}$ is a regular partition (\ref{shape-regular}), it is obvious that
\begin{equation*}
\int\!\!\!\int_{s^{*}}|\varphi|^{2}\, \mathrm{d} S \lesssim h_{K}^{2}\int\!\!\!\int_{\hat{s}^{*}}|\hat{\varphi}|^{2}\, \mathrm{d} \hat{S}.
\end{equation*}
According to the trace theorem, we have
\begin{equation*}
\sum\limits_{\hat{s}^{*}\in \mathcal{S}_{\hat{K}}^{*}}\int\!\!\!\int_{\hat{s}^{*}}|\hat{\varphi}|^{2}\, \mathrm\mathrm{d} \hat{S}  \lesssim  \|\hat{\varphi}\|^{2}_{1,\hat{K}}.
\end{equation*}
The Sobolev norms and semi-norms of $ \varphi $ satisfy \cite{LiRH2000}:
\begin{equation*}
\|\hat{\varphi}\|^{2}_{0,\hat{K}} \lesssim h_{K}^{-3}\|\varphi\|^{2}_{0,K},\quad |\hat{\varphi}|^{2}_{1,\hat{K}} \lesssim h_{K}^{-1}|\varphi|^{2}_{1,K},
\end{equation*}
which leads to
\begin{equation*}
\sum\limits_{s^{*}\in \mathcal{S}_{K}^{*}}\int\!\!\!\int_{s^{*}}\varphi^{2}\, \mathrm\mathrm{d} S \lesssim h_{K}^{-1}\|\varphi\|^{2}_{0,K}+h_{K}|\varphi|^{2}_{1,K}.
\end{equation*}
Then 
\begin{equation*}
\begin{split}
\left( \sum\limits_{K\in \mathcal{T}_{h}}\sum\limits_{s^{*}\in \mathcal{S}_{K}^{*}}\frac{|s^{*}|}{h_{K}}\int\!\!\!\int_{s^{*}}(\nabla u)^{2}\, \mathrm\mathrm{d} S \right) ^{\frac{1}{2}} &\lesssim \left( \sum\limits_{K\in\mathcal{T}_{h}}h_{K}(h_{K}^{-1}\|\nabla u\|^{2}_{0,K}+h_{K}|\nabla u|^{2}_{1,K}) \right) ^{\frac{1}{2}}\\
&\lesssim \left( \sum\limits_{K\in\mathcal{T}_{h}}(|u|^{2}_{1,K}+h_{K}^{2}|u
|^{2}_{2,K}) \right) ^{\frac{1}{2}},
\end{split}
\end{equation*}
which together with (\ref{lemma_eq:continuity_1}) and (\ref{Pip_uh}) completes the proof. 
 \end{proof}

Based on Theorem~\ref{theorem:stability}, we give the stability for variable $ \kappa(x_1,x_2,x_3) $.
\begin{lemma}\label{lemma:uniformly ellipitic}
Under the same conditions of Theorem~\ref{theorem:stability} and more generally, assuming that $ \kappa $ is piecewise $ W^{1,\infty} $ over $ \mathcal{T}_{h} $, 
then the bilinear form $ a_{h}(\cdot,\Pi_{\lambda}^{*}\cdot) $  is uniformly ellipitic for sufficiently small $  h>0 $, i.e.,
\begin{equation}\label{stability}
a_{h}(u_{h},\Pi_{\lambda}^{*}u_{h})\gtrsim |u_{h}|_{1}^{2},   \qquad  \forall  u_{h}\in \mathit{U}_{h}.
\end{equation}
\end{lemma}
\begin{proof}
Let 
\begin{equation*}
\overline{a}_{h}(u_{h},\Pi_{\lambda}^{*}u_{h})=-\sum\limits_{K\in\mathcal{T}_{h}}\sum\limits_{K^{*}\in\mathcal{T}_{h}^{*}}  {\int\!\!\!\int_{\partial K^{*}\cap K}\overline{\kappa}\, \nabla u_{h}\cdot \textbf{n}\, \Pi_{\lambda}^{*}u_{h}\mathrm{d} S},
\end{equation*}
where $ \overline{\kappa} $ is a piecewise constant function that $\overline{\kappa}\vert_{K}=|K|^{-1}\int\!\!\!\int\!\!\!\int_{K}\kappa\,\mathrm{d} x_1  \mathrm{d} x_2 \mathrm{d} x_3 $  $\forall K\in \mathcal{T}_{h} $. 
Theorem~\ref{theorem:stability} indicates that
\begin{equation*}
\overline{a}_{h}(u_{h},\Pi_{\lambda}^{*}u_{h})\gtrsim |u_{h}|_{1}^{2}   \qquad \forall u_{h}\in \mathit{U}_{h}.
\end{equation*}
Similar to the proof of the continuity (\ref{continue}) and by the inverse estimate, we obtain
\begin{equation*}
\begin{split}
|a_{h}(u_{h},\Pi_{\lambda}^{*}u_{h})-\overline{a}_{h}(u_{h},\Pi_{\lambda}^{*}u_{h})|&= \left|\sum\limits_{K\in \mathcal{T}_{h}}\sum\limits_{s^{*}\in \mathcal{S}_{K}^{*}}\int\!\!\!\int_{s^{*}}(\kappa-\overline{\kappa})\nabla u_{h}\cdot \textbf{n}\, [\Pi_{\lambda}^{*}u_{h}]_{s^{*}}\mathrm{d} S\right|\\
&\lesssim \|\kappa-\overline{\kappa}\|_{0,\infty}(|u_{h}|_{1}+h|u_{h}|_{2,h})|u_{h}|_{1}\lesssim h|u_{h}|^2_{1},  \quad \forall u_{h}\in \mathit{U}_{h}.
\end{split}
\end{equation*}
Therefore, 
\begin{equation*}
a_{h}(u_{h},\Pi_{\lambda}^{*}u_{h}) \geq\overline{a}_{h}(u_{h},\Pi_{\lambda}^{*}u_{h})-|a_{h}(u_{h},\Pi_{\lambda}^{*}u_{h})-\overline{a}_{h}(u_{h},\Pi_{\lambda}^{*}u_{h})|\gtrsim|u_{h}|^2_{1},  \quad \forall u_{h}\in \mathit{U}_{h},
\end{equation*}
when $ h $ is small enough. 
 \end{proof}

Then, the $ H^{1} $ error estimate follows the stability.
\begin{theorem}[$ H^{1} $ error estimate]
Suppose that $u\in \mathit{H}_{0}^{1}(\Omega)\cap\mathit{H}^{3}(\Omega) $ is the solution of (\ref{eq:epselliptic}). If the conditions of Lemma~\ref{lemma:uniformly ellipitic} are satisfied, then (\ref{eq:FVM}) has a unique solution $ u_{h}\in \mathit{U}_{h} $, and 
\begin{equation}\label{H1}
|u-u_{h}|_{1} \lesssim h^{2}|u|_{3}.
\end{equation}
\end{theorem}
\begin{proof}
Lemma~\ref{lemma:uniformly ellipitic} indicates
\begin{equation*}
a_h(u_{h},\Pi_{\lambda}^{*}u_{h})\geq0,\quad \forall  u_{h}\in \mathit{U}_{h},
\end{equation*}
and the equality holds if and only if  $ u_{h}=0 $ which verifies the existence and uniqueness of $ u_{h}$.

Apparently, $ a_{h}(u,v_{h})=(f,v_{h})\,\,\,\forall v_{h}\in \mathit{V}_{h}$, which together with (\ref{eq:FVM}) leads to the orthogonality
\begin{equation} \label{orthogonality:a_h}
a_{h}(u_{h}-u,v_{h})=0,\quad \forall v_{h}\in \mathit{V}_{h}.
\end{equation}
Let $ u_{I}\in \mathit{U}_{h} $ be the standard  Lagrange quadratic interpolation of u over $ \mathcal{T}_{h} $. Then from (\ref{continue}), (\ref{stability}) and (\ref{orthogonality:a_h}), we have
\begin{equation*} 
|u_{h}-u_{I}|^{2}_{1} \lesssim a_{h}(u_{h}-u_{I},\Pi_{\lambda}^{*}(u_{h}-u_{I}))=a_{h}(u-u_{I},\Pi_{\lambda}^{*}(u_{h}-u_{I}))\lesssim (|u-u_{I}|_{1}+h|u-u_{I}|_{2,h})|u_{h}-u_{I}|_{1}.
\end{equation*}
Eliminating $ |u_{h}-u_{I}|_{1} $ and by the standard interpolation error estimate, then
\begin{equation*}
|u_{h}-u_{I}|_{1} \lesssim |u-u_{I}|_{1}+h|u-u_{I}|_{2,h} \lesssim h^{2}|u|_{3}.
\end{equation*}
Together with $|u-u_{h}|_{1}\leq |u-u_{I}|_{1}+|u_{h}-u_{I}|_{1}  $, we obtain (\ref{H1}). 
 \end{proof}

We present optimal $ L^{2} $ error estimate in the following Theorem~\ref{L2}, which benefits from \cite{WangX2016}. 
\begin{theorem}[$ L^{2} $ error estimate]\label{L2}
Suppose that $u\in H_{0}^{1}(\Omega)\cap H^{4}(\Omega)$ is the solution of $(\ref{eq:epselliptic})$. If the conditions of Lemma~\ref{lemma:uniformly ellipitic} and the orthogonal condition on the volume (\ref{orthogonal_space}) are satisfied, then 
\begin{equation} \label{eq:thm:L2}
\| u-u_{h}\|_{0}\lesssim h^{3}\|u\|_{4}.
\end{equation}
\end{theorem}
\begin{proof}
Consider an auxiliary problem: given $ g\in L^{2}(\Omega) $, find $ \omega_{g}\in H_{0}^{1}(\Omega) $ such that
\begin{equation} \label{eq:dual}
a(v,\omega_{g})=(g,v),\quad \forall v\in H^{1}_{0}(\Omega),
\end{equation}
where
$ a(v,\omega_{g})=\int\!\!\!\int\!\!\!\int_{\Omega} (\kappa\nabla v)\cdot \nabla \omega_{g} \,\mathrm{d} x_1  \mathrm{d} x_2 \mathrm{d} x_3 $.
It is well known that this problem is regular, i.e., it attains a unique solution $\omega_{g}\in H_{0}^{1}(\Omega)\cap H^{2}(\Omega)$ satifying $\|\omega_{g}\|_{2} \lesssim\| g\|_{0}$.

Let $ v=u-u_{h} $ in (\ref{eq:dual}). By the orthogonality (\ref{orthogonality:a_h}) and the Green's formula, we have
\begin{align}
(g,u-u_{h})=a(u-u_{h},w_{g})&=a(u-u_{h},w_{g}-\Pi_{h}^{1}w_{g})+a(u-u_{h},\Pi_{h}^{1}w_{g})-a_h(u-u_{h},\Pi_{\lambda}^{\ast}(\Pi_{h}^{1}w_{g}))\nonumber\\
&=E_{1}+E_{2}+E_{3},\label{E1+E2+E3}
\end{align}
and
\begin{equation*}
\begin{split}
E_{1 } &=a(u-u_{h},w_{g}-\Pi_{h}^{1}w_{g}),\\
E_{2}&=\sum_{K \in \mathcal{T}_{h}} \iiint_{K}-\nabla\cdot\big(\kappa\nabla (u-u_{h})\big)\, \big(\Pi_{h}^{1}w_{g}-\Pi_{\lambda}^{\ast}(\Pi_{h}^{1}w_{g})\big) \, \mathrm{d} x_1 \mathrm{d} x_2 \mathrm{d} x_3,\\
E_{3}&=\sum_{K \in \mathcal{T}_{h}}\iint_{\partial K}(\kappa\nabla (u-u_{h}))\cdot  \textbf{n}\, (\Pi_{h}^{1}w_{g}-\Pi_{\lambda}^{\ast}(\Pi_{h}^{1}w_{g}))  \,\mathrm{d}S,
\end{split}
\end{equation*}	
where $\Pi_{h}^{1}$ is  the piecewise linear interpolation projection over $ \mathcal{T}_{h} $. For  convenience of writing, let
$M_{1}^{\ast}w_{g}=\Pi_{h}^{1}w_{g}-\Pi_{\lambda}^{\ast}(\Pi_{h}^{1}w_{g})$.
Consider $ E_{2}=E_{21}+E_{22} $ and $ E_{3}=E_{31}+E_{32} $ with
\begin{equation*}
\begin{split}
E_{21}&=\sum_{K \in \mathcal{T}_{h}}^{}  \iiint_{K}-\nabla\cdot\big((\kappa-\overline{\kappa})\nabla (u-u_{h})\big)\, M_{1}^{\ast}w_{g} \, \mathrm{d} x_1 \mathrm{d} x_2 \mathrm{d} x_3,\\
E_{22}&=\sum_{K \in \mathcal{T}_{h}}^{}  \iiint_{K}-\nabla\cdot\big(\overline{\kappa}\,\nabla (u-u_{h})\big) \, M_{1}^{\ast}w_{g}  \, \mathrm{d} x_1 \mathrm{d} x_2 \mathrm{d} x_3,\\
E_{31}&=\sum_{K \in \mathcal{T}_{h}}\iint_{\partial K}\big((\kappa-\tilde{\kappa})\nabla (u-u_{h})\big)\cdot  \textbf{n}\, M_{1}^{\ast}w_{g} \,\mathrm{d}S,\\
E_{32}&=\sum_{K \in \mathcal{T}_{h}}\iint_{\partial K}\big(\tilde{\kappa}\,\nabla (u-u_{h})\big)\cdot  \textbf{n}\, M_{1}^{\ast}w_{g}  \,\mathrm{d}S,
\end{split}
\end{equation*}
where $ \overline{\kappa} $ and $ \tilde{\kappa} $ are two piecewise constant functions that for $ K\in \mathcal{T}_{h} $, $\overline{\kappa}\vert_{K}=|K|^{-1}\int\!\!\!\int\!\!\!\int_{K}\kappa\,\mathrm{d} x_1  \mathrm{d} x_2 \mathrm{d} x_3 $, and $\tilde{\kappa}\vert_{T_{i}}=|T_{i}|^{-1}\int\!\!\!\int\!\!\!\int_{T_{i}}\kappa\,\mathrm{d} S\,\,\,\forall T_{i}\in \partial K$.

It follows from the proof of Theorem~5.3 in \cite{WangX2016} that
\begin{equation}\label{E1,E21,E31}
\begin{split}
|E_{1}|\lesssim h^{3} |u|_{3}|w_{g}|_{2},\,\,\,|E_{21}|\lesssim h^{3} |u|_{3}|w_{g}|_{1},\,\,\,|E_{31}|\lesssim h^{3} |u|_{3}|w_{g}|_{1}.
\end{split}
\end{equation}
The orthogonal condition on the volume (\ref{orthogonal_space}) is used to estimate $ E_{22} $. Noticing that $\frac{\partial^{2} u_{h}}{\partial x_{i}^{2}}-\Pi_{h}^{1}\frac{\partial^{2} u}{\partial x_{i}^{2}}$  $(i \in \mathcal{Z}_{3}^{(1)})$ restricted on $ K $ are linear functions, we obtain
\begin{align}
|E_{22}|\leq&\big| \sum_{K \in \mathcal{T}_{h}}\sum_{i \in \mathcal{Z}_{3}^{(1)}}  \overline{\kappa}\,\iiint_{K}\frac{\partial^{2} (u-u_{h})}{\partial x_{i}^{2}} \,M_{1}^{\ast}w_{g} \, \mathrm{d} x_1 \mathrm{d} x_2 \mathrm{d} x_{3}\big|\nonumber\\
\lesssim&\sum_{K \in \mathcal{T}_{h}}\|\frac{\partial^{2} u}{\partial x_{i}^{2}}-\Pi_{h}^{1}\frac{\partial^{2} u}{\partial x_{i}^{2}}\|_{0,K}\| M_{1}^{\ast}w_{g} \|_{0,K}
\lesssim \,h^{3}|u|_{4}|w_{g}|_{1}.\label{E22}
\end{align}
The orthogonal condition on the surface (\ref{orthogonal_plane}) is used to estimate $ E_{32} $. By the boundary condition $ w_{g}\vert_{\partial \Omega}=0 $, we have
\begin{equation*}
\sum_{K \in \mathcal{T}_{h}}\iint_{\partial K}(\tilde{\kappa}\,\nabla u)\cdot  \textbf{n}\, M_{1}^{\ast}w_{g}  \,\mathrm{d}S=\iint_{\partial \Omega}(\tilde{\kappa}\,\nabla u)\cdot  \textbf{n}\, M_{1}^{\ast}w_{g}  \,\mathrm{d}S=0,
\end{equation*}
which together with (\ref{orthogonal_plane}) yields
\begin{equation}\label{E32}
E_{32}=\sum_{K \in \mathcal{T}_{h}}\iint_{\partial K}(\tilde{\kappa}\,\nabla u)\cdot  \textbf{n}\, M_{1}^{\ast}w_{g}  \,\mathrm{d}S-\sum_{K \in \mathcal{T}_{h}}\iint_{\partial K}(\tilde{\kappa}\,\nabla u_{h})\cdot  \textbf{n}\, M_{1}^{\ast}w_{g}  \,\mathrm{d}S=0.
\end{equation}

Let $ g=u-u_{h} $ in (\ref{E1+E2+E3}). With the estimates (\ref{E1,E21,E31}), (\ref{E22}), (\ref{E32}), and the regularity $\|\omega_{g}\|_{2} \lesssim\| g\|_{0}$, we obatin the $ L^{2} $ error estimate (\ref{eq:thm:L2}). 
 \end{proof}
The above optimal $ L^{2} $ error estimate
strongly depends on the orthogonal conditions on the surface (\ref{orthogonal_plane}) and volume (\ref{orthogonal_space}).
We would like to point out that  the orthogonal conditions are sufficient conditions in the proof of $ L^{2} $ error estimate, however,
a large number of experiments indicates that they are necessary to achieve optimal  convergence rate in $ L^{2} $ norm, which are shown in Section~\ref{Section:5}.

\section{Numerical experiments}\label{Section:5}
\begin{table}[ht!]	
\renewcommand\arraystretch{1.8}
\centering
{\small 
\begin{tabular}{|c|c|c|c|c|} \hline
Scheme&$\alpha$ &$\beta$ & $\gamma$ &  $ v^{*}\big(\alpha,\beta,\gamma,1/(1-3\alpha\beta)\big) $\\ \hline
QFVS-1&$1/10$&$ 14/15-2\sqrt{66}/45 $&$ 0.050667311760225 $ &  $ 20.5^{\circ} $\\ \hline
QFVS-2&$1/2-\sqrt{3}/6$&$ 2/3+\sqrt{3}/9-\sqrt{21+6\sqrt{3}}/9 $&$ 0.052883196779577 $& $ 17.0^{\circ} $\\  \hline
QFVS-3&$2/5$&$ 11/15-2\sqrt{714}/45 $&$ 0.051085694878555  $ & $ 16.7^{\circ} $\\  \hline
QFVS-4&$1/2-\sqrt{3}/6$&$ 2/3+\sqrt{3}/9-\sqrt{21+6\sqrt{3}}/9 $&$ 1/4 $ & $ 18.2^{\circ} $\\  \hline
\end{tabular}
}
\caption{Four quadratic FVM schemes.}
\label{tab1}
\end{table}
Here in Table~\ref{tab1}, we present four quadratic FVM schemes, which will be used in the numerical experiments. The first three schemes satisfy the orthogonal conditions on the surface (\ref{orthogonal_plane_para}) and volume (\ref{orthogonal_space_para}), while the last scheme only satisfies the orthogonal condition on the surface (\ref{orthogonal_plane_para}). 
\begin{figure}[ht!]
\centering
\subfigure[Scheme QFVS-1.]{
\includegraphics[width=150pt]{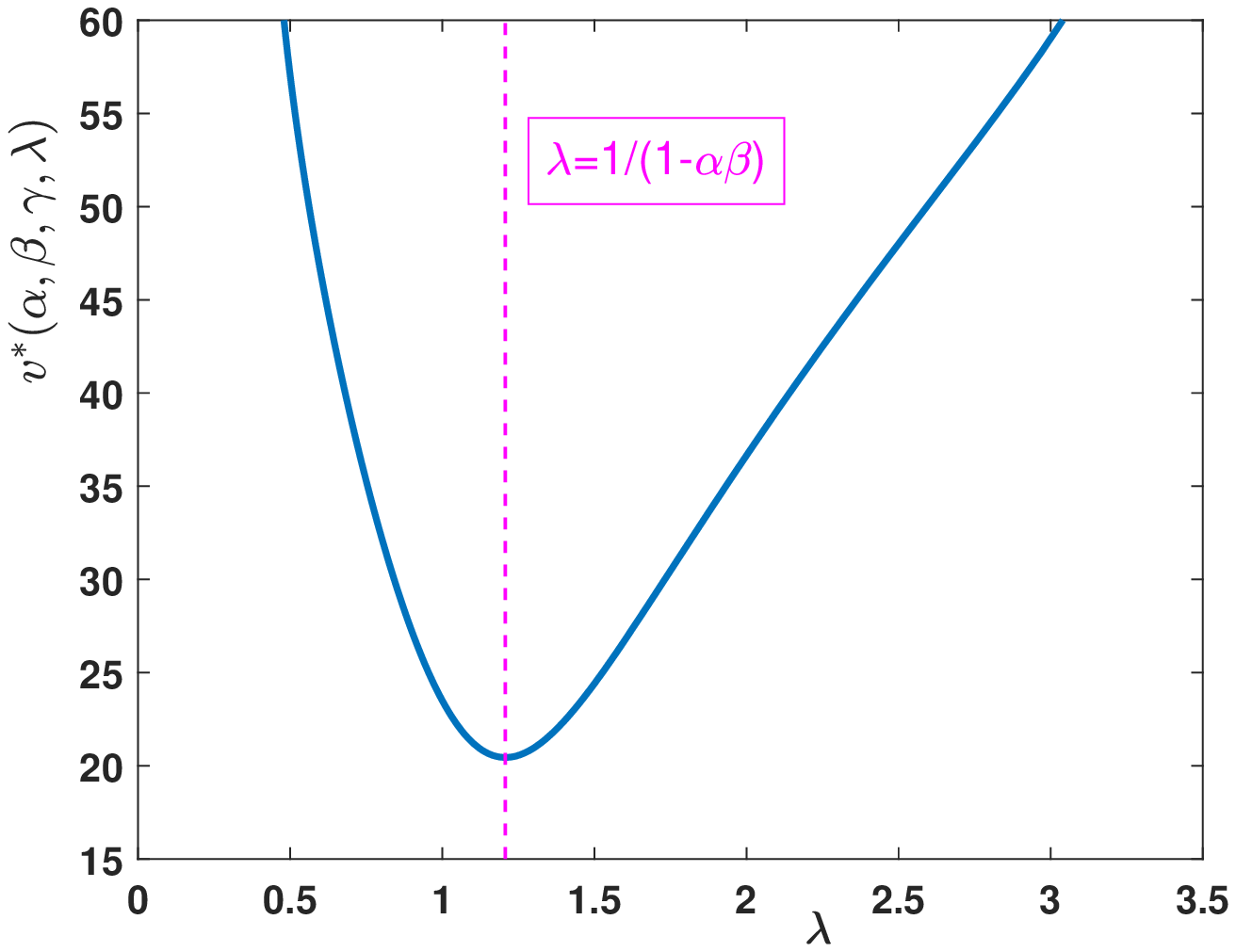}
}
\subfigure[Scheme QFVS-2.]{
\includegraphics[width=150pt]{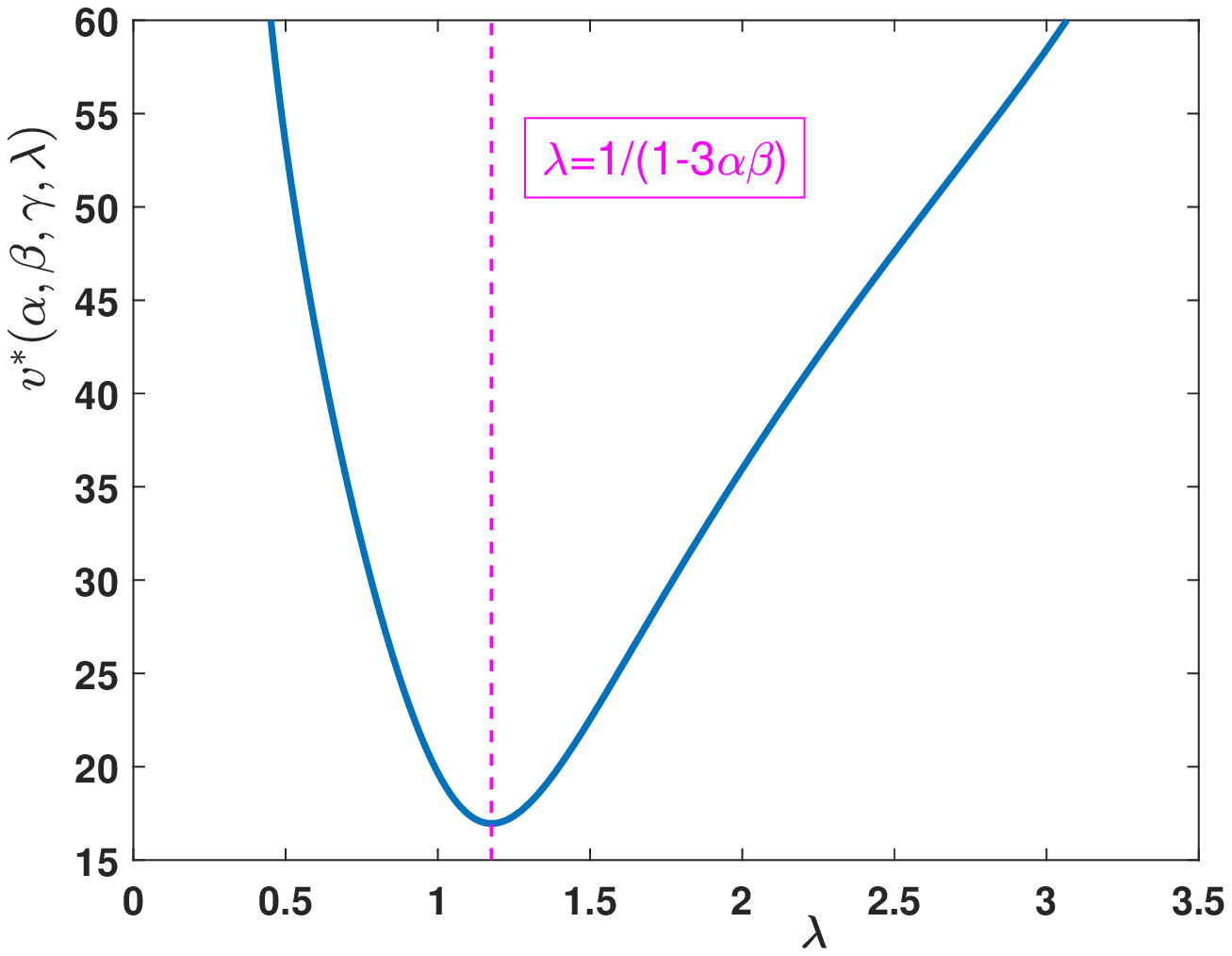}
}
\subfigure[Scheme QFVS-3.]{
\includegraphics[width=150pt]{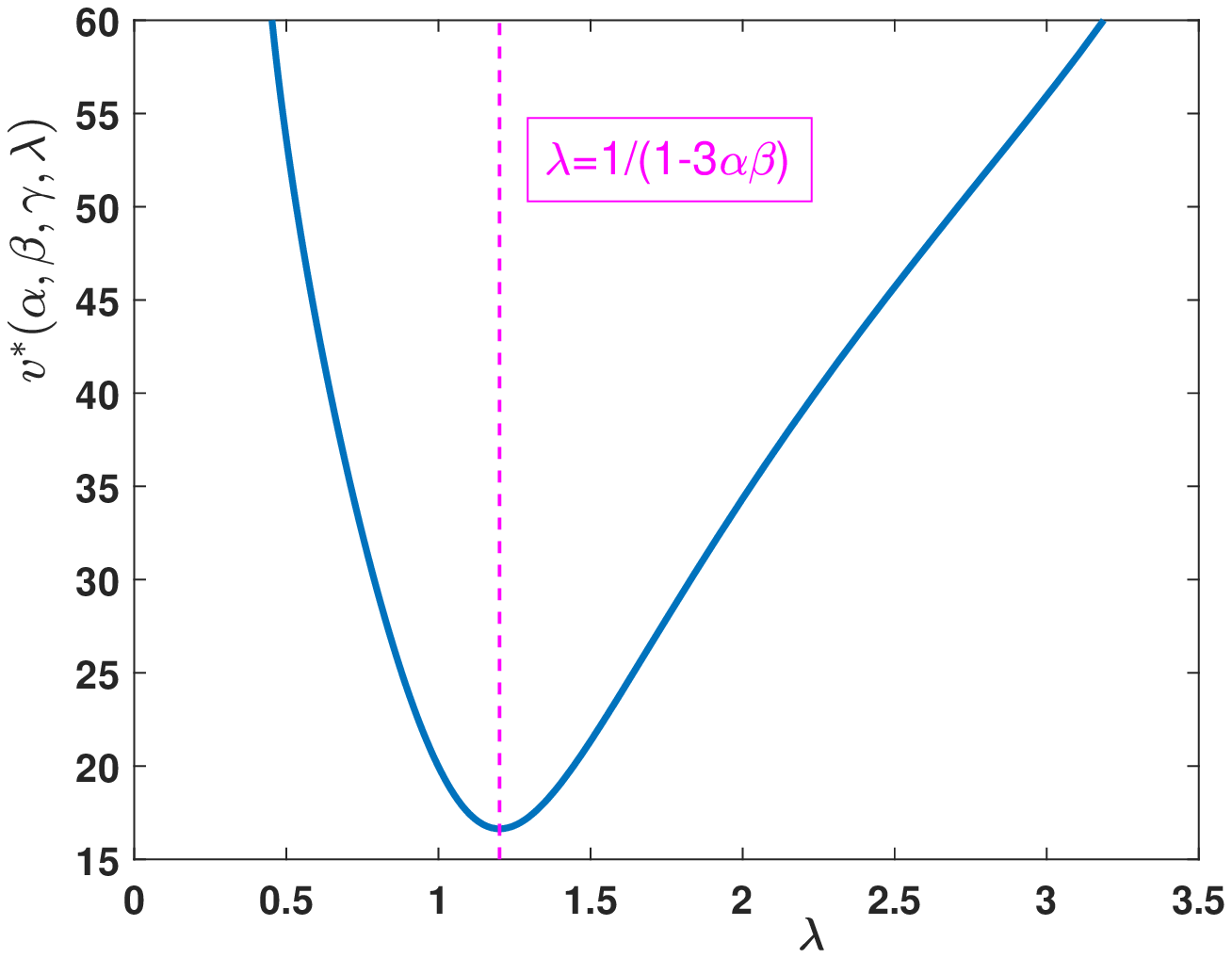}
}
\subfigure[Scheme QFVS-4.]{
\includegraphics[width=150pt]{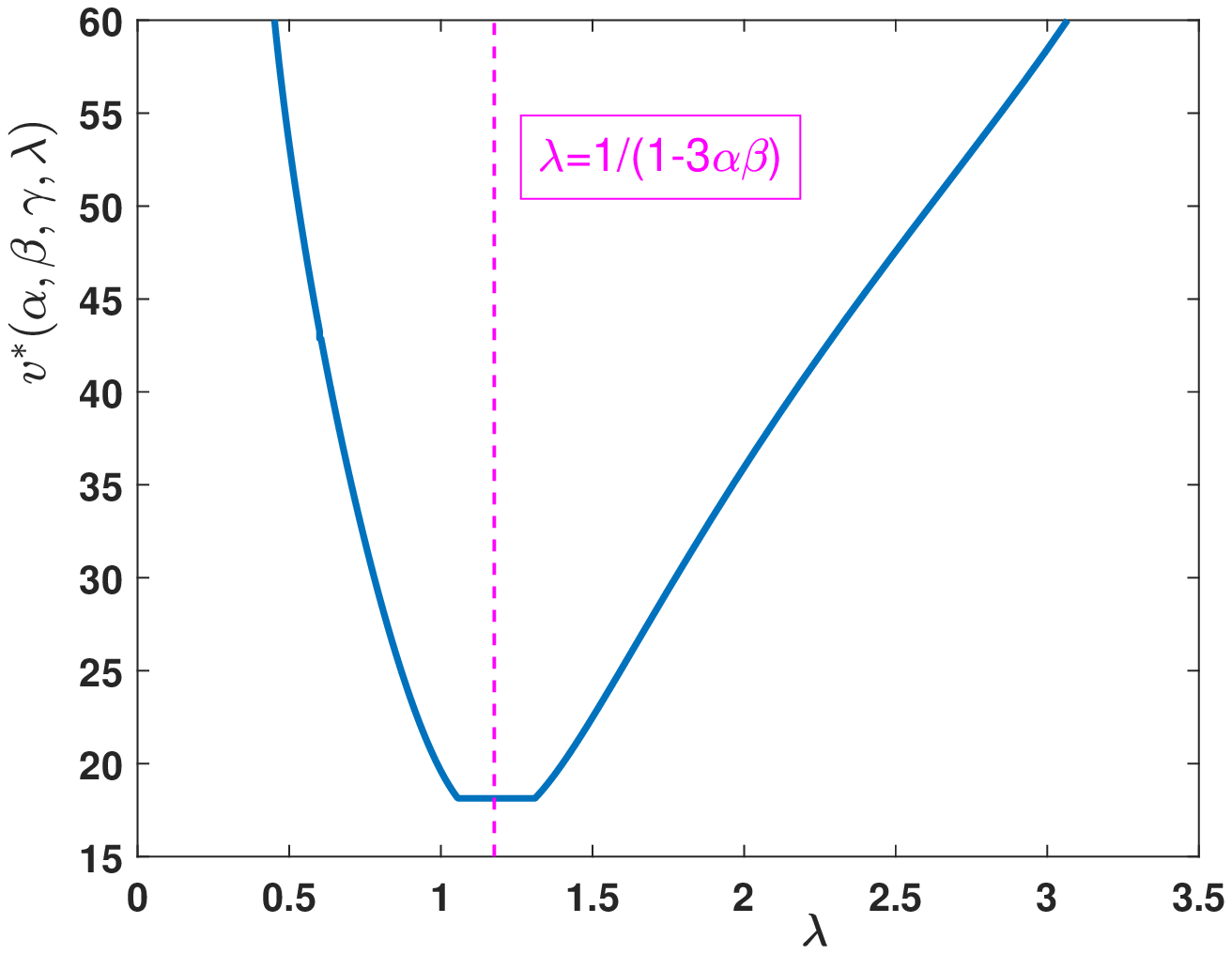}
}
\caption{The relationships between  $ v^{*}(\alpha,\beta,\gamma,\lambda) $ and  $ \lambda $.}
\label{fig:graph}
\end{figure}   
\begin{example}According to \textbf{Algorithm}~\ref{searching v}, we show in Fig.~\ref{fig:graph} the relationships between  $ v^{*}(\alpha,\beta,\gamma,\lambda) $ and  $ \lambda $  for the schemes in Table~\ref{tab1}. From the four figures, $ v^{*}(\alpha,\beta,\gamma,\lambda) $ reaches its minimum value at $ \lambda=1/(1-3\alpha\beta) $. The values of $ v^{*}\big(\alpha,\beta,\gamma,1/(1-3\alpha\beta)\big) $  are given in  the last column of Table~\ref{tab1}.
\end{example}
%
\begin{figure}[ht!]
\centering
\begin{minipage}[t]{.8\textwidth}
\centering
\includegraphics[width=80pt]{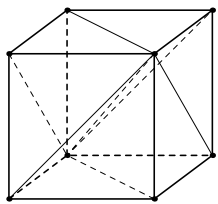}
\caption{A division of a cube.}
\label{fig:partition}
\end{minipage}
\end{figure}   
\begin{table}[ht!]
\centering
\begin{tabular}{|l|l|l|l|l|l|l|l|l|} \hline
\multirow{2}{*}{N}&
\multicolumn{2}{c|}{QFVS-1}&
\multicolumn{2}{c|}{QFVS-2}&
\multicolumn{2}{c|}{QFVS-3}&
\multicolumn{2}{c|}{QFVS-4} \\ 
\cline{2-9}
&$|u-u_h|_{1}$ & Order&$|u-u_h|_{1}$ &Order&$|u-u_h|_{1}$ &Order
&$|u-u_h|_{1}$ &Order\\ \hline
5 &1.1526e-01& $\setminus$ &1.1454e-01& $\setminus$&1.1478e-01& $\setminus$&1.1294e-01&$\setminus$  \\ \hline
15&1.3110e-02&1.9787&1.3098e-02&1.9738&1.3102e-02&1.9755 &1.3105e-02&1.9605\\ \hline

25&4.7300e-03&1.9956&4.7285e-03&1.9946&4.7289e-03&1.9949 &4.7380e-03&1.9917\\ \hline

35&2.4148e-03&1.9981&2.4144e-03&1.9976&2.4145e-03&1.9978 &2.4203e-03&1.9964\\ \hline

45&1.4612e-03&1.9989&1.4611e-03&1.9987&1.4611e-03&1.9988 &1.4649e-03&1.9980\\ \hline
\end{tabular}
\caption{$H^{1}$  convergence rate of Example~\ref{example}.}
\label{tab:H1(example_1)}
\end{table}
\begin{table}[ht!]
\centering
\begin{tabular}{|l|l|l|l|l|l|l|l|l|} \hline
\multirow{2}{*}{N}&
\multicolumn{2}{c|}{QFVS-1}&
\multicolumn{2}{c|}{QFVS-2}&
\multicolumn{2}{c|}{QFVS-3}&
\multicolumn{2}{c|}{QFVS-4}\\ 
\cline{2-9}
&$\|u-u_h\|_{0}$ & Order&$\|u-u_h\|_{0}$ &Order&$\|u-u_h\|_{0}$ &Order&$\|u-u_h\|_{0}$ &Order\\ \hline
5&2.7360e-03& $\setminus$ &2.7238e-03& $\setminus$&2.7202e-03& $\setminus$&2.4588e-03&$\setminus$  \\ \hline
15&9.7046e-05&3.0394&9.6983e-05&3.0359&9.6969e-05&3.0348&1.8766e-04&2.3418\\ \hline

25&2.0875e-05&3.0081&2.0870e-05&3.0073&2.0869e-05&3.0072&6.5523e-05&2.0599\\ \hline

35&7.5988e-06&3.0034&7.5978e-06&3.0030&7.5976e-06&3.0030&3.3157e-05&2.0243\\ \hline

45&3.5736e-06&3.0019&3.5733e-06&3.0017&3.5733e-06&3.0016&1.9992e-05&2.0132\\ \hline
\end{tabular}
\caption{$L^{2}$  convergence rate  of Example~\ref{example}.}
\label{tab:L2(example_1)}
\end{table}
\begin{example}\label{example}
Consider the model problem (\ref{eq:epselliptic}) with $\Omega=[0,1]^{3}$.  The primary mesh $\mathcal{T}_{h}$ is constructed by dividing $\Omega$ into $ N^{3} $ cubes first, then each cube is divided into six tetrahedrons, see Fig.~\ref{fig:partition}. By definition (\ref{definition:vertex angles}), the uniform minimum  V-angle is $ \min\left\lbrace \theta_K,\,K\in\mathcal{T}_{h}\right\rbrace=\arctan(\sqrt{2}/2)+45^{\circ}-\arctan\sqrt{2}\approx25.529^{\circ} $. Then for the four schemes in Table~\ref{tab1}, it is observed from the last column of Table~\ref{tab1} that  the primary mesh $\mathcal{T}_{h}$ satisfies the minimum V-angle condition (\ref{restriction_3}).

We apply the four schemes in Table~\ref{tab1} to equation (\ref{eq:epselliptic}) with the  coefficient $ \kappa(x_1,x_2,x_3)=e^{x_1+2x_2+3x_3} $,
and $ f $ is chosen so that the exact solution is $ u=\sin(\pi x_1)\sin(\pi x_2)\sin(\pi x_3) $. 
See Table~\ref{tab:H1(example_1)}, all of the four schemes  possess  optimal $ H^1 $ convergence rate. And Table~\ref{tab:L2(example_1)} shows that the first three schemes  have optimal $ L^2 $ convergence rate, while the last scheme doesn't. These numerical results demonstrate our theoretical results.
\end{example}
\begin{table}[ht!]
\centering
\begin{tabular}{|l|l|l|l|l|l|l|l|l|} \hline
\multirow{2}{*}{N}&
\multicolumn{2}{c|}{QFVS-1}&
\multicolumn{2}{c|}{QFVS-2}&
\multicolumn{2}{c|}{QFVS-3}&
\multicolumn{2}{c|}{QFVS-4}\\ 
\cline{2-9}
&$|u-u_h|_{1}$ & Order&$|u-u_h|_{1}$ &Order&$|u-u_h|_{1}$ &Order
&$|u-u_h|_{1}$ &Order\\ \hline
5 &1.2602e-01& $\setminus$ &1.2090e-01& $\setminus$&1.2429e-01& $\setminus$&1.2061e-01&$\setminus$ \, \\ \hline
15&1.4108e-02&1.9931&1.4116e-02&1.9549&1.4200e-02&1.9746&1.4072e-02&1.9555\\ \hline

25&5.0788e-03&2.0001&5.0835e-03&1.9994&5.0786e-03&2.0129&5.0767e-03&1.9958\\ \hline

35&2.5967e-03&1.9937&2.5947e-03&1.9988&2.5977e-03&1.9924&2.5959e-03&1.9934\\ \hline

45&1.5700e-03&2.0023&1.5688e-03&2.0021&1.5732e-03&1.9955&1.5793e-03&1.9774\\ \hline
\end{tabular}
\caption{$H^{1}$  convergence rate  of Example~\ref{example3}.}
\label{tab:H1(example_2)}
\end{table}
\begin{table}[ht!]
\centering
\begin{tabular}{|l|l|l|l|l|l|l|l|l|} \hline
\multirow{2}{*}{N}&
\multicolumn{2}{c|}{QFVS-1}&
\multicolumn{2}{c|}{QFVS-2}&
\multicolumn{2}{c|}{QFVS-3}&
\multicolumn{2}{c|}{QFVS-4}\\ 
\cline{2-9}
&$\|u-u_h\|_{0}$ & Order&$\|u-u_h\|_{0}$ &Order&$\|u-u_h\|_{0}$ &Order&$\|u-u_h\|_{0}$ &Order\\ \hline
5&3.7870e-03& $\setminus$ &3.3785e-03& $\setminus$&3.4369e-03& $\setminus$&3.1009e-03&$\setminus$ \\ \hline
15&1.2821e-04&3.0818&1.2956e-04&2.9683&1.3487e-04&2.9474&2.0669e-04&2.4652\\ \hline

25&2.6984e-05&3.0508&2.7407e-05&3.0409&2.6863e-05&3.1588&6.9489e-05&2.1339\\ \hline

35&9.8073e-05&3.0081&9.8279e-05&3.0480&9.8367e-05&2.9857&3.4767e-05&2.0582\\ \hline

45&4.6334e-06&2.9837&4.6514e-06&2.9766&4.6999e-06&2.9389&2.0862e-05&2.0323\\ \hline
\end{tabular}
\caption{$L^{2}$  convergence rate of Example~\ref{example3}.}
\label{tab:L2(example_2)}
\end{table}
%
%
%
\begin{example}[Random mesh]\label{example3}
To further	illustrate the performance of the proposed quadratic FVM schemes, we
disturb the vertices of the tetrahedrons in Example~\ref{example} with random rate 
$ 0.2/N$ in three directions, in which
\begin{itemize}
\item[i)] the eight vertices of $ \Omega=[0,1]^{3} $ are fixed;
\item[ii)] the vertices on the twelve edges of $ \Omega=[0,1]^{3} $ are repositioned along the edges;
\item[iii)] the vertices on the six faces of $ \Omega=[0,1]^{3} $ are repositioned on the faces in two directions;
\end{itemize}
For the four schemes in Table~\ref{tab1}, we show in Table~\ref{tab:H1(example_2)} and Table~\ref{tab:L2(example_2)} the numerical perfermances, which are consistent with the numerical perfermances in Example~\ref{example}. 
\end{example}
\begin{figure}[ht!]
\centering
\subfigure[The same $ (\alpha,\beta) $ as QFVS-1.]{
\includegraphics[width=150pt]{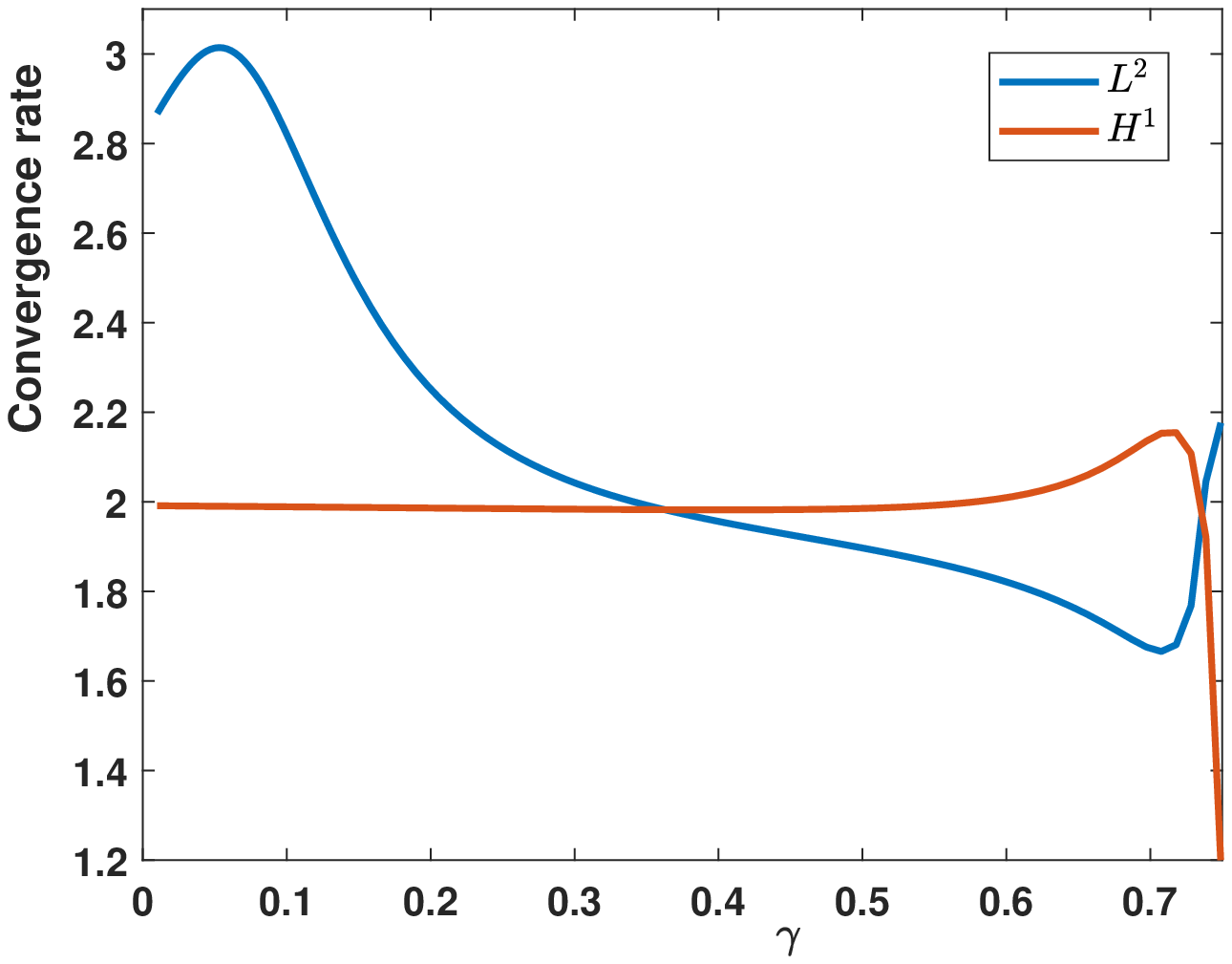}
}
\subfigure[The same $ (\alpha,\beta) $ as QFVS-2.]{
\includegraphics[width=150pt]{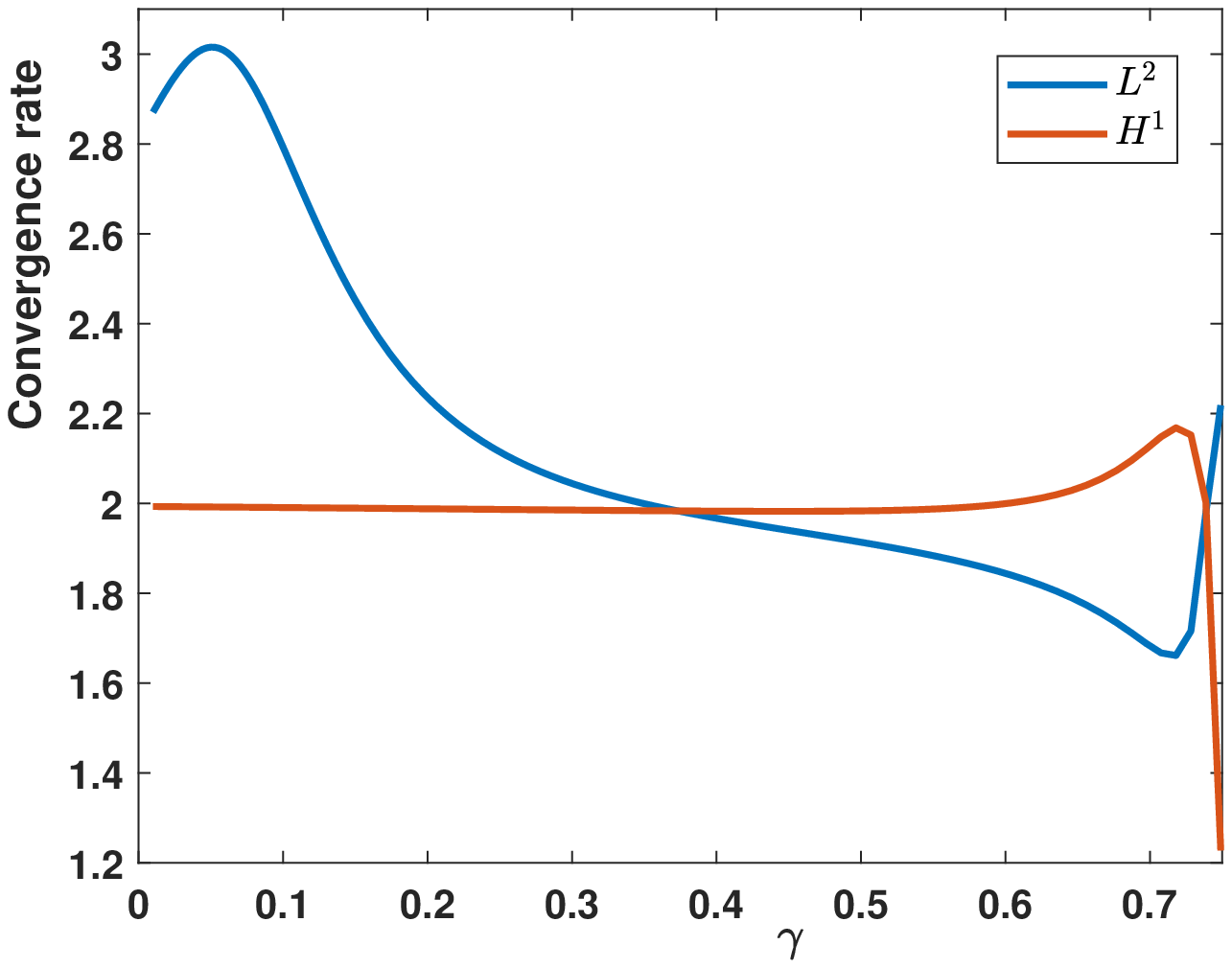}
}
\subfigure[The same $ (\alpha,\beta) $ as QFVS-3.]{
\includegraphics[width=150pt]{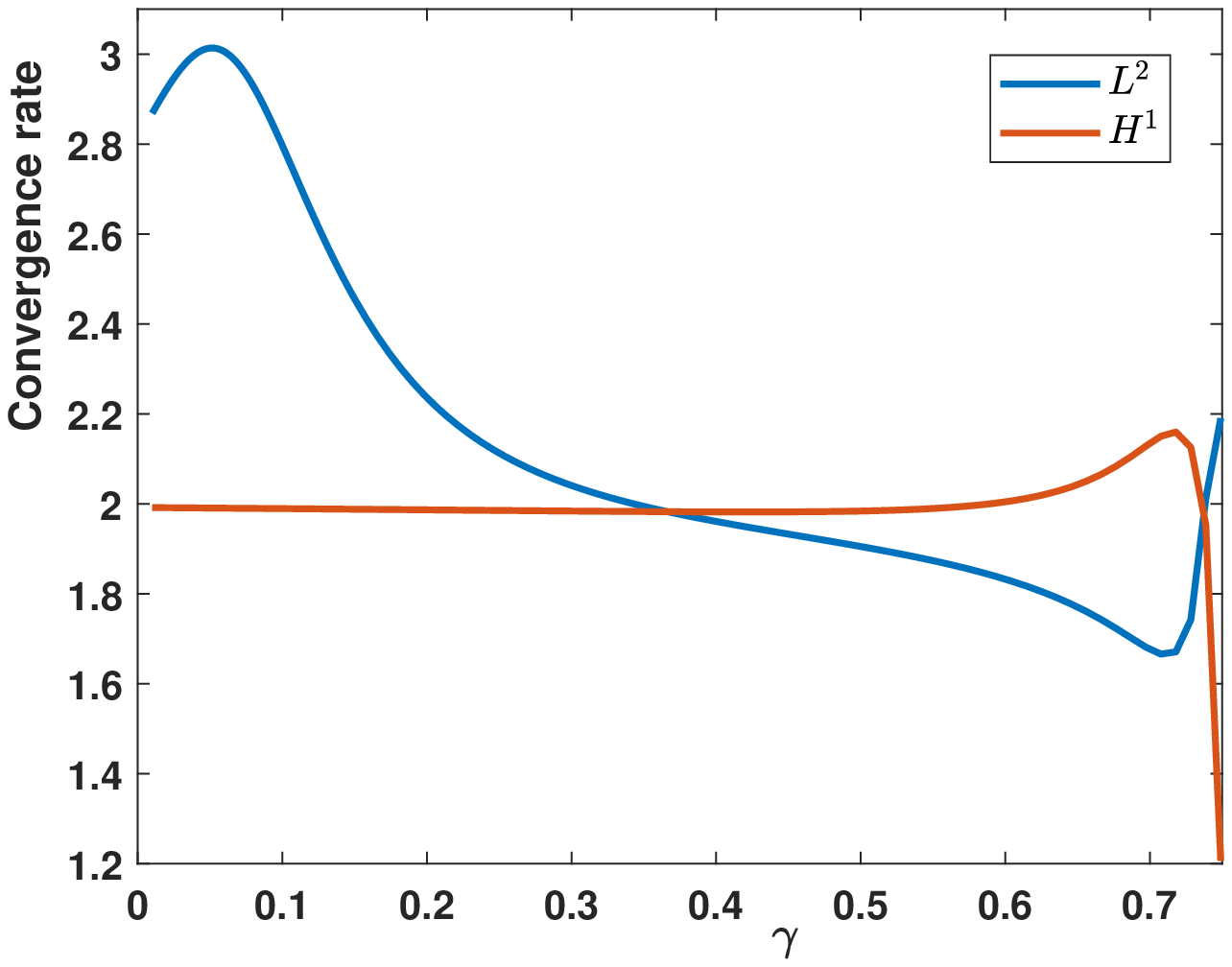}
}
\subfigure[$ (\alpha=0.3, \beta=0.4)$.]{
\includegraphics[width=150pt]{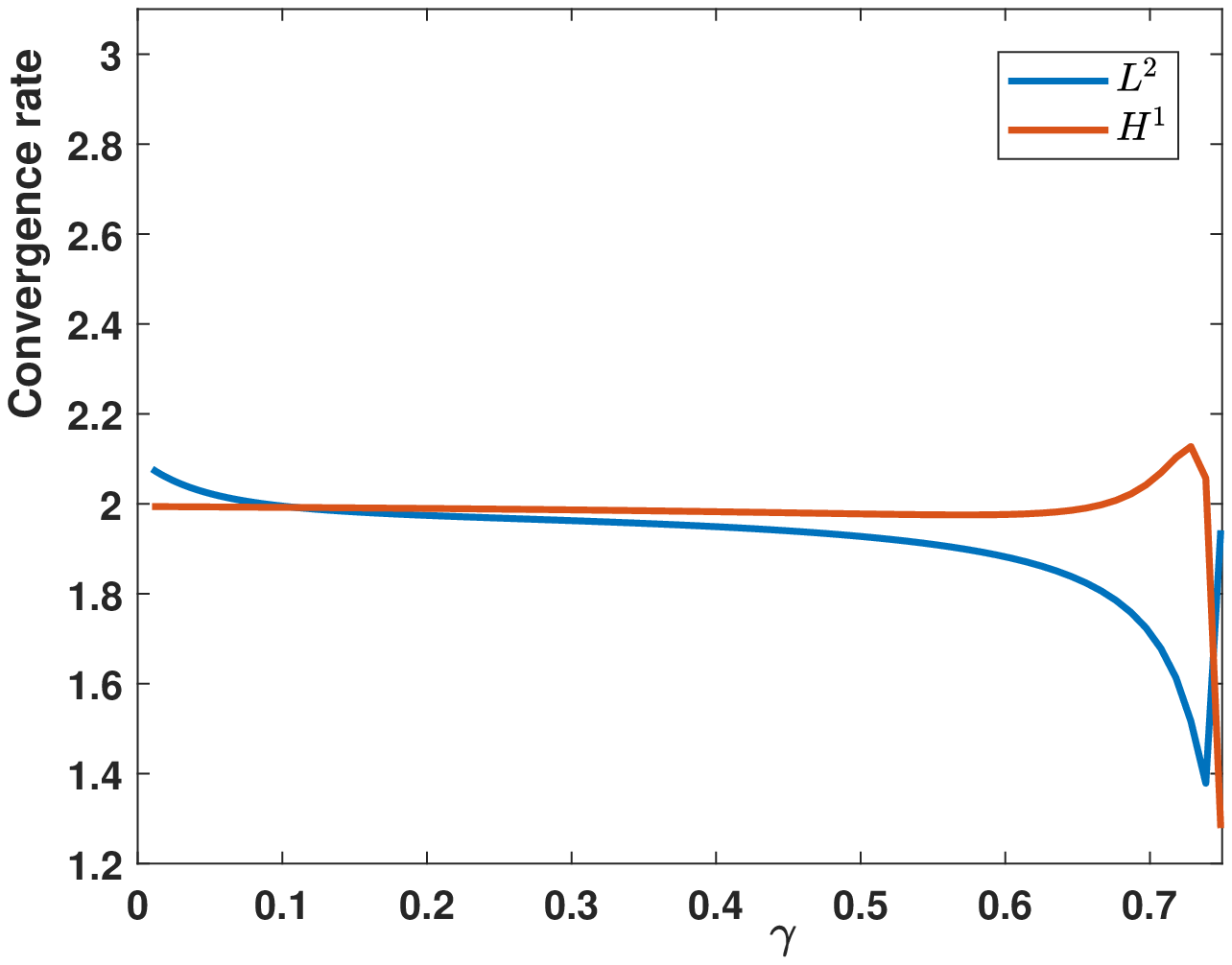}
}
\caption{The relationships between convergence rate and $ \gamma $.}
\label{fig:graph_2}
\end{figure}   
\begin{remark}
Consider the same model problem as Example~\ref{example}. Then for fixed $ (\alpha,\beta) $, Fig.~\ref{fig:graph_2} shows how convergence rates in $ H^{1} $ and $ L^{2} $ norms between $ N=10 $ and $ N=20 $ change with $ \gamma $, where the first three pairs of $ (\alpha,\beta) $ are the same as those in  QFVS-1, QFVS-2 and QFVS-3 respectively, and the last pair is taken as $ (\alpha=0.3, \beta=0.4) $.  In fact, a large number of experiments indicates that the orthogonal conditions are not only sufficient but also necessary to achieve optimal convergence rate in $ L^{2} $ norm.
\end{remark}

\section{Conclusion}\label{Section:6}

In this paper,  we have constructed a family of quadratic FVM schemes on tetrahedral meshes by introducing three parameters $ (\alpha,\beta,\gamma) $ on the dual mesh. Under the proposed orthogonal conditions and minimum V-angle condition, we derive the theoretical analysis. This includes
\begin{itemize}
\item Stability, which is the most important result in this paper;
\item Optimal $ H^{1} $ and $ L^{2} $ error estimates, where the $ L^{2} $ convergence rate strongly depends on the orthogonal conditions.
\end{itemize}
These theoretical results are confirmed by some numerical experiments.

For higher $ r $-order ($r\geq3$) FVMs on tetrahedral meshes,  the stability analysis is much more complex and it needs further study.

\appendix
\setcounter{lemma}{0}
\renewcommand{\thelemma}{B\arabic{lemma}}
\section*{Appendix A: Some symbolic matrices}\label{Section:Appendix A}
\subsection*{A.1 The symbolic matrix $ \textbf{A} $ in (\ref{A_k1})}\label{Apsubsection: A_K}
\begin{equation*}
\textbf{A}=\begin{pmatrix}
\textbf{A}^{(1)}_{_{4\!\times\!4}}&\,4(\textbf{A}^{(2)}_{_{6\!\times\!4}})^{T}\\[0.15cm]
\textbf{A}^{(3)}_{_{6\!\times\!4}}&\,4\textbf{A}^{(4)}_{_{6\!\times\!6}}\\
\end{pmatrix},
\end{equation*}
where
\begin{flalign*}
\begin{split}
&\textbf{A}^{(1)}_{_{4\!\times\!4}}\!\!=\!\!{\scriptsize
\left(\begin{array}{cccc} 
(\!3t_{_1}\!\!-\!2t_{_2}\!)R_{_1} &4t_{_2}R_{_2} \!\!+\!(\!t_{_1}\!\!-\!2t_{_2}\!)r_{_{\!34}}&
4t_{_2}R_{_3} \!\!+\!(\!t_{_1}\!\!-\!2t_{_2}\!)r_{_{\!24}}& 4t_{_2}R_{_4} \!\!+\!(\!t_{_1}\!\!-\!2t_{_2}\!)r_{_{\!23}}\\[0.1cm] 
4t_{_2}R_{_1} \!\!+\!(\!t_{_1}\!\!-\!2t_{_2}\!)r_{_{\!34}}&(\!3t_{_1}\!\!-\!2t_{_2}\!)R_{_2} & 
4t_{_2}R_{_3} \!\!+\! (\!t_{_1}\!\!-\!2t_{_2}\!)r_{_{\!14}}&4t_{_2}R_{_4} \!\!+\! (\!t_{_1}\!\!-\!2t_{_2}\!)r_{_{\!13}}\\[0.1cm] 
4t_{_2}R_{_1} \!\!+\! (\!t_{_1}\!\!-\!2t_{_2}\!)r_{_{\!24}}& 4t_{_2}R_{_2} \!\!+\! (\!t_{_1}\!\!-\!2t_{_2}\!)r_{_{\!14}}&
(\!3t_{_1}\!\!-\!2t_{_2}\!)R_{_3}&4t_{_2}R_{_4}\!\!+\!(\!t_{_1}\!\!-\!2t_{_2}\!)r_{_{\!12}}\\[0.1cm]
4t_{_2}R_{_1} \!\!+\! (\!t_{_1}\!\!-\!2t_{_2}\!)r_{_{\!23}}&4t_{_2}R_{_2} \!\!+\! (\!t_{_1}\!\!-\!2t_{_2}\!)r_{_{\!13}}&
4t_{_2}R_{_3} \!\!+\! (\!t_{_1}\!\!-\!2t_{_2}\!)r_{_{\!12}}&(\!3t_{_1}\!\!-\!2t_{_2}\!)R_{_4}
\end{array}\right)},\\[0.15cm]
&\textbf{A}^{(2)}_{_{6\!\times\!4}}\!\!=\!\!{\scriptsize
\left(\begin{array}{cccc} 
-\!t_{_2}(\!R_{_2}\!\!+\!R_{_3}\!\!-\!r_{_{\!12}}\!\!-\!r_{_{\!13}}\!)&t_{_2}R_{_2}\!\!-\!(\!t_{_1}\!\!+\!t_{_2}\!)r_{_{\!14}}&
t_{_2}R_{_3}\!\!-\!(\!t_{_1}\!\!+\!t_{_2}\!)r_{_{\!14}}&-\!t_{_2}(\!R_{_2}\!\!+\!R_{_3}\!\!-\!r_{_{\!24}}\!\!-\!r_{_{\!34}}\!) \\[0.1cm]
t_{_2}R_{_1}\!\!-\!(\!t_{_1}\!\!+\!t_{_2}\!)r_{_{\!24}}&-\!t_{_2}(\!R_{_1}\!\!+\!R_{_3}\!\!-\!r_{_{\!12}}\!\!-\!r_{_{\!23}}\!)&
t_{_2}R_{_3}\!\!-\!(\!t_{_1}\!\!+\!t_{_2}\!)r_{_{\!24}}&-\!t_{_2}(\!R_{_1}\!\!+\!R_{_3}\!\!-\!r_{_{\!14}}\!\!-\!r_{_{\!34}}\!) \\[0.1cm]
t_{_2}R_{_1}\!\!-\!(\!t_{_1}\!\!+\!t_{_2}\!)r_{_{\!34}}&t_{_2}R_{_2}\!\!-\!(\!t_{_1}\!\!+\!t_{_2}\!)r_{_{\!34}}&
-\!t_{_2}(\!R_{_1}\!\!+\!R_{_2}\!\!-\!r_{_{\!13}}\!\!-\!r_{_{\!23}}\!)&-\!t_{_2}(\!R_{_1}\!\!+\!R_{_2}\!\!-\!r_{_{\!14}}\!\!-\!r_{_{\!24}}\!)\\[0.1cm]
t_{_2}R_{_1}\!\!-\!(\!t_{_1}\!\!+\!t_{_2}\!)r_{_{\!23}}&-\!t_{_2}(\!R_{_1}\!\!+\!R_{_4}\!\!-\!r_{_{\!12}}\!\!-\!r_{_{\!24}}\!)&
-\!t_{_2}(\!R_{_1}\!\!+\!R_{_4}\!\!-\!r_{_{\!13}}\!\!-\!r_{_{\!34}}\!)&t_{_2}R_{_4}\!\!-\!(\!t_{_1}\!\!+\!t_{_2}\!)r_{_{\!23}} \\[0.1cm]
-\!t_{_2}(\!R_{_2}\!\!+\!R_{_4}\!\!-\!r_{_{\!12}}\!\!-\!r_{_{\!14}}\!)&t_{_2}R_{_2}\!\!-\!(\!t_{_1}\!\!+\!t_{_2}\!)r_{_{\!13}}&
-\!t_{_2}(\!R_{_2}\!\!+\!R_{_4}\!\!-\!r_{_{\!23}}\!\!-\!r_{_{\!34}}\!)&t_{_2}R_{_4}\!\!-\!(\!t_{_1}\!\!+\!t_{_2}\!)r_{_{\!13}} \\[0.1cm]
-\!t_{_2}(\!R_{_3}\!\!+\!R_{_4}\!\!-\!r_{_{\!13}}\!\!-\!r_{_{\!14}}\!)&-\!t_{_2}(\!R_{_3}\!\!+\!R_{_4}\!\!-\!r_{_{\!23}}\!\!-\!r_{_{\!24}}\!)&
t_{_2}R_{_3}\!\!-\!(\!t_{_1}\!\!+\!t_{_2}\!)r_{_{\!12}}&t_{_2}R_{_4}\!\!-\!(\!t_{_1}\!\!+\!t_{_2}\!)r_{_{\!12}} \end{array}\right)},\\[0.15cm]
&\textbf{A}^{(3)}_{_{6\!\times\!4}}\!\!=\!\!{\scriptsize
\left(\begin{array}{cccc} 
(\!2t_{_3}\!\!+\!t_{_4}\!)R_{_1}\!\!-\!(\!2t_{_3}\!\!-\!3t_{_4}\!)r_{_{\!23}}&(\!2t_{_3}\!\!-\!t_{_4}\!)(\!R_{_2} \!\!-\! r_{_{\!14}}\!)& 
(\!2t_{_3}\!\!-\!t_{_4}\!)(\!R_{_3}\!\!-\! r_{_{\!14}}\!)&(\!2t_{_3}\!\!+\!t_{_4}\!)R_{_4} \!\!-\!(\!2t_{_3}\!\!-\!3t_{_4}\!) r_{_{\!23}}\\[0.1cm]
(\!2t_{_3}\!\!-\!t_{_4}\!)(\!R_{_1}\!\!-\! r_{_{\!24}}\!)&(\!2t_{_3}\!\!+\!t_{_4}\!)R_{_2} \!\!-\!(\!2t_{_3}\!\!-\!3t_{_4}\!) r_{_{\!13}} & 
(\!2t_{_3}\!\!-\!t_{_4}\!)(\!R_{_3}\!\!-\! r_{_{\!24}}\!)&(\!2t_{_3}\!\!+\!t_{_4}\!)R_{_4} \!\!-\!(\!2t_{_3}\!\!-\!3t_{_4}\!) r_{_{\!13}}\\[0.1cm]
(\!2t_{_3}\!\!-\!t_{_4}\!)(\!R_{_1}\!\!-\! r_{_{\!34}}\!)&(\!2t_{_3}\!\!-\!t_{_4}\!)(\!R_{_2} \!\!-\!r_{_{\!34}}\!)& 
(\!2t_{_3}\!\!+\!t_{_4}\!)R_{_3} \!\!-\!(\!2t_{_3}\!\!-\!3t_{_4}\!) r_{_{\!12}} & (\!2t_{_3}\!\!+\!t_{_4}\!)R_{_4} \!\!-\!(\!2t_{_3}\!\!-\!3t_{_4}\!) r_{_{\!12}} \\[0.1cm]
(\!2t_{_3}\!\!-\!t_{_4}\!)(\!R_{_1} \!\!-\! r_{_{\!23}}\!) & (\!2t_{_3}\!\!+\!t_{_4}\!)R_{_2} \!\!-\!(\!2t_{_3}\!\!-\!3t_{_4}\!) r_{_{\!14}} & 
(\!2t_{_3}\!\!+\!t_{_4}\!)R_{_3} \!\!-\!(\!2t_{_3}\!\!-\!3t_{_4}\!) r_{_{\!14}} & (\!2t_{_3}\!\!-\!t_{_4}\!)(\!R_{_4} \!\!-\! r_{_{\!23}}\!)\\ [0.1cm]
(\!2t_{_3}\!\!+\!t_{_4}\!)R_{_1} \!\!-\!(\!2t_{_3}\!\!-\!3t_{_4}\!) r_{_{\!24}} & (\!2t_{_3}\!\!-\!t_{_4}\!)(\!R_{_2} \!\!-\! r_{_{\!13}}\!)& 
(\!2t_{_3}\!\!+\!t_{_4}\!)R_{_3} \!\!-\!(\!2t_{_3}\!\!-\!3t_{_4}\!) r_{_{\!24}} & (\!2t_{_3}\!\!-\!t_{_4}\!)(\!R_{_4} \!\!-\! r_{_{\!13}}\!)\\[0.1cm] 
(\!2t_{_3}\!\!+\!t_{_4}\!)R_{_1} \!\!-\!(\!2t_{_3}\!\!-\!3t_{_4}\!) r_{_{\!34}} & (\!2t_{_3}\!\!+\!t_{_4}\!)R_{_2} \!\!-\!(\!2t_{_3}\!\!-\!3t_{_4}\!)r_{_{\!34}}& 
(\!2t_{_3}\!\!-\!t_{_4}\!)(\!R_{_3} \!\!-\! r_{_{\!12}}\!)&(\!2t_{_3}\!\!-\!t_{_4}\!)(\!R_{_4} \!\!-\!r_{_{\!12}}) \end{array}\right)},\\[0.15cm]
&\textbf{A}^{(4)}_{_{6\!\times\!6}}\!\!=\!\!{\scriptsize
\left(\begin{array}{cccccc} 
t_{_3}\!(\! R_{_1} \!\!+\! R_{_4} \!\!-\! 2r_{_{\!23}}\!) & 
t_{_4}\!r_{_{\!12}}\!\!-\!t_{_3}\!(\!R_{_1} \!\!-\! r_{_{\!23}}\!) & 
t_{_4}\!r_{_{\!13}}\!\!-\!t_{_3}\!(\!R_{_1} \!\!-\! r_{_{\!23}}\!)&
\!\!\!-\!t_{4}(\!R_{_1} \!\!+\! R_{_4}\!)& 
t_{_4}\!r_{_{\!34}}\!\!-\!t_{_3}\!(\!R_{_4} \!\!-\! r_{_{\!23}}\!) & 
t_{_4}\!r_{_{\!24}}\!\!-\!t_{_3}\!(\!R_{_4} \!\!-\! r_{_{\!23}}\!)\\ [0.1cm]
t_{_4}\!r_{_{\!12}}\!\!-\!t_{_3}\!(\!R_{_2} \!\!-\! r_{_{\!13}}\!) & 
t_{_3}\!(\! R_{_2} \!\!+\! R_{_4} \!\!-\! 2r_{_{\!13}}\!) & 
t_{_4}\!r_{_{\!23}}\!\!-\!t_{_3}\!(\!R_{_2} \!\!-\! r_{_{\!13}}\!)&
t_{_4}\!r_{_{\!34}}\!\!-\!t_{_3}\!(\!R_{_4} \!\!-\! r_{_{\!13}}\!) & 
\!\!\!-\!t_{4}(\!R_{_2} \!\!+\! R_{_4}\!) & 
t_{_4}\!r_{_{\!14}}\!\!-\!t_{_3}\!(\!R_{_4} \!\!-\! r_{_{\!13}}\!)\\ [0.1cm]
t_{_4}\!r_{_{\!13}}\!\!-\!t_{_3}\!(\!R_{_3} \!\!-\! r_{_{\!12}}\!) & 
t_{_4}\!r_{_{\!23}}\!\!-\!t_{_3}\!(\!R_{_3} \!\!-\! r_{_{\!12}}\!) & 
t_{_3}\!(\! R_{_3} \!\!+\! R_{_4} \!\!-\! 2r_{_{\!12}}\!)&
t_{_4}\!r_{_{\!24}}\!\!-\!t_{_3}\!(\!R_{_4} \!\!-\! r_{_{\!12}}\!) & 
t_{_4}\!r_{_{\!14}}\!\!-\!t_{_3}\!(\!R_{_4} \!\!-\! r_{_{\!12}}\!) & 
\!\!\!-\!t_{4}(\!R_{_3} \!\!+\! R_{_4}\!)\\[0.1cm]
\!\!\!-\!t_{4}(\!R_{_2} \!\!+\! R_{_3}\!) & 
t_{_4}\!r_{_{\!34}}\!\!-\!t_{_3}\!(\!R_{_3} \!\!-\! r_{_{\!14}}\!) & 
t_{_4}\!r_{_{\!24}}\!\!-\!t_{_3}\!(\!R_{_2} \!\!-\! r_{_{\!14}}\!)&
t_{_3}\!(\! R_{_2} \!\!+\! R_{_3} \!\!-\! 2r_{_{\!14}}\!) & 
t_{_4}\!r_{_{\!12}}\!\!-\!t_{_3}\!(\!R_{_2} \!\!-\! r_{_{\!14}}\!)& 
t_{_4}\!r_{_{\!13}}\!\!-\!t_{_3}\!(\!R_{_3} \!\!-\! r_{_{\!14}}\!)\\ [0.1cm]
t_{_4}\!r_{_{\!34}}\!\!-\!t_{_3}\!(\!R_{_3} \!\!-\! r_{_{\!24}}\!) & 
\!\!\!-\!t_{4}(\!R_{_1} \!\!+\! R_{_3}\!)  & 
t_{_4}\!r_{_{\!14}}\!\!-\!t_{_3}\!(\!R_{_1} \!\!-\! r_{_{\!24}}\!)&
t_{_4}\!r_{_{\!12}}\!\!-\!t_{_3}\!(\!R_{_1} \!\!-\! r_{_{\!24}}\!) & 
t_{_3}\!(\! R_{_1} \!\!+\! R_{_3} \!\!-\! 2r_{_{\!24}}\!) & 
t_{_4}\!r_{_{\!23}}\!\!-\!t_{_3}\!(\!R_{_3} \!\!-\! r_{_{\!24}}\!)\\ [0.1cm]
t_{_4}\!r_{_{\!24}}\!\!-\!t_{_3}\!(\!R_{_2} \!\!-\! r_{_{\!34}}\!) & 
t_{_4}\!r_{_{\!14}}\!\!-\!t_{_3}\!(\!R_{_1} \!\!-\! r_{_{\!34}}\!) & 
\!\!\!-\!t_{4}(\!R_{_1} \!\!+\! R_{_2}\!) &
t_{_4}\!r_{_{\!13}}\!\!-\!t_{_3}\!(\!R_{_1} \!\!-\! r_{_{\!34}}\!) & 
t_{_4}\!r_{_{\!23}}\!\!-\!t_{_3}\!(\!R_{_2} \!\!-\! r_{_{\!34}}\!)& 
t_{_3}\!(\! R_{_1} \!\!+\! R_{_2} \!\!-\! 2r_{_{\!34}}\!) 
\end{array}\right)}.
\end{split} &
\end{flalign*}
\subsection*{A.2 The symbolic matrices
$\textbf{Q}^{(1)}_{_{3\!\times\!3}}  $, $\textbf{Q}^{(2)}_{_{3\!\times\!3}}  $, $ \textbf{L}^{(1)}_{_{3\!\times\!3}} $, $ \textbf{L}^{(2)}_{_{3\!\times\!3}} $, $ \textbf{L}^{(3)}_{_{3\!\times\!3}} $, $ \textbf{L}^{(4)}_{_{3\!\times\!3}} $ in (\ref{part2})}\label{Apsubsection: several complex matrices}
\begin{equation*}
\begin{split}
&\textbf{L}^{(1)}_{_{3\!\times\!3}}=(\frac{s_{_1}}{2}-\frac{s_{_2}+s_{_3}}{2})\textbf{M}^{K}_{_{3\!\times\!3}}+(\frac{3s_{_1}}{20}-\frac{s_{_2}+s_{_3}}{2})\textbf{J}^{(1)}_{_{3\!\times\!3}}+s_{_3}\textbf{D}^{(1)}_{_{3\!\times\!3}}+\frac{s^{*}}{2}\textbf{Q}^{(1)}_{_{3\!\times\!3}},\\
&\textbf{L}^{(2)}_{_{3\!\times\!3}}=(\frac{s_{_1}}{4}-\frac{s_{_2}+s_{_3}}{4})(\textbf{M}^{K}_{_{3\!\times\!3}}\textbf{C}_{_{3\!\times\!3}})-(\frac{3s_{_1}}{20}-\frac{s_{_2}}{2})\textbf{J}^{(2)}_{_{3\!\times\!3}}+\frac{s_{_3}}{2}\textbf{D}^{(2)}_{_{3\!\times\!3}}+\frac{s^{*}}{4}\textbf{Q}^{(2)}_{_{3\!\times\!3}},\\
&\textbf{L}^{(3)}_{_{3\!\times\!3}}=\frac{s_{_2}}{2}(\textbf{C}_{_{3\!\times\!3}}\textbf{M}^{K}_{_{3\!\times\!3}})-\frac{s_{_1}}{20}\textbf{J}^{(1)}_{3\times 3}+\frac{s_{_2}}{2}\textbf{D}^{(1)}_{_{3\!\times\!3}}-\frac{s_{_2}-s_{_3}}{2}\textbf{D}^{(3)}_{_{3\!\times\!3}},\\
&\textbf{L}^{(4)}_{_{3\!\times\!3}}=\frac{s_{_2}}{4}(\textbf{C}_{_{3\!\times\!3}}\textbf{M}^{K}_{_{3\!\times\!3}}\textbf{C}_{_{3\!\times\!3}})+(\frac{s_{_1}}{20}-\frac{s_{_2}}{4})\textbf{J}^{(2)}_{_{3\!\times\!3}}+\frac{s_{_2}}{4}\textbf{D}^{(2)}_{_{3\!\times\!3}}-\frac{s_{_2}-s_{_3}}{4}\textbf{D}^{(4)}_{_{3\!\times\!3}}.
\end{split}
\end{equation*}
where $ s_{1} $, $ s_{2} $, $ s_{3} $, $ s^{*} $, $ \textbf{M}^{K}_{_{3\!\times\!3}} $, $ \textbf{C}_{_{3\!\times\!3}} $ are given by (\ref{s_{_0}-3}) and (\ref{M,C}), and
\begin{gather*}
\textbf{J}^{(1)}_{_{3\!\times\!3}}\!=\!{\footnotesize \begin{pmatrix}
R_{_1}&R_{_2}&R_{_3}\\
R_{_1}&R_{_2}&R_{_3}\\
R_{_1}&R_{_2}&R_{_3}\\
\end{pmatrix}}, \quad
\textbf{J}^{(2)}_{_{3\!\times\!3}}\!=\!{\footnotesize \begin{pmatrix}
r_{_{\!14}}&r_{_{\!24}}&r_{_{\!34}}\\
r_{_{\!14}}&r_{_{\!24}}&r_{_{\!34}}\\
r_{_{\!14}}&r_{_{\!24}}&r_{_{\!34}}\\
\end{pmatrix}},\\	
\textbf{Q}^{(1)}_{_{3\!\times\!3}}= \begin{pmatrix}
\!-\!2R_{_1}&r_{_{\!13}}\!+\!r_{_{\!14}}&r_{_{\!12}}\!+\!r_{_{\!14}}\\
r_{_{\!23}}\!+\!r_{_{\!24}}&\!-\!2R_{_2}&r_{_{\!12}}\!+\!r_{_{\!24}}\\
r_{_{\!23}}\!+\!r_{_{\!34}}&r_{_{\!13}}\!+\!r_{_{\!34}}&\!-\!2R_{_3}\\
\end{pmatrix},\quad
\textbf{Q}^{(2)}_{_{3\!\times\!3}}= \begin{pmatrix}
\!-\!2r_{_{\!14}}\!-\!\!r_{_{\!24}}\!-\!\!r_{_{\!34}}&R_{_1}\!+\!r_{_{\!24}}&R_{_1}\!+\!r_{_{\!34}}\\
R_{_2}\!+\!r_{_{\!14}}&\!-\!r_{_{\!14}}\!-\!2r_{_{\!24}}\!-\!r_{_{\!34}}&R_{_2}\!+\!r_{_{\!34}}\\
R_{_3}\!+\!r_{_{\!14}}&R_{_3}\!+\!r_{_{\!24}}&\!-\!r_{_{\!14}}\!-\!\!r_{_{\!24}}\!-\!\!2r_{_{\!34}}\\
\end{pmatrix},\\
\textbf{D}^{(1)}_{_{3\!\times\!3}}={\footnotesize \begin{pmatrix}
R_{_1}&0&0\\
0&R_{_2}&0\\
0&0&R_{_3}\\
\end{pmatrix}},\quad
\textbf{D}^{(2)}_{_{3\!\times\!3}}={\footnotesize \begin{pmatrix}
r_{_{\!14}}&0&0\\
0&r_{_{\!24}}&0\\
0&0&r_{_{\!34}}\\
\end{pmatrix}},\quad
\textbf{D}^{(3)}_{_{3\!\times\!3}}={\footnotesize \begin{pmatrix}
r_{_{\!23}}&0&0\\
0&r_{_{\!13}}&0\\
0&0&r_{_{\!12}}\\
\end{pmatrix}},\quad
\textbf{D}^{(4)}_{_{3\!\times\!3}}={\footnotesize \begin{pmatrix}
0&r_{_{\!12}}&r_{_{\!13}}\\
r_{_{\!12}}&0&r_{_{\!23}}\\
r_{_{\!13}}&r_{_{\!23}}&0\\
\end{pmatrix}}.
\end{gather*}
\section*{Appendix B: Some relations in a tetrahedron and two related proofs}\label{Section:Appendix B}
\subsection*{B.1 Some relations in a tetrahedron}\label{relations in a tetra}
In this subsection, we discuss some relations in a tetrahedron $ K=\bigtriangleup^{4} P_{1}P_{2}P_{3}P_{4} $ about its plane angles $ \theta_{i_{1},P_{i_{2}}}$  $(\textstyle i_1\in\mathcal{Z}_{3}^{(1)}\!\!,i_2\in\mathcal{Z}_{4}^{(1)} )$ in (\ref{definition:plane angles}), dihedral angles $ \theta_{jk} $  $ \big(\textstyle(j,k)\!\in\!\mathcal{Z}_{4}^{(2)}\big)$, edge lengths $ |\overline{P_{j}P_{k}|}$ $ \big(\textstyle(j,k)\!\in\!\mathcal{Z}_{4}^{(2)}\big)$, and  circumradius $ R_{K} $. 

See Fig.~\ref{fig:relation}, we take a positively oriented orthogonal coordinate system $ (x_1,x_2,x_3) $, such that $ P_{1} $ is the origin,  $ \overline{P_{1}P_{2}} $ on the $ x_{1} $ axis, and the $ x_1P_{1}x_2 $ plane coincides with the plane spanned by $ \overline{P_{1}P_{2}} $ and $ \overline{P_{1}P_{3}} $. Here, $ O, O_1,O_2 $ are the circumcenters of $ \bigtriangleup^{4} P_{1}P_{2}P_{3}P_{4}, \bigtriangleup P_1P_2P_4,\bigtriangleup P_1P_2P_3  $, respectively, and $ M_{12} $ is the midpoint of $\overline{P_{1}P_{2}}  $. 
\begin{figure}[ht!]
\centering
\begin{minipage}[t]{.33\textwidth}
\centering
\includegraphics[width=150pt]{./l_R_relation.eps}
\caption{}
\label{fig:relation}
\end{minipage}
\end{figure}   
\begin{lemma}\label{lemma:angles}
For given $ K=\bigtriangleup^{4} P_{1}P_{2}P_{3}P_{4} $ (see Fig.~\ref{fig:relation}), its  plane angles and dihedral angles  satisfy 
\begin{align}
&\cos\theta_{12}=\dfrac{\cos\theta_{_{3,P_1}}-\cos\theta_{_{1,P_1}}\cos\theta_{_{2,P_1}}}{\sin\theta_{_{1,P_1}}\sin\theta_{_{2,P_1}}},\label{angle1}\\
&\cos\theta_{_{3,P_1}}=\dfrac{\cos\theta_{_{2,P_1}}\sin\theta_{12}\cos\theta_{13}+\cos\theta_{12}\sin\theta_{13}}{\sqrt{(\cos\theta_{_{2,P_1}}\sin\theta_{12}\cos\theta_{13}+\cos\theta_{12}\sin\theta_{13})^2+\sin^2\theta_{_{2,P_1}}\sin^2\theta_{12}}}.\label{angle2}
\end{align}
\end{lemma}
\begin{proof}
Let $ \textbf{n}_{_i}=(n_{_{i,1}},n_{_{i,2}},n_{_{i,3}})^T$ be the unit outer normal vector of the triangular face  $ T_{i} $,  and $\textbf{n}_{_{kl}}=\frac{\overrightarrow{P_{k}P_{l}}}{|\overline{P_{k}P_{l}|}}=(n_{_{kl,1}},n_{_{kl,2}},n_{_{kl,3}})^T  $ be the unit direction vector from  $ P_{k} $ to  $ P_{l} $. From Fig.~\ref{fig:relation}, it is easy to find that
\begin{equation*}
\begin{array}{lcl}
\textbf{n}_{_{12}}=(1,0,0)^T,\,\,\textbf{n}_{_4}=(0,0,-1)^T,\,\,\textbf{n}_{_3}=(0,-\sin\theta_{12},\cos\theta_{12})^T,\\[2mm]
\textbf{n}_{_{13}}=(\cos\theta_{_{2,P_1}},\sin\theta_{_{2,P_1}},0)^T,\,\,\textbf{n}_{_{14}}=(\cos\theta_{_{1,P_1}},n_{_{14,2}},n_{_{14,3}})^T,\,\,\textbf{n}_{_2}=(n_{_{2,1}},n_{_{2,2}},\cos\theta_{13})^T.
\end{array}
\end{equation*}

Since $\theta_{12} $ equals to the angle between  $-\textbf{n}_{_4}$ and $\textbf{n}_{_3}$, where $ \textbf{n}_{_3} $ has the same direction with  $ \textbf{n}_{_{12}}\times\textbf{n}_{_{14}} $, we have
\begin{equation}\label{lemmaeq:teta3,p1}
\cos\theta_{12}=-\textbf{n}_{_3}\cdot \textbf{n}_{_4}=-\dfrac{\textbf{n}_{_{12}}\times\textbf{n}_{_{14}}}{|\textbf{n}_{_{12}}\times\textbf{n}_{_{14}}|}\cdot \textbf{n}_{_4}=\dfrac{(0,n_{_{14,3}},-n_{_{14,2}})^T}{\sin\theta_{_{1,P_1}}}\cdot(0,0,-1)^T=\dfrac{n_{_{14,2}}}{\sin\theta_{_{1,P_1}}}.
\end{equation}
Noticing that $ \theta_{_{3,P_1}} $ is the angle between $\textbf{n}_{_{13}}$ and $\textbf{n}_{_{14}}$,  satisfying
$ \cos\theta_{_{3,P_1}}=\textbf{n}_{_{13}}\cdot\textbf{n}_{_{14}}=\cos\theta_{_{1,P_1}}\cos\theta_{_{2,P_1}}+\sin\theta_{_{2,P_1}}n_{_{14,2}} $,
we have
\begin{equation*}
\begin{array}{lcl}
n_{_{14,2}}=\dfrac{\cos\theta_{_{3,P_1}}-\cos\theta_{_{1,P_1}}\cos\theta_{_{2,P_1}}}{\sin\theta_{_{2,P_1}}},
\end{array}
\end{equation*}
which together with (\ref{lemmaeq:teta3,p1}) leads to (\ref{angle1}).

On the other hand, since $\theta_{_{3,P_{1}}} $ is the angle between $\textbf{n}_{_{13}}$ and $\textbf{n}_{_{14}}$, where $\textbf{n}_{_{14}}$ has the same direction with  $ \textbf{n}_{_2}\times\textbf{n}_{_3} $, we have
\begin{align}
\cos\theta_{_{3,P_{1}}}&=\textbf{n}_{_{13}}\cdot \textbf{n}_{_{14}}= \textbf{n}_{_{13}}\cdot\dfrac{\textbf{n}_{_2}\times\textbf{n}_{_3}}{|\textbf{n}_{_2}\times\textbf{n}_{_3}|}\nonumber\\
&=(\cos\theta_{_{2,P_1}},\sin\theta_{_{2,P_1}},0)^T\cdot\dfrac{(n_{_{2,2}}\cos\theta_{12}+\cos\theta_{13}\sin\theta_{12},-n_{_{2,1}}\cos\theta_{12},-n_{_{2,1}}\sin\theta_{12})^T}{\sqrt{(n_{_{2,2}}\cos\theta_{12}+\cos\theta_{13}\sin\theta_{12})^{2}+n_{_{2,1}}^{2}}}\nonumber\\
&=\dfrac{\cos\theta_{_{2,P_1}}(n_{_{2,2}}\cos\theta_{12}+\cos\theta_{13}\sin\theta_{12})-\sin\theta_{_{2,P_1}}n_{_{2,1}}\cos\theta_{12}}{\sqrt{(n_{_{2,2}}\cos\theta_{12}+\cos\theta_{13}\sin\theta_{12})^{2}+n_{_{2,1}}^{2}}}.\label{lemmaeq:costeat}
\end{align}
Noticing that 
$ \textbf{n}_{_2}\cdot\textbf{n}_{_{13}}=\cos\theta_{_{2,P_1}}n_{_{2,1}}+\sin\theta_{_{2,P_1}}n_{_{2,2}}=0$ and $
|\textbf{n}_{_2}|^{2}=n_{_{2,1}}^2+n_{_{2,2}}^2+\cos\theta_{12}^{2}=1 $,
one obtains
\begin{equation*}
\begin{split}
n_{_{2,1}}=-\sin\theta_{_{2,P_1}}\sin\theta_{12},\quad n_{_{2,2}}=\cos\theta_{_{2,P_1}}\sin\theta_{12},
\end{split} 
\end{equation*}
which together with (\ref{lemmaeq:costeat}) lead to (\ref{angle2}). 
 \end{proof}
\begin{lemma}\label{lemma:L-R}
For given $ K=\bigtriangleup^{4} P_{1}P_{2}P_{3}P_{4} $ (see Fig.~\ref{fig:relation}), the edge length $|\overline{P_1P_2}|  $ can be represented as follows
\begin{equation}\label{l_R}
|\overline{P_1P_2}|=2R_{K}\dfrac{\sin\theta_{12}}{\sqrt{\sin^2\theta_{12}+\cot^2\theta_{2,P_3}+\cot^2\theta_{1,P_4}-2\,\cot\theta_{2,P_3}\cot\theta_{1,P_4}\cos\theta_{12}}}.
\end{equation}
\end{lemma}
\begin{proof}
The relationship between the diameter and the chord length in the circumcircles of $\bigtriangleup P_1P_2P_4$ and $\bigtriangleup P_1P_2P_3  $ yields
\begin{equation*}
\dfrac{2\,|\overline{O_1M_{12}}|}{|\overline{P_1P_2}|}=\cot\theta_{1,P_4},\quad
\dfrac{2\,|\overline{O_2M_{12}}|}{|\overline{P_1P_2}|}=\cot\theta_{2,P_3}.
\end{equation*}
Besides, the circumcenters  $ O$, $O_1$, $O_2 $ are on a common plane which passes through $ M_{12} $ and is perpendicular to  $\overline{P_{1}P_{2}}  $, and $ \angle OO_1M_{12}=\angle OO_2M_{12}=\pi/2 $. Thus, the four points $ O,O_1,O_2,M_{12} $ are on a common circle with the diameter $ |\overline{OM_{12}}| $. Then
\begin{equation*}
\dfrac{|\overline{O_1O_2}|}{ |\overline{OM_{12}}|}=\sin\theta_{12}.\\[0.2cm]
\end{equation*}
Substituting the above equations and  $ |\overline{OM_{12}}|^{2}+\frac{|\overline{P_{1}P_{2}}|^{2}}{4}=R_{K}^{2} $ into the law of cosines on $ \bigtriangleup O_1M_{12}O_2 $
\begin{equation*}
|\overline{O_1O_2}|^2=|\overline{O_1M_{12}}|^2+|\overline{O_2M_{12}}|^2-2\,|\overline{O_1M_{12}}||\overline{O_2M_{12}}|\cos\theta_{12},\\[0.2cm]
\end{equation*}
one arrives at (\ref{l_R}). 
 \end{proof}
Other relations of angles and lengths in $ K=\bigtriangleup^{4} P_{1}P_{2}P_{3}P_{4} $ similar to those given by Lemma~\ref{lemma:angles} and Lemma~\ref{lemma:L-R} can be easily derived. Without causing confusion, we use (\ref{angle1}), (\ref{angle2}) and (\ref{l_R}) to represent themselves and other similar relations.
\subsection*{B.2 A proof of Lemma~\ref{determine_relation}}\label{Apsubsection:2}
We shall give a proof of Lemma~\ref{determine_relation}. It suffices to verify that the three leading principal minors of  $ \frac{1}{h_{K}}\textbf{M}^{K}_{_{3\!\times\!3}} $ are positive. By Lemma~\ref{lemma:relation Li}, we have $ \frac{1}{h_{K}}R_{_1}=\frac{2|T_{1}|^{2}}{3h_{K}|K|} $ and
\begin{equation*}
R_{_1}R_{_2}=(\dfrac{2|T_{1}||T_{2}|}{3|K|})^{2}=(\dfrac{2|T_{1}||T_{2}|}{3|K|}\textbf{n}_{1}\cdot\textbf{n}_{2})^{2}(1+\tan^2\theta_{34})=r^2_{34}+|\overline{P_{3}P_{4}}|^2,
\end{equation*}
which leads to 
\begin{equation*}
\frac{1}{h_{K}^{2}}\left|\begin{array}{cc} 
R_{_1} & -r_{_{\!34}}\\
-r_{_{\!34}} & R_{_2}\\
\end{array}\right|=\frac{|\overline{P_{3}P_{4}}|^2}{h_{K}^{2}}.
\end{equation*}
Moreover, by $ 2S_{\triangle ABC}=|AB||AC|\sin\angle A $, we can write the determinant of $ \frac{1}{h_{K}}\textbf{M}^{K}_{_{3\!\times\!3}} $ as
\hspace{-2mm}
\begin{equation}\label{eq:A2-1}
\det(\frac{1}{h_{K}}\textbf{M}^{K}_{_{3\!\times\!3}})=\frac{1}{h_{K}^{3}}\left|\begin{array}{ccc} 
\frac{2|T_{1}|^{2}}{3|K|}&\!-\!\frac{2|T_{1}||T_{2}|}{3|K|}\cos\theta_{_{\!34}}&\!-\!\frac{2|T_{1}||T_{3}|}{3|K|}\cos\theta_{_{\!24}}\\[0.2cm]
\!-\!\frac{2|T_{1}||T_{2}|}{3|K|}\cos\theta_{_{\!34}}&\frac{2|T_{2}|^{2}}{3|K|}&\!-\!\frac{2|T_{2}||T_{3}|}{3|K|}\cos\theta_{_{\!14}}\\[0.2cm]
\!-\!\frac{2|T_{1}||T_{3}|}{3|K|}\cos\theta_{_{\!24}}&\!-\!\frac{2|T_{2}||T_{3}|}{3|K|}\cos\theta_{_{\!14}}&\frac{2|T_{3}|^{2}}{3|K|}
\end{array}\right|=\frac{c_{L}^{4}}{h_{K}^{3}(6|K|)^{3}}c_{d},
\end{equation}
where $ c_{L}= |\overline{P_{1}P_{4}}||\overline{P_{2}P_{4}}||\overline{P_{3}P_{4}}|$ and
\begin{equation*}
c_d=\left|\begin{array}{ccc} 
\sin^2\theta_{_{\!2,P_4}}\!&\,\,-\!\sin\theta_{_{\!2,P_4}}\!\!\sin\theta_{_{\!3,P_4}}\!\!\cos\theta_{_{\!34}}&\,\,-\!\sin\theta_{_{\!1,P_4}}\!\!\sin\theta_{_{\!2,P_4}}\!\!\cos\theta_{_{\!24}}\\[0.2cm]
\!-\!\sin\theta_{_{\!2,P_4}}\!\!\sin\theta_{_{\!3,P_4}}\!\!\cos\theta_{_{\!34}}&\,\,\sin^2\theta_{_{\!3,P_4}}\!&\,\,-\!\sin\theta_{_{\!1,P_4}}\!\!\sin\theta_{_{\!3,P_4}}\!\!\cos\theta_{_{\!14}}\\[0.2cm]
\!-\!\sin\theta_{_{\!1,P_4}}\!\!\sin\theta_{_{\!2,P_4}}\!\!\cos\theta_{_{\!24}}&\,\,-\!\sin\theta_{_{\!1,P_4}}\!\!\sin\theta_{_{\!3,P_4}}\!\!\cos\theta_{_{\!14}}&\,\,\sin^2\theta_{_{\!1,P_4}}\!\\
\end{array}\right|.
\end{equation*}
Relations (\ref{angle1}) and $ \sin^{2}\theta+\cos^{2}\theta=1 $ yield
\begin{equation}\label{eq:A2-2}
c_d=\left|\begin{array}{ccc} 
1\!-\!\cos^2\theta_{_{\!2,P_4}}\!&\cos\theta_{_{\!2,P_4}}\!\!\cos\theta_{_{\!3,P_4}}\!\!\!\!-\!\cos\theta_{_{\!1,P_4}}\!&\cos\theta_{_{\!1,P_4}}\!\!\cos\theta_{_{\!2,P_4}}\!\!\!\!-\!\cos\theta_{_{\!3,P_4}}\!\\[0.2cm]
\cos\theta_{_{\!2,P_4}}\!\!\cos\theta_{_{\!3,P_4}}\!\!\!\!-\!\cos\theta_{_{\!1,P_4}}\!&1\!-\!\cos^2\theta_{_{\!3,P_4}}\!&\cos\theta_{_{\!1,P_4}}\!\!\cos\theta_{_{\!3,P_4}}\!\!\!\!-\!\cos\theta_{_{\!2,P_4}}\!\\[0.2cm]
\cos\theta_{_{\!1,P_4}}\!\!\cos\theta_{_{\!2,P_4}}\!\!\!\!-\!\cos\theta_{_{\!3,P_4}}\!&\cos\theta_{_{\!1,P_4}}\!\!\cos\theta_{_{\!3,P_4}}\!\!\!\!-\!\cos\theta_{_{\!2,P_4}}\!&1\!-\!\cos^2\theta_{_{\!1,P_4}}\!\\
\end{array}\right|.
\end{equation}

For $ c_{L} $ in (\ref{eq:A2-1}), there is the volume formula that
\begin{align}
6|K|&=|T_{1}||\overline{P_{1}P_{4}}|\sin\theta_{_{\!2,P_4}}\!\sin\theta_{_{\!34}}=c_{L}\sin\theta_{_{\!2,P_4}}\!\sin\theta_{_{\!3,P_4}}\!\sqrt{1\!-\!\cos^{2}\theta_{_{\!34}}}\nonumber\\
&=c_{L}\sin\theta_{_{\!2,P_4}}\!\sin\theta_{_{\!3,P_4}}\!\sqrt{1\!-\!(\frac{\cos \theta_{_{1,P_{4}}}\!-\!\cos \theta_{_{2,P_{4}}}\cos \theta_{_{3,P_{4}}}}{\sin \theta_{_{2,P_{4}}}\sin \theta_{_{3,P_{4}}}})^{2}}
=c_{L}\sqrt{c_{K}},\label{eq:A2-3}
\end{align}
with $ c_{K}=1\!-\!\cos^2\theta_{1,P_4}\!\!-\!\cos^2\theta_{2,P_4}\!\!-\!\cos^2\theta_{3,P_4}\!\!+\!2\cos\theta_{1,P_4}\cos\theta_{2,P_4}\cos\theta_{3,P_4} $. 
For $ c_{d} $ in (\ref{eq:A2-1}), multiplying (\ref{eq:A2-2}) by a determinant whose value is 1, one obtains
\begin{equation}\label{eq:A2-4}
\begin{split}
c_d\left|\begin{array}{ccc} 
1&0&\,\,\,\,0\\
\cos\theta_{1,P_4}&1&\,\,\,\,0\\
\cos\theta_{3,P_4}&0&\,\,\,\,1\\
\end{array}\right|
=\left|\begin{array}{ccc} 
c_K&\,\,\,\cos\theta_{2,P_4}\cos\theta_{3,P_4}\!\!-\!\!\cos\theta_{1,P_4}&\,\,\cos\theta_{1,P_4}\cos\theta_{2,P_4}\!\!-\!\!\cos\theta_{3,P_4}\\[0.2cm]
0&\,\,\,1\!\!-\!\!\cos^2\theta_{3,P_4}&\,\,\cos\theta_{1,P_4}\cos\theta_{3,P_4}\!\!-\!\!\cos\theta_{2,P_4}\\[0.2cm]
0&\,\,\,\cos\theta_{1,P_4}\cos\theta_{3,P_4}\!\!-\!\!\cos\theta_{2,P_4}&\,\,1\!\!-\!\!\cos^2\theta_{1,P_4}\\
\end{array}\right|= c_K^2.
\end{split}
\end{equation}
Then, substituting (\ref{eq:A2-3}) and (\ref{eq:A2-4}) into (\ref{eq:A2-1}) yields
\begin{equation*}
\det(\frac{1}{h_{K}}\textbf{M}^{K}_{_{3\!\times\!3}})=\frac{6|K|}{h_{K}^{3}}.
\end{equation*}
Since $\mathcal{T}_{h}$ is a regular partition,  the proof is completed. 
\subsection*{B.3 A proof of Lemma~\ref{lemma:indentity}}\label{Apsubsection:3}
We shall give a proof of Lemma~\ref{lemma:indentity}. By representations (\ref{angle1}) and (\ref{angle2}), the remaining seven of the twelve plane angles $ \theta_{i_1,P_{i_2}}$  $ (i_1\in\mathcal{Z}_{3}^{(1)},i_2\in\mathcal{Z}_{4}^{(1)}) $  can be represented by $\Theta_{5}$ in the following order
\begin{equation*}
\begin{split}
&\Theta_{5}\xrightarrow{}
\left\{
\begin{array}{lcl}
\theta_{_{1,P_{4}}}=180^{\circ}-\theta_{_{1,P_{1}}}-\theta_{_{1,P_{2}}}\\[0.1cm]
\theta_{_{2,P_{3}}}=180^{\circ}-\theta_{_{2,P_{1}}}-\theta_{_{2,P_{2}}}
\end{array}
\right.
\xrightarrow{(\ref{angle1})}
\left\{
\begin{array}{lcl}
\theta_{12}=\arccos(\dfrac{\cos\theta_{_{3,P_{2}}}-\cos \theta_{_{1,P_{2}}}\cos \theta_{_{2,P_{2}}}}{\sin \theta_{_{1,P_{2}}}\sin \theta_{_{2,P_{2}}}})\\[0.35cm]
\theta_{24}=\arccos(\dfrac{\cos \theta_{_{2,P_{2}}}-\cos \theta_{_{1,P_{2}}}\cos\theta_{_{3,P_{2}}}}{\sin \theta_{_{1,P_{2}}}\sin\theta_{_{3,P_{2}}}})
\end{array}
\right. \\
&\xrightarrow{(\ref{angle1})}
\left\{
\begin{array}{lcl}
\theta_{_{3,P_{1}}}=\arccos(\cos\theta_{12}\sin \theta_{_{1,P_{1}}}\sin \theta_{_{2,P_{1}}}+\cos \theta_{_{1,P_{1}}}\cos \theta_{_{2,P_{1}}})\\[0.15cm]
\theta_{14}=\arccos(\dfrac{\cos \theta_{_{2,P_{2}}}-\cos \theta_{_{1,P_{2}}}\cos\theta_{_{3,P_{2}}}}{\sin \theta_{_{1,P_{2}}}\sin\theta_{_{3,P_{2}}}})
\end{array}
\right.\\
&\xrightarrow{(\ref{angle2})}
\begin{array}{lcl}
\theta_{3,P_4}=\arccos(\dfrac{\cos\theta_{24}\sin\theta_{14}+\cos\theta_{1,P_4}\sin\theta_{24}\cos\theta_{14}}{\sqrt{(\cos\theta_{24}\sin\theta_{14}+\cos\theta_{1,P_4}\sin\theta_{24}\cos\theta_{14})^2+\sin^2\theta_{1,P_4}\sin^2\theta_{24}}})
\end{array}\\
&\xrightarrow{(\ref{angle1})}\!\!
\begin{array}{lcl}
\theta_{_{2,P_{4}}}\!\!=\arccos(\cos\theta_{24}\sin \theta_{_{1,P_{4}}}\sin \theta_{_{2,P_{4}}}\!\!+\cos \theta_{_{1,P_{4}}}\cos \theta_{_{2,P_{4}}})
\end{array}
\!\!\!\!\xrightarrow{}\!\!
\left\{
\begin{array}{lcl}
\!\theta_{_{1,P_{3}}}\!=180^{\circ}\!-\theta_{_{3,P_{1}}}\!-\theta_{_{2,P_{4}}}\\[1.5mm]
\!\theta_{_{3,P_{3}}}\!=180^{\circ}\!-\theta_{3,P_{2}}\!-\theta_{_{3,P_{4}}}
\end{array}
\right..
\end{split}
\end{equation*}
Then  
by (\ref{angle1}), the remaining three dihedral angles $ \theta_{13},\theta_{23},\theta_{34} $ can be derived similarly.

%

\end{document}